\documentclass[oneside]{article}
\usepackage[utf8]{inputenc}
\usepackage{amsmath}
\usepackage{amsfonts}
\usepackage{amssymb}
\usepackage{amsthm}

\usepackage[
  margin=3.5cm,
]{geometry}

\usepackage[nice]{nicefrac}

\usepackage{tikz} 
\usetikzlibrary{calc, decorations.markings}
\usepackage{subfig}

\usepackage{relsize} 

\newtheorem{theorem}{Theorem}[section]
\newtheorem{proposition}[theorem]{Proposition}
\newtheorem{remark}[theorem]{Remark}
\newtheorem{definition}[theorem]{Definition}
\newtheorem{lemma}[theorem]{Lemma}
\newtheorem{corollary}[theorem]{Corollary}
\newtheorem{example}[theorem]{Example}

\DeclareMathOperator*{\Res}{Res} 
\newcommand{\LegendreP}[3]{P_{#1}^{#2}{\hspace*{-2mm}\left( #3 \right)}} 
\newcommand{\LegendreQ}[3]{Q_{#1}^{#2}{\hspace*{-2mm}\left(#3\right)}}

\DeclareMathOperator*{\offdiag}{o-diag}
\newcommand{\diverg}{\mathrm{div}}
\DeclareMathOperator{\arcosh}{arcosh} 
\DeclareMathOperator{\arsinh}{arsinh}

\begin{document}

\title{
New generalized Mehler-Fock transformations and applications to the resolvent equation
\footnote{The results of the present work are part of my dissertation \cite{ErenDiss}.}
}

\author{Eren U\c{c}ar}

\newcommand{\Addresses}{{
  \bigskip
  \footnotesize

   \textsc{Institut f\"{u}r Mathematik, Humboldt-Universit\"{a}t zu Berlin, 10099 Berlin, Germany}\par\nopagebreak
  \textit{E-mail address:} \texttt{ucar@math.hu-berlin.de}
  
  }}


\date{\today}

\maketitle

\begin{abstract}
We investigate and solve a special class of integrals involving associated Legendre functions, which can be regarded as generalized Mehler-Fock transformations. Some of the integrals appear naturally when dealing with the heat or resolvent equation of a wedge in the hyperbolic plane. As an application, we derive explicit formulas for the Green's function and the heat kernel of a wedge. 
\end{abstract}

\section{Introduction}

The concept of generalized Mehler-Fock transformation has a long history, and it has important applications to the realm of partial differential equations. It can be used, for example, to solve various boundary value problems of mathematical physics involving wedge or conically shaped boundaries (see \cite{Oberhettinger}, \cite{Rosenthal} and references therein). Moreover, it appears in the spectral decomposition of the Laplacian in the hyperbolic plane for geodesic polar coordinates (see \cite{Terras}). The kernel of the generalized Mehler-Fock transformation is a so-called associated Legendre function, and therefore these functions are important for our purposes.

In order to fix our notation, let us introduce the associated Legendre functions of the first and second kinds and summarize some of their properties. Note that our discussion of the associated Legendre functions will not be complete, and that extensive treatises can be found in the literature on associated Legendre functions (see for example \cite[Chapter $8.7-8.8$]{Gradshteyn}, \cite{Hobson}, \cite{Virchenko}, \cite[ Chapter III]{Erdelyi}, \cite[Chapter $14$]{NIST}, \cite[Chapter $5$]{Olver}, and \cite[Chapter $8$]{Temme}).

\begin{definition}
\label{definition:LegendreDGL}
The \emph{associated Legendre equation} is an ordinary differential equation, defined as
\begin{align}
\label{equation:LegendreDGL}
\left( 1-z^2 \right)\frac{d^2 u}{dz^2} - 2z \frac{d u}{dz} + \left( \nu \left( \nu+1 \right) - \frac{\mu}{1-z^2} \right) u = 0,
\end{align}
with parameters $\nu, \mu \in \mathbb{C}$ and variable $z\in \mathbb{C}\backslash\lbrace{ -1, +1 \rbrace}$. Any solution of \eqref{equation:LegendreDGL} is called an \emph{associated Legendre function}.
\end{definition}

We are interested in two special solutions, namely, the associated Legendre functions of the first and second kinds. These can be defined in several equivalent ways, where we choose to define them via the hypergeometric function.

\begin{definition}
\label{definition:Pochhammer}
For any $z\in\mathbb{C}, k\in\mathbb{N}_0$, the \emph{Pochhammer symbol} is defined as
\begin{align*}
\left( z \right)_{k}:=\begin{cases}
1 , & \text{if $k=0$,}\\
z(z+1)(z+2)\cdots (z+k-1), & \text{if $k\geq 1$.}
\end{cases}
\end{align*}
It is related to the gamma function $\Gamma(z)$ through the equation
\begin{align*}
(z)_k = \frac{\Gamma(z+k)}{\Gamma(z)},
\end{align*}
which is easily established using the identity $\Gamma(z+1)=z\cdot \Gamma(z)$.
\end{definition}

For powers, we always use the standard branch: $z^{\mu}=e^{\mu\log(z)}$ with $\log(z):=\log(\vert z \vert) + i\arg(z)$ for $z\in\mathbb{C}\backslash (-\infty,0],$ with $\arg(z)\in(-\pi,\pi)$. For the following facts we refer to \cite[Section 9 of Chapter 5]{Olver}. 

\begin{definition} 
\label{definition:Hypergeom}
Let $a,b,c \in\mathbb{C}$ such that $c\notin \lbrace 0,-1,-2,... \rbrace$, and let $z\in\mathbb{C}\backslash [1,\infty)$. The \emph{hypergeometric function} $F(a,b;c;z)$ is defined through
\begin{align*}
F(a,b;c;z):=\sum\limits_{k=0}^{\infty}\frac{(a)_k  (b)_k}{(c)_k \text{ } k!} z^k,\quad \text{for }\vert z \vert<1,
\end{align*}
and via analytic continuation for $\vert z \vert \geq 1$ with $z\notin [ 1,\infty)$. 

The function
\begin{align*}
\frac{1}{\Gamma(c)}F(a,b;c;z)
\end{align*}
is an entire function in all three parameters $a,b,c$ and is analytic in $z \in \mathbb{C}\backslash [1,\infty)$. In particular, we think of the above function as analytically extended in $c\in\lbrace 0,-1,-2,... \rbrace$, even though $F(a,b;c;z)$ is formally not defined for those values of $c$.
%
\end{definition}

\begin{definition} Let $z\in\mathbb{C}\backslash (-\infty, 1].$
\label{definition:LegendreFirstSecond}
The \emph{associated Legendre function of the first kind} is defined by
\begin{align}
\label{equation:LegendreFirst}
P_{\nu}^{\mu}(z) := \left( \frac{z+1}{z-1} \right)^{\frac{\mu}{2}}\frac{1}{\Gamma(1-\mu)}F\left( -\nu, \nu + 1; 1-\mu;\frac{1-z}{2} \right),
\end{align}
where $\nu,\mu\in\mathbb{C}$.
The \emph{associated Legendre function of the second kind} is defined by
\begin{align}
\label{equation:LegendreSecond}
Q_{\nu}^{\mu}(z) := \frac{ \Gamma\left( \nu+\mu+1 \right)\sqrt{\pi} \left( z^2-1 \right)^{\frac{\mu}{2}}}{e^{-\mu\pi i} 2^{\nu +1} z^{\nu+\mu+1}\cdot \Gamma\left(\nu + \frac{3}{2}\right)}F\left(\frac{\nu+\mu+2}{2}, \frac{\nu+\mu+1}{2}; \nu+\frac{3}{2};\frac{1}{z^2} \right),
\end{align}
where $\mu,\nu\in\mathbb{C}$ such that $\mu+\nu\notin\lbrace -1,-2,-3,,... \rbrace$.

For $\mu=0$ we use $P_{\nu}(z) := P_{\nu}^{0}(z)$ and $Q_{\nu}(z) := Q_{\nu}^{0}(z)$. These functions are also known as \emph{Legendre functions of the first and second kinds}, respectively.

\end{definition} 
The associated Legendre functions of the first and second kinds are linearly independent solutions of \eqref{equation:LegendreDGL}, and are analytic in the parameters $\mu$ and $\nu$, whenever defined, as well as in the variable $z.$ Definition \ref{definition:LegendreFirstSecond} was first introduced by Hobson, it can be found in his classical treatise \cite{Hobson} and is commonly used in the literature (e.g. in \cite{Gradshteyn}, \cite{Erdelyi}, \cite{Temme}). Some authors, however, define the associated Legendre functions differently, in particular, the function of the second kind. For example, E. Barnes and G. Watson define $Q_{\nu}^{\mu}(z)$ differently in \cite{BarnesLegendre} and \cite{WatsonAsymp}, respectively. Also \cite{Olver}, \cite{NIST} do not prefer to work with $Q_{\nu}^{\mu}(z)$ as defined in \eqref{equation:LegendreSecond}.

The paper is organized as follows. In the next section, we will investigate some integrals involving associated Legendre functions. Our main result is Theorem \ref{theorem:TheoremZentralIntegral}, which can be used to solve several interesting integrals (see Example \ref{example:LegendreIntegral}, Corollary \ref{corollary:Corollary1Schritt2}, and Corollary \ref{corollary:IntegralProduktLegendreFuerGreensFunction}).
These integrals can be rewritten as so-called \emph{generalized Mehler-Fock transformations}. From this point of view we basically compute new such transformations. Despite the fact that the generalized Mehler-Fock transformation has been known for some time (see \cite{Lowndes}, \cite{Sneddon}, \cite{Rosenthal}), apparantly there hardly exist calculated examples compared with other integral transformations (see \cite{Oberhettinger} for a list of known generalized Mehler-Fock transformations).

In section 3, we will use those integrals to deduce new explicit formulas for the Green's function and the heat kernel for a wedge in the hyperbolic plane.
\section{Generalized Mehler-Fock transformations}
\label{section:Legendre}

There are plenty of remarkable relations and identities between the asociated Legendre functions of the first and second kinds (see \cite{Gradshteyn}, \cite{Erdelyi}). For the convenience of the reader we summarize those which will be important for us.

\begin{lemma}
\label{lemma:MagischeFormeln}
The associated Legendre functions of the first and second kinds satisfy the following identities for all allowed values of $\nu, \mu, z,\omega$:
\begin{align}
P_{\nu}^{\mu}(z) &=P_{-\nu-1}^{\mu}(z), \label{equation:ReflectionLegendreP} \\
Q_{\nu}^{\mu}(z) &= e^{2i\mu\pi}\frac{\Gamma\left( \nu+\mu+1 \right)}{\Gamma\left( \nu-\mu+1 \right)}Q_{\nu}^{-\mu}(z), \label{equation:MagischeFormel1}\\
P_{\nu}^{-\mu}(z) &= \frac{\Gamma\left( \nu-\mu+1 \right)}{\Gamma\left( \nu+\mu+1 \right)}\left(P_{\nu}^{\mu}(z) - \frac{2}{\pi}e^{-i\mu\pi}\sin\left( \mu\pi \right) Q_{\nu}^{\mu}(z) \right),\label{equation:MagischeFormel2} \\
Q_{\nu}^{-\mu}(z)Q_{\nu}^{\mu}(\omega) &= Q_{\nu}^{-\mu}(\omega)Q_{\nu}^{\mu}(z), \label{equation:SymmetrieProduktLegendreQ}\\
-2\frac{\sin\left( \mu\pi \right)}{\pi}Q_{\nu}^{-\mu}(z)Q_{\nu}^{\mu}(\omega) &= e^{-i\mu\pi}P_{\nu}^{-\mu}(\omega)Q_{\nu}^{\mu}(z) - e^{i\mu\pi}P_{\nu}^{\mu}(\omega)Q_{\nu}^{-\mu}(z). \label{equation:MagischeFormel3}
\end{align}
\end{lemma}

\begin{proof}
Equation \eqref{equation:ReflectionLegendreP}, \eqref{equation:MagischeFormel1}, and \eqref{equation:MagischeFormel2} are stated in \cite[formulas 8.731 5, 8.736 4, 8.736 1]{Gradshteyn}. Equation \eqref{equation:SymmetrieProduktLegendreQ} follows from \eqref{equation:MagischeFormel1}, when applied twice.

Lastly, \eqref{equation:MagischeFormel3} follows easily from \eqref{equation:MagischeFormel1} and \eqref{equation:MagischeFormel2}. 
\end{proof}

Another remarkable connection between the associated Legendre functions of the first and second kinds is given by \emph{Whipple's formula} (see \cite[formulas (13) and (14) on p. 141]{Erdelyi}).
\begin{lemma}
\label{lemma:Whipple}
For all $z\in\mathbb{C}\backslash (-\infty,1]$ with $\Re(z)>0$ the following relations hold:
\begin{align}
e^{-i\mu\pi}Q_{\nu}^{\mu}(z)&=\sqrt{\frac{\pi}{2}}\Gamma\left( \nu+\mu+1 \right)\frac{1}{\left( z^2-1 \right)^{\frac{1}{4}}}P_{-\mu-\frac{1}{2}}^{-\nu-\frac{1}{2}}\left( \frac{z}{\left( z^2 -1 \right)^{\frac{1}{2}}} \right),\label{equation:Whipple1} \\
P_{\nu}^{\mu}(z)&=ie^{i\nu\pi}\sqrt{\frac{2}{\pi}}\cdot \frac{1}{\Gamma\left(-\nu-\mu\right)}\cdot \frac{1}{\left( z^2-1 \right)^{\frac{1}{4}}}Q_{-\mu-\frac{1}{2}}^{-\nu-\frac{1}{2}}\left( \frac{z}{\left( z^2 -1 \right)^{\frac{1}{2}}} \right). \label{equation:Whipple2}
\end{align}
\end{lemma}

In order to investigate the integrals later in this section, it is helpful to know the following asymptotic behaviour of the associated Legendre functions of the first and second kinds.

\begin{lemma}
\label{lemma:LegendreAsymp}
Let $a>0$.
\begin{itemize}
\item[$(i)$] For any $\mu\in\mathbb{C}$:
\begin{align} P_{\nu}^{\mu}\left( \cosh(a) \right) = &\frac{\Gamma\left( \nu+1 \right)}{\Gamma\left( \nu-\mu+1 \right)} \cdot \frac{1}{\sqrt{2\pi (\nu+1)\sinh(a)}} \nonumber\\
& \cdot \left( e^{\left( \nu+\frac{1}{2} \right)a} + e^{-\pi i \left( \mu-\frac{1}{2} \right)-\left( \nu+\frac{1}{2} \right)a} \right)\left( 1+O\left( \frac{1}{\vert \nu \vert} \right) \right) \label{equation:LegendreAsymp1}
\end{align}
as $\vert \nu \vert \rightarrow\infty$ with $\Re\left(\nu \right)>-1$.
\item[$(ii)$] For any $\mu\in\mathbb{C}$:
\begin{align}
P_{-\frac{1}{2}+i\rho}^{\mu}\left(\cosh(a)\right) = \rho^{\mu-\frac{1}{2}}\sqrt{\frac{2}{\pi\sinh(a)}}\cos\left( a\rho+\frac{\pi}{4}\left( 2\mu-1 \right) \right)\left( 1+O\left(\frac{1}{\rho}\right)\right) \label{equation:LegendreAsymp2}
\end{align}
as $\rho\rightarrow\infty$ with $\rho\in\mathbb{R}$.
\item[$(iii)$] For any $\mu\in\mathbb{C}$ and $\delta\in (0,\pi)$:
\begin{align}
Q_{\nu}^{\mu}\left(\cosh(a)\right) = &\sqrt{\frac{\pi}{2\sinh(a)}}\nu^{\mu-\frac{1}{2}}e^{i\mu\pi-a\left( \nu+\frac{1}{2} \right)}\left(1+O\left(\frac{1}{\vert \nu \vert}\right)\right) \label{equation:LegendreAsymp3}
\end{align}
as $ \vert \nu \vert \rightarrow\infty$ with $\vert \arg(\nu) \vert<\pi-\delta $.
\end{itemize}
\end{lemma}

\begin{proof}
All three asymptotic estimates are stated at the beginning of $\S 8$ in \cite{Virchenko}.

A proof of \eqref{equation:LegendreAsymp1} can be found in \cite{Goetze}, which is also referred to in \cite{Virchenko}. At this point we want to remark that both sources \cite{Virchenko} and \cite{Goetze} state more general results. They deal with asymptotic estimates for so-called \emph{generalized associated Legendre functions} of \emph{first} and \emph{second kinds}, denoted by $P_{\nu}^{n, m}(z)$ and $Q_{\nu}^{n, m}(z)$ respectively. For $n=m=\mu$ these functions reduce to $P_{\nu}^{\mu}(z)$, respectively $Q_{\nu}^{\mu}(z)$.

Equation \eqref{equation:LegendreAsymp2} can be deduced from \eqref{equation:LegendreAsymp1} (see \cite[Lemma 2.19]{ErenDiss} for a proof). Equation \eqref{equation:LegendreAsymp3} is proven in Lemma $2$ of \S 8 in \cite{Virchenko}.

\end{proof}

\begin{lemma}
\label{lemma:LegendreProdAsymp}
Let $\nu\in \mathbb{C}$ and $z,\omega\in (1,\infty)$ be given. Further, let $\tilde{z},\tilde{\omega}\in (0,\infty)$ be such that 
\begin{align}
\label{equation:TildeDefinitionen}
\cosh(\tilde{z})=\frac{z}{\left( z^2-1\right)^{\nicefrac{1}{2}}}, \quad \cosh(\tilde{\omega})=\frac{\omega}{\left( \omega^2-1\right)^{\nicefrac{1}{2}}}.
\end{align}
Then
\begin{itemize}
\item[$1)$]
\begin{align} 
e^{-i\pi\mu} P_{\nu}^{-\mu}\left( \omega \right)Q_{\nu}^{\mu}\left( z \right) = \frac{e^{- \tilde{\omega}\mu}}{2\mu } \left( e^{\mu\tilde{z}}+e^{i\pi(\nu+1)}e^{-\mu\tilde{z}} \right) \left( 1+O\left(\frac{1}{\vert \mu \vert}\right) \right) \label{equation:AsympLegendreProdKond1}
\end{align}
as $\vert \mu \vert\rightarrow\infty$ with $\Re(\mu)>-\frac{1}{2}$.
\item[$2)$] 
\begin{align}
\frac{\sin(i\pi\rho)}{\pi} Q_{\nu}^{-i\rho}\left( z \right)Q_{\nu}^{i\rho}\left( \omega \right) = &\frac{i}{ \rho} \cos\left(\tilde{z}\rho-\frac{\pi}{2}\left( \nu+1 \right)\right)\cdot \nonumber\\
&\cdot \cos\left(\tilde{\omega}\rho-\frac{\pi}{2}\left( \nu+1 \right)\right)\left( 1+O\left( \frac{1}{\rho} \right) \right) \label{equation:AsympLegendreProdKond2}
\end{align}
as $\rho\rightarrow\infty$ with $\rho\in \mathbb{R}$.
\end{itemize}
\end{lemma}

\begin{proof}
Since $\cosh^2(s)-\sinh^2(s)=1$ for all $s\in\mathbb{C}$, it follows that
\begin{align}
\label{equation:BeziehungTilde}
\frac{1}{\sqrt{\sinh(\tilde{z})}}=\left( z^2-1 \right)^{\nicefrac{1}{4}}, \quad \frac{1}{\sqrt{\sinh(\tilde{\omega})}}=\left( \omega^2-1 \right)^{\nicefrac{1}{4}}.
\end{align}
Let us consider situation $1)$ first.

When we use Whipple's formulas \eqref{equation:Whipple1}, \eqref{equation:Whipple2} and then \eqref{equation:ReflectionLegendreP}, we get
\begin{align*}
 e^{-i\pi \mu}&P_{\nu}^{-\mu}\left( \omega \right)  Q_{\nu}^{\mu}\left( z \right) \\
&=\frac{\Gamma\left( \mu+\nu+1 \right)}{\Gamma\left( \mu-\nu \right)}\frac{ie^{i\nu\pi}}{\left( z^2-1 \right)^{\nicefrac{1}{4}}\left( \omega^2-1 \right)^{\nicefrac{1}{4}}} P_{\mu-\frac{1}{2}}^{-\nu-\frac{1}{2}}\left( \cosh(\tilde{z}) \right) Q_{\mu-\frac{1}{2}}^{-\nu-\frac{1}{2}}\left( \cosh(\tilde{\omega}) \right).
\end{align*}
We can now apply the asymptotic estimates \eqref{equation:LegendreAsymp1}, \eqref{equation:LegendreAsymp3} to the right-hand side of the above equation. When we do so and then use the relation \eqref{equation:BeziehungTilde}, we obtain for $\vert \mu \vert\rightarrow\infty$ with $\Re(\mu)>-\frac{1}{2}$:
\begin{align}
e^{-i\pi \mu}&P_{\nu}^{-\mu}\left( \omega \right) Q_{\nu}^{\mu}\left( z \right)\nonumber \\
&=\frac{\Gamma\left( \mu+\frac{1}{2} \right)}{\Gamma\left( \mu-\nu\right)} \frac{e^{-\tilde{\omega}\mu} \mu^{-\nu-\frac{3}{2}}}{2} \left( e^{\mu\tilde{z}}+e^{i\pi(\nu+1)}e^{-\mu\tilde{z}} \right) \left( 1+O\left(\frac{1}{\vert \mu \vert}\right) \right). \label{equation:AsymptProd9}
\end{align}
Now use the following well-known asymptotic estimate for the quotient of two gamma functions (see \cite[formula (12)]{Tricomi} or \cite[formula (11) on p. 33]{Luke1}):
\begin{align}
\frac{\Gamma\left( z+\alpha \right)}{\Gamma\left( z+\beta \right)} = z^{\alpha-\beta}\left(1 + O\left( \frac{1}{\vert z\vert}\right) \right) \label{equation:GammaQuotAsymp}
\end{align}  
as $z\rightarrow\infty$ with $\vert arg(z) \vert\leq \pi-\epsilon$ for some $\epsilon>0$.

Applying the asymptotic estimate \eqref{equation:GammaQuotAsymp}, we get
\begin{align}
\frac{\Gamma\left( \mu+\frac{1}{2} \right)}{\Gamma\left( \mu-\nu\right)} = \mu^{\nu+\frac{1}{2}}\left( 1+O\left( \frac{1}{\vert \mu \vert} \right) \right) \label{equation:GammaQuotSpez}
\end{align}
as $\mu\rightarrow\infty$ with $\Re(\mu)>-\frac{1}{2}$. Lastly, when we combine \eqref{equation:GammaQuotSpez} and \eqref{equation:AsymptProd9}, we obtain
\begin{align*}
e^{-i\pi \mu}P_{\nu}^{-\mu}\left( \omega \right)Q_{\nu}^{\mu}\left( z \right) = \frac{e^{-\tilde{\omega}\mu}}{2 \mu } \left( e^{\mu\tilde{z}}+e^{i\pi(\nu+1)}e^{-\mu\tilde{z}} \right) \left( 1+O\left(\frac{1}{\vert \mu  \vert}\right) \right),
\end{align*}
as $\mu\rightarrow\infty$ with $\Re(\mu)>-\frac{1}{2}$. This shows $1)$. 

For $2)$, we first use Whipple's formula \eqref{equation:Whipple1} twice and \eqref{equation:ReflectionLegendreP} once, to obtain
\begin{align}
Q_{\nu}^{-i\rho}(z)Q_{\nu}^{i\rho}(\omega) =\, &\frac{\pi}{2}\Gamma\left( i\rho+\nu+1 \right)\Gamma\left( -i\rho+\nu+1 \right)\frac{1}{\left(z^2-1\right)^{\nicefrac{1}{4}}\left(\omega^2-1\right)^{\nicefrac{1}{4}}}\cdot  \nonumber \\
& \cdot P_{-\frac{1}{2}+i\rho}^{-\nu-\frac{1}{2}}\left( \cosh(\tilde{z})  \right)P_{-\frac{1}{2}+i\rho}^{-\nu-\frac{1}{2}}\left( \cosh(\tilde{\omega}) \right). \label{equation:ConditionsSchritt1.1}
\end{align}
It is known (see e.g. \cite{Lebedev} on p. 15) that for any $\varepsilon \in (0,\pi)$ and $\alpha\in\mathbb{C}$,
\begin{align}
\label{equation:AsymptotikGamma}
\Gamma\left( z+\alpha \right) = e^{\left( z+\alpha-\frac{1}{2} \right)\log(z)-z+\frac{1}{2}\log(2\pi)}\left(1+O\left(\frac{1}{\vert z \vert}\right)\right),
\end{align}
as $\vert z \vert\rightarrow\infty$ with $\vert \arg(z) \vert<\pi-\varepsilon$.

Hence we obtain that, for $\rho\rightarrow\infty$ with $\rho\in \mathbb{R}$,
\begin{align*}
\Gamma\left(\pm i\rho+\nu+1\right) &= e^{\left( \pm i\rho+\nu+\frac{1}{2} \right)\log(\pm i\rho)}e^{\mp i\rho}\cdot \sqrt{2\pi}\left(1+O\left(\frac{1}{ \rho }\right)\right),
\end{align*}
and therefore
\begin{align}
\label{equation:ProduktGamma}
\Gamma\left( i\rho+\nu+1\right) \Gamma\left( -i\rho+\nu+1\right) = \rho^{2\nu+1} e^{-\pi\rho}\cdot 2\pi \left(1+O\left(\frac{1}{ \rho }\right)\right).
\end{align}

On the other hand, when we use \eqref{equation:LegendreAsymp2} and \eqref{equation:BeziehungTilde}, we get
\begin{align}
&\frac{\pi}{2}\frac{1}{\left(z^2-1\right)^{\nicefrac{1}{4}}\left(\omega^2-1\right)^{\nicefrac{1}{4}}}\cdot 
P_{-\frac{1}{2}+i\rho}^{-\nu-\frac{1}{2}}\left( \cosh(\tilde{z}) \right)P_{-\frac{1}{2}+i\rho}^{-\nu-\frac{1}{2}}\left( \cosh(\tilde{\omega})\right) \nonumber \\
&= \rho^{-2\nu-2}\cos\left( \tilde{z}\rho-\frac{\pi}{2}\left( \nu+1 \right) \right)\cos\left( \tilde{\omega}\rho-\frac{\pi}{2}\left( \nu+1 \right) \right)\left( 1+O\left( \frac{1}{\rho} \right) \right) \label{equation:ConditionsBeweis2}
\end{align}
as $\rho\rightarrow\infty$ with $\rho \in \mathbb{R}$. From \eqref{equation:ConditionsBeweis2}, \eqref{equation:ProduktGamma} and \eqref{equation:ConditionsSchritt1.1} we obtain
\begin{align*}
Q_{\nu}^{-i\rho}(z)Q_{\nu}^{i\rho}(\omega) = \frac{2\pi}{\rho} e^{-\pi \rho} \cos\left( \tilde{z}\rho-\frac{\pi}{2}\left( \nu+1 \right) \right)\cdot  \cos\left( \tilde{\omega}\rho-\frac{\pi}{2}\left( \nu+1 \right) \right)\left( 1+O\left( \frac{1}{\rho} \right) \right)
\end{align*}
as $\rho\rightarrow\infty$ with $\rho\in \mathbb{R}$. Hence, as $\rho\rightarrow\infty$ in $\mathbb{R}$:
\begin{align*}
\frac{\sin(\pi i \rho)}{\pi} Q_{\nu}^{-i\rho}(z)Q_{\nu}^{i\rho}(\omega)   =\frac{i}{\rho} \cos\left( \tilde{z}\rho-\frac{\pi}{2}\left( \nu+1 \right) \right)\cos\left( \tilde{\omega}\rho-\frac{\pi}{2}\left( \nu+1 \right) \right)\left( 1+O\left( \frac{1}{\rho} \right) \right). 
\end{align*}
\end{proof}

In the last part of this section we want to discuss a class of integrals involving associated Legendre functions of the second kind. We start with a rather technical lemma.
\begin{lemma}
\label{lemma:Schritt1}
Let $\nu\in\mathbb{C}$ with $\Re(\nu)>-1$ and $z,\omega\in\mathbb{C}\backslash\left(-\infty, 1 \right]$. Suppose further that we are given the following situation:
\begin{itemize}
\item[$(i)$] Let $g$ be a meromorphic function on $\mathbb{C}$ and let $S\subset \mathbb{C}$ denote the set of all poles of $g$. Suppose that $g$ is complex differentiable at all points on the imaginary axis, possibly except for a simple pole at the origin, and that $g$ is an odd function, i.e. $g(-s)=-g(s)$ for all $s\in\mathbb{C}\backslash S$.
\item[$(ii)$] Fix some $\varepsilon\in (0,1)$ such that $0<2\varepsilon<  1+\Re(\nu) $ and $B_{2\varepsilon}(0)\cap S\backslash\lbrace{ 0\rbrace}=\emptyset$, where $B_{2\epsilon}(0)$ denotes the open disc of radius $2\varepsilon$ centered at the origin. 
\item[$(iii)$] Let $N:=\left(N_k \right)_{k\in\mathbb{N}}\subset \left[1,\infty \right)$ be an unbounded and monotonically increasing sequence.
\end{itemize}  

Then there exists a function $\varphi:\mathbb{N}\rightarrow \mathbb{C}$ with $\lim_{k\rightarrow\infty}\varphi(k)=0$ and such that for any $k\in\mathbb{N}$ the following holds:
If $C_{k}$ denotes any injective and piecewise differentiable curve from $-iN_k$ to $iN_k$ such that all other points on the curve lie in the domain $\lbrace z\in\mathbb{C}: \Re(z)>0 \rbrace\backslash S$ \emph{(}see Fig. $\ref{Skizze:LemmaSchritt1}$\emph{)}, then we have:
\begin{align}
\label{equation:Schritt1}
2\left( \sum\limits_{p\in S\left( C_{k} \right)} \emph{Res}\left( g;p \right) e^{-i\pi p} P_{\nu}^{-p}(\omega)Q_{\nu}^{p}(z) \right)& + \emph{Res}\left( g;0 \right) P_{\nu}( \omega )Q_{\nu}(z) \nonumber \\
=\frac{2}{\pi}\int\limits_{\frac{\varepsilon}{N_k}}^{N_k} \frac{\sin\left(\pi i\rho \right)}{\pi}g(i\rho)Q_{\nu}^{-i\rho}(z)  Q_{\nu}^{i\rho}(\omega ) d\rho &+ \frac{1}{\pi i}\int\limits_{C_{k}}g(s)e^{-i\pi s}P_{\nu}^{-s}( \omega )  Q_{\nu}^{s}(z) ds  \nonumber \\
&  + \varphi(k),
\end{align}
where $S\left( C_{k} \right)$ denotes the set of all poles of $g$ contained in the bounded domain which is enclosed by the imaginary axis and the curve $C_{k}$.
\end{lemma}

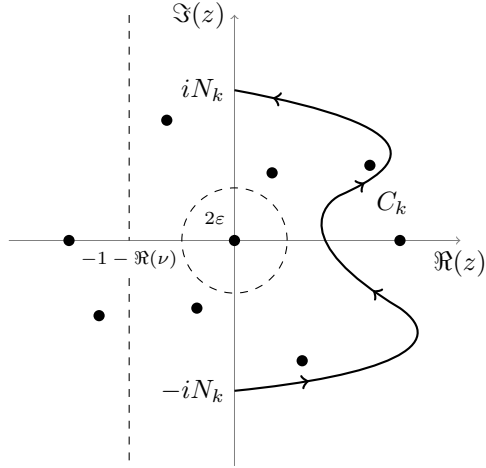
\begin{figure} [ht] 
 \centering
\begin{tikzpicture}[decoration={markings,							
mark=at position 1cm with {\arrow[line width=1pt]{>}},
mark=at position 3.5cm with {\arrow[line width=1pt]{>}},
mark=at position 5.5cm with {\arrow[line width=1pt]{>}},
mark=at position 7.85cm with {\arrow[line width=1pt]{>}},
mark=at position 9cm with {\arrow[line width=1pt]{>}},
mark=at position 10.4cm with {\arrow[line width=1pt]{>}},
mark=at position 12.3cm with {\arrow[line width=1pt]{>}}
}
]
\draw[help lines,->] (-3,0) -- (3,0) coordinate (xaxis);
\draw[help lines,->] (0,-3) -- (0,3) coordinate (yaxis);

\path[draw,line width=0.8pt,postaction=decorate] (0,-2) node[left] {$-iN_k$} ..controls (2,-1.8) and (2.9,-1.4) .. (2.2,-0.9) ..controls (0.7,0) and (1.2,0.5) ..(1.4,0.6)..controls (2.3, 1) and (2.7, 1.5)..(0,2) node[left] {$iN_k$} ;


\draw[dashed] (-1.4,3) -- (-1.4, 0); 
\draw[dashed] (-1.4,-0.5) -- (-1.4, -3);

\draw[dashed] (0,0) circle (0.7); 


\draw (0,0) node[circle,fill,inner sep=1.5pt]{};
\draw (0.5,0.9) node[circle,fill,inner sep=1.5pt]{};
\draw (-0.5,-0.9) node[circle,fill,inner sep=1.5pt]{};
\draw (1.8,1) node[circle,fill,inner sep=1.5pt]{};
\draw (-1.8,-1) node[circle,fill,inner sep=1.5pt]{};
\draw (0.9,-1.6) node[circle,fill,inner sep=1.5pt]{};
\draw (-0.9,1.6) node[circle,fill,inner sep=1.5pt]{};
\draw (2.2,0) node[circle,fill,inner sep=1.5pt]{};
\draw (-2.2,0) node[circle,fill,inner sep=1.5pt]{};

\node[below] at (xaxis) {$\Re(z)$};
\node[left] at (yaxis) {$\Im(z)$};
\node at (2.1,0.5) {$C_{k}$};
\node[left] at (0,0.3) {\scriptsize{$2\varepsilon$}};
\node[below] at (-1.4,0) {\scriptsize $-1-\Re(\nu)$};
\end{tikzpicture}
\caption{Situation of Lemma \ref{lemma:Schritt1}} \label{Skizze:LemmaSchritt1}
\end{figure}

\begin{proof}
Let $k\in\mathbb{N}$ and $\omega,z\in \mathbb{C}\backslash (-\infty,1]$ be arbitrary. First we close up the curve $C_k$ in two different ways and obtain two piecewise smooth curves $\gamma_{k}^1$ and $\gamma_{k}^2$, as shown in Fig. \ref{Skizze:Gamma12}.

\begin{figure} [ht] 
  \subfloat[The closed path $\gamma_{k}^1$.]{
\begin{tikzpicture}[decoration={markings,							
mark=at position 1cm with {\arrow[line width=1pt]{>}},
mark=at position 3.5cm with {\arrow[line width=1pt]{>}},
mark=at position 5.5cm with {\arrow[line width=1pt]{>}},
mark=at position 7.85cm with {\arrow[line width=1pt]{>}},
mark=at position 9cm with {\arrow[line width=1pt]{>}},
mark=at position 10.4cm with {\arrow[line width=1pt]{>}},
mark=at position 12.3cm with {\arrow[line width=1pt]{>}}
}
]
\draw[help lines,->] (-3,0) -- (3,0) coordinate (xaxis);
\draw[help lines,->] (0,-3) -- (0,3) coordinate (yaxis);

\path[draw,line width=0.8pt,postaction=decorate] (0,-2)..controls (2,-1.8) and (2.9,-1.4) .. (2.2,-0.9) ..controls (0.7,0) and (1.2,0.5) ..(1.4,0.6)..controls (2.3, 1) and (2.7, 1.5)..(0,2) node[left] {$iN_k$} -- (0,0.4) node[right] {$i\frac{\varepsilon}{N_k}$} arc (90:270:0.4) -- (0,- 0.4) -- (0,-2) node[left] {$-iN_k$};


\draw[dashed] (-1.4,3) -- (-1.4, 0); 
\draw[dashed] (-1.4,-0.5) -- (-1.4, -3);


\draw (0,0) node[circle,fill,inner sep=1.5pt]{};
\draw (0.5,0.9) node[circle,fill,inner sep=1.5pt]{};
\draw (-0.5,-0.9) node[circle,fill,inner sep=1.5pt]{};
\draw (1.8,1) node[circle,fill,inner sep=1.5pt]{};
\draw (-1.8,-1) node[circle,fill,inner sep=1.5pt]{};
\draw (0.9,-1.6) node[circle,fill,inner sep=1.5pt]{};
\draw (-0.9,1.6) node[circle,fill,inner sep=1.5pt]{};
\draw (2.2,0) node[circle,fill,inner sep=1.5pt]{};
\draw (-2.2,0) node[circle,fill,inner sep=1.5pt]{};

\node[below] at (xaxis) {$\Re(z)$};
\node[left] at (yaxis) {$\Im(z)$};
\node at (2.1,0.5) {$C_{k}$};
\node[below] at (-1.4,0) {\scriptsize $- 1-\Re(\nu) $};
\end{tikzpicture}
}
\quad
  \subfloat[The closed path $\gamma_{k}^2$.]{
\begin{tikzpicture}[decoration={markings,							
mark=at position 1cm with {\arrow[line width=1pt]{>}},
mark=at position 3.5cm with {\arrow[line width=1pt]{>}},
mark=at position 5.5cm with {\arrow[line width=1pt]{>}},
mark=at position 7.85cm with {\arrow[line width=1pt]{>}},
mark=at position 9cm with {\arrow[line width=1pt]{>}},
mark=at position 10.4cm with {\arrow[line width=1pt]{>}},
mark=at position 12.3cm with {\arrow[line width=1pt]{>}}
}
]
\draw[help lines,->] (-3,0) -- (3,0) coordinate (xaxis);
\draw[help lines,->] (0,-3) -- (0,3) coordinate (yaxis);

\path[draw,line width=0.8pt,postaction=decorate] (0,-2)..controls (2,-1.8) and (2.9,-1.4) .. (2.2,-0.9) ..controls (0.7,0) and (1.2,0.5) ..(1.4,0.6)..controls (2.3, 1) and (2.7, 1.5)..(0,2) node[left] {$iN_k$} -- (0,0.4) node[left] {$i\frac{\varepsilon}{N_k}$} arc (90:-90:0.4) -- (0,- 0.4) -- (0,-2) node[left] {$-iN_k$};


\draw[dashed] (-1.4,3) -- (-1.4, 0); 
\draw[dashed] (-1.4,-0.5) -- (-1.4, -3);


\draw (0,0) node[circle,fill,inner sep=1.5pt]{};
\draw (0.5,0.9) node[circle,fill,inner sep=1.5pt]{};
\draw (-0.5,-0.9) node[circle,fill,inner sep=1.5pt]{};
\draw (1.8,1) node[circle,fill,inner sep=1.5pt]{};
\draw (-1.8,-1) node[circle,fill,inner sep=1.5pt]{};
\draw (0.9,-1.6) node[circle,fill,inner sep=1.5pt]{};
\draw (-0.9,1.6) node[circle,fill,inner sep=1.5pt]{};
\draw (2.2,0) node[circle,fill,inner sep=1.5pt]{};
\draw (-2.2,0) node[circle,fill,inner sep=1.5pt]{};

\node[below] at (xaxis) {$\Re(z)$};
\node[left] at (yaxis) {$\Im(z)$};
\node at (2.1,0.5) {$C_{k}$};
\node[below] at (-1.4,0) {\scriptsize $- 1-\Re(\nu)$};
\end{tikzpicture}

}
\caption{\label{Skizze:Gamma12}}
\end{figure}

By definition, both curves $\gamma_{k}^1$ and $\gamma_{k}^2$ can be written as a composition of the following curves:
\begin{align*}
\gamma_{k}^1 = C_k + \alpha_{k} + \beta_{k}^1 + \delta_k, \\
\gamma_{k}^2 = C_k + \alpha_{k} + \beta_{k}^2 + \delta_k,
\end{align*}
where
\begin{align*}
\alpha_k: \left[ -N_k,-\frac{\varepsilon}{N_k} \right]\rightarrow \mathbb{C},&\quad \rho\mapsto -\rho i; \\
\delta_k: \left[ \frac{\varepsilon}{N_k}, N_k \right]\rightarrow \mathbb{C},&\quad \rho\mapsto -\rho i; \\
\beta_{k}^1 : \left[ 0,\pi \right]\rightarrow \mathbb{C},&\quad \theta\mapsto \frac{\varepsilon}{N_k}ie^{i\theta}; \\
\beta_{k}^2 : \left[ 0,\pi \right]\rightarrow \mathbb{C}, &\quad \theta\mapsto \frac{\varepsilon}{N_k}ie^{-i\theta}.
\end{align*}
We observe that the function $u\mapsto P_{\nu}^{-u}(\omega )Q_{\nu}^{u} ( z ) $ is holomorphic on the half plane $\lbrace u\in\mathbb{C} : \Re(u)>-1-\Re\left( \nu \right) \rbrace$. Hence, we can rewrite the left-hand side of $\eqref{equation:Schritt1}$ by using first the residue theorem and then the above decomposition of $\gamma_{k}^1, \gamma_{k}^2$ as follows:
\begin{align}
\label{equation:Schritt1Bew1}
&\qquad 2\left( \sum\limits_{p \in S\left( C_k \right)} \Res\left( g;p \right) e^{-i\pi p} P_{\nu}^{-p}\left( \omega \right)Q_{\nu}^{p}\left( z \right) \right) + \Res\left( g;0 \right) P_{\nu}\left( \omega \right)Q_{\nu}\left( z \right) \nonumber \\
&=\frac{1}{2\pi i}\oint\limits_{\gamma_{k}^1} g(s) e^{-i\pi s} P_{\nu}^{-s}\left( \omega \right)Q_{\nu}^{s}\left( z \right) ds + \frac{1}{2\pi i}\oint\limits_{\gamma_{k}^2} g(s) e^{-i\pi s} P_{\nu}^{-s}\left( \omega \right)Q_{\nu}^{s}\left( z \right)  ds \nonumber \\
&= \frac{1}{\pi i}\int\limits_{\alpha_k} g(s) e^{-i\pi s} P_{\nu}^{-s}\left( \omega \right)Q_{\nu}^{s}\left( z \right) ds + \frac{1}{\pi i}\int\limits_{\delta_k} g(s) e^{-i\pi s} P_{\nu}^{-s}\left( \omega \right)Q_{\nu}^{s}\left( z \right)  ds \nonumber \\
&\quad +\frac{1}{\pi i}\int\limits_{C_k}g(s)e^{-i\pi s}P_{\nu}^{-s}\left( \omega \right)  Q_{\nu}^{s}\left( z \right) ds + \varphi(k),
\end{align}
with
\begin{align}
\label{equation:DefFehlerterm}
\varphi(k):= \frac{1}{2\pi i}\int\limits_{\beta_{k}^1} g(s) e^{-i\pi s} P_{\nu}^{-s}\left( \omega \right)Q_{\nu}^{s}\left( z \right) ds + \frac{1}{2\pi i}\int\limits_{\beta_{k}^2} g(s) e^{-i\pi s} P_{\nu}^{-s}\left( \omega \right)Q_{\nu}^{s}\left( z \right)  ds.
\end{align}
We observe that we can sum up the first two summands of \eqref{equation:Schritt1Bew1} as follows:
\begin{align*}
&\frac{1}{\pi i}\int\limits_{\alpha_k} g(s) e^{-i\pi s} P_{\nu}^{-s}\left( \omega \right)Q_{\nu}^{s}\left( z \right) ds + \frac{1}{\pi i}\int\limits_{\delta_k} g(s) e^{-i\pi s} P_{\nu}^{-s}\left( \omega \right)Q_{\nu}^{s}\left( z \right)  ds \\
= & -\frac{1}{\pi}\int\limits_{\frac{\varepsilon}{N_k}}^{N_k} g\left( i\rho \right)\left( e^{-i\pi\left(i\rho \right)}P_{\nu}^{-i\rho}\left( \omega \right)Q_{\nu}^{i\rho}\left( z \right) - e^{i\pi\left( i\rho \right)}P_{\nu}^{i\rho}\left( \omega \right)Q_{\nu}^{-i\rho}\left( z \right) \right)d\rho,
\end{align*}
where we used that the function $g$ is odd. Now we can use \eqref{equation:MagischeFormel3} to simplify the  expression in paranthesis under the integral sign and obtain, in total:
\begin{align}
\label{equation:Schritt1Bew2}
& \frac{1}{\pi i}\int\limits_{\alpha_k} g(s) e^{-i\pi s} P_{\nu}^{-s}\left( \omega \right)Q_{\nu}^{s}\left( z \right) ds + \frac{1}{\pi i}\int\limits_{\delta_k} g(s) e^{-i\pi s} P_{\nu}^{-s}\left( \omega \right)Q_{\nu}^{s}\left( z \right)  ds \nonumber \\
=\, &\frac{2}{\pi}\int\limits_{\frac{\varepsilon}{N_k}}^{N_k} \frac{\sin\left( i\rho \pi \right)}{\pi}g\left( i\rho \right)Q_{\nu}^{-i\rho} \left( z \right) Q_{\nu}^{i\rho}\left( \omega \right) d\rho.
\end{align}
Thus the claimed equation \eqref{equation:Schritt1} follows from \eqref{equation:Schritt1Bew1} together with \eqref{equation:Schritt1Bew2}.

It remains to show that $\lim_{k\rightarrow\infty}\varphi(k)=0$, where $\varphi(k)$ is defined as in \eqref{equation:DefFehlerterm}. When we use the parametrisation of $\beta_{k}^{1}$ and $\beta_{k}^2$ given above, we obtain
\begin{align*}
&\varphi\left(k \right) = \, \frac{1}{2\pi}\int\limits_{0}^{\pi} g\left( i\frac{\varepsilon}{N_k}e^{i\theta} \right) e^{-i\pi \left(i\frac{\varepsilon}{N_k}e^{i\theta}\right)} \LegendreP{\nu}{-i\frac{\varepsilon}{N_k}e^{i\theta}}{\omega} \cdot \LegendreQ{\nu}{i\frac{\varepsilon}{N_k}e^{i\theta}}{z}\cdot \left( i\frac{\varepsilon }{N_k}e^{i\theta} \right) d\theta  \\
& -\frac{1}{2\pi} \int\limits_{0}^{\pi}  g\left( i\frac{\varepsilon}{N_k}e^{-i\theta} \right) e^{-i\pi \left(i\frac{\varepsilon}{N_k}e^{-i\theta}\right)} \LegendreP{\nu}{-i\frac{\varepsilon}{N_k}e^{-i\theta}}{\omega} \cdot \LegendreQ{\nu}{i\frac{\varepsilon}{N_k}e^{-i\theta}}{z}\cdot \left( i\frac{\varepsilon }{N_k}e^{-i\theta} \right)  d\theta.
\end{align*}
By assumption, the meromorphic function $g$ has either a simple pole at the origin or is complex differentiable there. Hence the limit of the function $s g(s)$ exists as $s\rightarrow 0$ and is given by $\lim_{s\rightarrow 0} sg(s) = \Res\left( g;0 \right)\in\mathbb{C}$. Thus the integrand $s\mapsto sg(s)e^{-i\pi s}P_{\nu}^{-s}(\omega)Q_{\nu}^{ s}(z)$ is continuous in $B_{\varepsilon}(0)$, in particular at the origin $s=0$, and is bounded on $B_{\varepsilon}(0)$. By Lebesgue's dominated convergence theorem and since $\lim_{k\rightarrow\infty} N_k=\infty$, we have
\begin{align*}
\lim\limits_{k\rightarrow \infty}\varphi(k) &= \frac{1}{2\pi}\int\limits_{0}^{\pi} \Res(g;0)P_{\nu}(\omega)Q_{\nu}(z) d\theta - \frac{1}{2\pi}\int\limits_{0}^{\pi} \Res(g;0)P_{\nu}(\omega)Q_{\nu}(z) d\theta  \\
&= 0.
\end{align*} 
\end{proof}

The following corollary follows immediately from Lemma \ref{lemma:Schritt1}.

\begin{corollary}
\label{corollary:Corollary1Schritt1}
Let $\nu\in\mathbb{C}$ with $\Re(\nu)>-1$ and $z,\omega\in\mathbb{C}\backslash\left( -\infty, 1 \right]$. Further, let $g$ be a meromorphic function with the same properties as in Lemma \emph{\ref{lemma:Schritt1}} and suppose:
\begin{enumerate}
\item[$1)$] There exists a sequence of curves $\left( C_k\right)_{k\in\mathbb{N}}$, each of them as in Lemma $\emph{\ref{lemma:Schritt1}}$, and such that
\begin{align}
\lim\limits_{k\rightarrow \infty}\int\limits_{C_k} g(s)e^{-i\pi s}P_{\nu}^{-s}(\omega)Q_{\nu}^{s}(z) ds = 0. \label{equation:Schritt1.1}
\end{align}
\item[$2)$] The following integral converges in $\mathbb{C}$:
\end{enumerate}
\begin{align}
\int\limits_{0}^{\infty} \frac{\sin(\pi i \rho)}{\pi} g\left(i\rho\right)Q_{\nu}^{-i\rho}(z)Q_{\nu}^{i\rho}(\omega) d\rho. \label{equation:Schritt1.2}
\end{align}
Then it follows that
\begin{align}
\label{equation:Corollary1Schritt1}
&\frac{2}{\pi}\int\limits_{0}^{\infty} \frac{\sin(\pi i \rho)}{\pi} g\left(i\rho\right)Q_{\nu}^{-i\rho}(z)Q_{\nu}^{i\rho}(\omega) d\rho \nonumber \\
= & \lim\limits_{k\rightarrow\infty} \left( 2 \sum\limits_{p\in S\left( C_k \right)} \emph{Res}\left( g;p \right) e^{-i\pi p} P_{\nu}^{-p}(\omega)Q_{\nu}^{p}(z) \right)  + \emph{Res}\left( g;0 \right) P_{\nu}( \omega )Q_{\nu}(z).
\end{align} 
\end{corollary}

It is an interesting problem to investigate for which functions $g$ as in Lemma \ref{lemma:Schritt1} the conditions $1)$ and $2)$ of Corollary \ref{corollary:Corollary1Schritt1} are satisfied. We will not pursue this problem in full generality. Instead, we will restrict the parameters $\omega, z$ to real values. 

The following lemma establishes a class of functions $g$ for which condition $2)$ of Corollary \ref{corollary:Corollary1Schritt1} is satisfied.

\begin{lemma}
\label{lemma:ConditionsSchritt1}
Let $\nu\in \mathbb{C}$ with $\Re(\nu)\notin \{ -1, -2, -3,... \}$ \emph{(}e.g. $\Re(\nu)>-1$\emph{)} and let $z,\omega\in \left(1,\infty \right)$. Let $g$ be as in Lemma \emph{\ref{lemma:Schritt1}}.
\begin{itemize}
\item[$(i)$] Suppose $g(i\rho)=O\left(\frac{1}{\rho}\right)$, as $\rho\rightarrow\infty$ in $\mathbb{R}$. Then the integral \eqref{equation:Schritt1.2} is absolutely convergent, and hence convergent.

\item[$(ii)$] Suppose there exists some constant $C\in\mathbb{C}$ such that $g(i\rho)=C\left(1+O\left(\frac{1}{\rho}\right)\right)$, as $\rho\rightarrow\infty$ in $\mathbb{R}$. Then the integral \eqref{equation:Schritt1.2} is convergent for $\omega\neq z$, but divergent for $\omega = z$. Further, it is never absolutely convergent.
\end{itemize}
\end{lemma}

\begin{proof}
Let us abbreviate the integrand in \eqref{equation:Schritt1.2} by
\begin{align*}
f(\rho):=\frac{\sin(\pi i \rho)}{\pi} g\left(i\rho\right)Q_{\nu}^{-i\rho}(z)Q_{\nu}^{i\rho}(\omega),\quad \rho\in (0,\infty).
\end{align*}
The function $f(\rho)$ is continuous on $(0,\infty)$ and can be extended continuously to $[0,\infty)$ by the assumptions on $g$.  Therefore, we only need to investigate how the integrand behaves asymptotically as $\rho\rightarrow\infty$ in $\mathbb{R}$. We will see that both $(i)$ and $(ii)$ follow from \eqref{equation:AsympLegendreProdKond2} by elementary calculations.

First consider case $(i)$.  Note that the function 
\begin{align*}
\mathbb{R}\ni \rho \mapsto \cos\left( \tilde{z}\rho-\frac{\pi}{2}\left( \nu+1 \right) \right)\cos\left( \tilde{\omega}\rho-\frac{\pi}{2}\left( \nu+1 \right) \right)\in\mathbb{C}
\end{align*}
is bounded, where $\tilde{z}, \tilde{\omega}\in (0,\infty)$ are defined by \eqref{equation:TildeDefinitionen}. Thus it follows from \eqref{equation:AsympLegendreProdKond2} that
\begin{align*}
\left\vert f(\rho) \right\vert = O\left( \frac{1}{\rho^2} \right)
\end{align*}
as $\rho\rightarrow\infty$ in $\mathbb{R}$. Hence the integral \eqref{equation:Schritt1.2} is absolutely convergent.

Now consider case $(ii)$. In this case it follows from \eqref{equation:AsympLegendreProdKond2} that
\begin{align}
f(\rho) = \underbrace{C\frac{i}{\rho} \cos\left( \tilde{z}\rho-\frac{\pi}{2}\left( \nu+1 \right) \right)\cos\left( \tilde{\omega}\rho-\frac{\pi}{2}\left( \nu+1 \right) \right)}_{=:\psi(\rho)}\left( 1+O\left( \frac{1}{\rho} \right) \right) \label{equation:ConditionBeweis4}
\end{align}
as $\rho\rightarrow\infty$ in $\mathbb{R}$. Note that $z\neq\omega$ if and only if $\tilde{z}\neq\tilde{\omega}$, which follows directly from \eqref{equation:TildeDefinitionen}.

We claim that if $z\neq \omega$ and $K>0$, then the function $\psi$ is integrable over the interval $[K,\infty)$, where ``integrable'' is meant here in the sense that the integral converges. Furthermore, $\psi$ is not integrable over $[K,\infty)$ if $z=\omega$, and for any values of $z,\omega\in (1,\infty)$ it is not absolutely integrable over $[K,\infty)$. 

Before we prove this claim, let us discuss why statement $(ii)$ follows from this behaviour of $\psi $. It follows from \eqref{equation:ConditionBeweis4} that there exists a function $\phi:[\Lambda,\infty)\rightarrow\mathbb{R}$ with $\Lambda>0$ such that for some constant $D>0$ we have $\vert \phi(\rho) \vert\leq \frac{D}{\rho}$  and $f(\rho)=\psi(\rho) + \psi(\rho)\cdot \phi(\rho)$ for all $\rho\in [\Lambda,\infty)$. Hence there exists some constant $E>0$ such that $\vert \psi(\rho)\cdot \phi(\rho)\vert \leq \frac{E}{\rho^2}$ for all $\rho\in [\Lambda,\infty)$ and the function $\psi\cdot \phi$ is continuous because $f,\psi$ are continuous and $\psi \cdot \phi  = f  - \psi $. Thus $\psi \cdot \phi $ must be (absolutely) integrable over $[\Lambda,\infty)$. Therefore it follows from the equation $f=\psi+\psi\cdot \phi$ that $f$ is integrable over $[\Lambda,\infty)$ if and only if $\psi$ is integrable over $[\Lambda,\infty)$. This shows that if the above claim is true, then the function $f$ is integrable if $z\neq \omega$, but it is not integrable if $z=\omega$.

Further, if the above claim is true then the integral of $f$ is never absolutely convergent. This is easily seen, since it follows from \eqref{equation:ConditionBeweis4} that
\begin{align*}
\left\vert f(\rho) \right\vert = \left\vert \psi(\rho)\right\vert \left( 1+o(1) \right)
\end{align*}
as $\rho\rightarrow\infty$ in $\mathbb{R}$.
Hence for any $K>0$ the function $\vert f \vert$ is integrable over $[K,\infty)$ if and only if $\left\vert \psi\right\vert$ is integrable over $[K,\infty)$.

It remains to prove the above claim for $\psi$. Let $K>0$ and suppose $z\neq \omega$ and thus $\tilde{z}\neq \tilde{\omega}$. To keep the notation short, let us introduce the following abbreviations:
\begin{align*}
c:=-\frac{\pi}{2}\left( \nu + 1 \right); \quad q(\rho):=\cos\left( \tilde{z}\rho +c \right)\cos\left( \tilde{\omega}\rho +c \right).
\end{align*}
 One can easily show that the function $q$ has the following antiderivative:
\begin{align*}
Q(\rho) := \frac{1}{\tilde{z}^2-\tilde{\omega}^2}\left( \tilde{z}\cdot \sin\left( \tilde{z}\rho +c \right)\cos\left( \tilde{\omega}\rho +c  \right) - \tilde{\omega} \cdot \cos\left( \tilde{z}\rho +c \right)\sin\left( \tilde{\omega}\rho +c  \right)\right).
\end{align*}
This antiderivative is obviously bounded. Hence, using integration by parts,
\begin{align*}
\int\limits_{K}^{\infty} \psi(\rho) d\rho &= \lim\limits_{L\rightarrow\infty} \Bigg(  C \cdot i\left( \frac{Q(L)}{L} - \frac{Q(K)}{K} \right) + C \cdot i \int\limits_{K}^{L} \frac{Q(\rho)}{\rho^2} d\rho \Bigg) \\
&= - C \cdot i \left( \frac{Q(K)}{K}  - \int\limits_{K}^{\infty} \frac{Q(\rho)}{\rho^2} d\rho \right).
\end{align*}
Observe that the remaining integral on the right-hand side is even absolutely convergent. Thus the integral on the left-hand side must be convergent as well.

Suppose now $z=\omega$, and thus $\tilde{z}=\tilde{\omega}$. Again, one can easily show that for all $\tau\in\mathbb{C}$ the function $\mathbb{R}\ni\rho\mapsto \cos^2\left( \tilde{z}\rho +\tau \right)\in\mathbb{C}$ has the following antiderivative:
\begin{align}
\label{equation:AntiderivativeCosineSquare}
\rho\mapsto R(\rho;\tau):=\frac{\sin\left( \tilde{z}\rho +\tau  \right)\cos\left( \tilde{z}\rho +\tau  \right)}{2\tilde{z}} + \frac{\rho}{2}.
\end{align}
When we apply integration by parts, we obtain:
\begin{align*}
\int\limits_{K}^{\infty} \psi(\rho) d\rho = \lim\limits_{L\rightarrow\infty} \Bigg(  C \cdot i\left( \frac{R(L;c)}{L} - \frac{R(K;c)}{K} \right) +& \frac{C \cdot i}{2\tilde{z}}  \int\limits_{K}^{L} \frac{\sin\left( \tilde{z}\rho +c  \right)\cos\left( \tilde{z}\rho +c  \right)}{ \rho^2} d\rho \Bigg. \\
\Bigg. +& \frac{C\cdot i}{2}\ln\left( \frac{L}{K} \right) \Bigg).
\end{align*}
The right-hand side is not convergent as $L\rightarrow\infty$, because we have $\lim_{L\rightarrow\infty}\ln\left( \frac{L}{K} \right)=\infty$, and all the other terms on the right-hand side converge to some complex number for $L\rightarrow\infty$. Therefore, the left-hand side must be divergent as well. In other words, $\psi$ is not integrable if $z=\omega$.

It remains to show that $\psi$ is not absolutely integrable over $[K,\infty)$ for all values of $z,\omega\in (1,\infty)$. It is easy to show that $\vert \cos(s) \vert\geq \vert \cos(\Re(s)) \vert$ for all $s\in\mathbb{C}$. Thus we have 
\begin{align*}
\vert \psi(\rho) \vert \geq \frac{C}{\rho}\cdot \vert \cos\left( \tilde{z}\rho +\Re(c) \right)\vert \cdot \vert\cos \left( \tilde{\omega}\rho +\Re(c) \right) \vert,\quad \forall \rho\in(0,\infty).
\end{align*}
 If $z=\omega$ then we have for all $L>K$:
 \begin{align*}
\int\limits_{K}^{L} \vert \psi(\rho) \vert \, d\rho &\geq C  \int\limits_{K}^{L} \frac{1}{\rho}\cdot \cos^2\left( \tilde{z}\rho +\Re(c) \right) d\rho \\
&= C  \left( \frac{R(L;\Re(c))}{L} - \frac{R(K;\Re(c))}{K} \right) + \frac{C }{2\tilde{z}}  \int\limits_{K}^{L} \frac{\sin\left( \tilde{z}\rho +c  \right)\cos\left( \tilde{z}\rho +c  \right)}{ \rho^2} d\rho \Bigg. \\
\Bigg. &\qquad\qquad\qquad\qquad\qquad\qquad\quad\quad\,\,  + \frac{C }{2}\ln\left( \frac{L}{K} \right) \Bigg.,
 \end{align*}
where $\rho\mapsto R(\rho;\Re(c))$ is the function defined in \eqref{equation:AntiderivativeCosineSquare} for $\tau=\Re(c)$. The  right-hand side diverges to $\infty$ as $L\rightarrow\infty$. Thus $\psi$ is also not absolutely integrable.

Suppose now $z\neq\omega$. Then we have for all $L>K$:
\begin{align*}
\int\limits_{K}^{L} \vert \psi(\rho) \vert\,  d\rho \geq C\cdot \int\limits_{K}^{L} \frac{1}{\rho}\cdot \cos^2\left( \tilde{z}\rho +\Re(c) \right)\cos^2\left( \tilde{\omega}\rho + \Re(c) \right)  d\rho.
\end{align*}
The integral on the right-hand side diverges to $\infty$ as $L\rightarrow\infty$: An easy calculation yields that the function $\mathbb{R}\ni \rho\mapsto \cos^2\left( \tilde{z}\rho +\Re(c) \right)\cos^2\left( \tilde{\omega}\rho + \Re(c) \right)\in\mathbb{R}$ has an elementary antiderivative of the form
\begin{align*}
S(\rho):=\frac{\rho}{4} + u(\rho),
\end{align*}
where $u$ is a bounded function. Thus, with integration by parts, we obtain
\begin{align*}
\lim\limits_{L\rightarrow\infty} \int\limits_{K}^{L} &\frac{1}{\rho}\cdot \cos^2\left( \tilde{z}\rho +\Re(c) \right)\cos^2\left( \tilde{\omega}\rho + \Re(c) \right)  d\rho = \\
&= \lim\limits_{L\rightarrow\infty} \left( \left( \frac{S(L)}{L} - \frac{S(K)}{K} \right) + \int\limits_{K}^{L} \frac{u(\rho)}{\rho^2} d\rho + \frac{1}{4}\ln\left( \frac{L}{K} \right)\right) = \infty.
\end{align*}
Consequently, it follows that $\lim_{L\rightarrow\infty} \int_{K}^{L} \vert \psi(\rho) \vert d\rho = \infty$ and thus our proof is complete.
\end{proof}

Now we want to consider functions $g$ for which condition $1)$ of Corollary \ref{corollary:Corollary1Schritt1} is satisfied. Note that there exists always a sequence $\left( N_k \right)$ as in Lemma \ref{lemma:Schritt1} such that all curves $C_k$ can be chosen as semicircles centered at the origin. Therefore we can assume the curves $C_k$ to be semicircles, which is no problem because $1)$ is a condition of existence for the $C_k$.

\begin{lemma}
Let $g$ be a function as in Lemma $\ref{lemma:Schritt1}$. Let $\nu \in \mathbb{C}$ with $\Re(\nu)>-1$ and  $\omega, z\in (1,\infty)$ with $\omega<z$.
Suppose that $(C_k)_{k\in\mathbb{N}}$, is a sequence of curves as in Lemma $\ref{lemma:Schritt1}$, each of which is a semicircle centered at the origin. Suppose further that $g$ is bounded on the union of the images of all $C_k$. Then $\lim_{k\rightarrow\infty}\int_{C_k} g(s) e^{-i\pi s}P_{\nu}^{-s}(\omega)Q_{\nu}^{s}(z) ds = 0$. In particular, condition $1)$ of Corollary $\ref{corollary:Corollary1Schritt1}$ is satisfied.
\end{lemma}

\begin{proof}
Let $\tilde{z}, \tilde{\omega}\in (0,\infty)$ be associated with $z, \omega$ as in \eqref{equation:TildeDefinitionen}. Observe that the function 
\begin{align*}
x\mapsto \frac{x}{\left( x^2-1 \right)^{\nicefrac{1}{2}}},\quad x>1,
\end{align*}
is monotonically decreasing on $ \left( 1,\infty \right)$ and thus $\tilde{z}<\tilde{\omega}$.

Let $\left( N_k \right)_{k=1}^{\infty}\subset \left[1,\infty \right)$ be the unbounded and monotonically increasing sequence such that $C_k$ runs from $-iN_k$ to $iN_k$. Let $C_k$ be parametrized as
\begin{align*}
C_k:\left[ -\frac{\pi}{2}, \frac{\pi}{2} \right]\rightarrow \mathbb{C}, \quad \theta\mapsto N_k\cdot e^{i\theta}.
\end{align*}
By assumption, there exists some $A>0$ such that $\vert g(s)\vert\leq A$ for all $s\in \cup_{k=1}^{\infty} C_k([-\frac{\pi}{2}, \frac{\pi}{2}])$. Further, by \eqref{equation:AsympLegendreProdKond1} there exists some $K>0$ such that for all $s \in\mathbb{C}$ with $\Re(s)\geq 0$ and  $\vert s \vert>K$:
\begin{align}
\left\vert e^{-i\pi s}P_{\nu}^{-s}(\omega)Q_{\nu}^{s}(z) \right\vert &\leq \frac{1}{\vert s \vert} \left( \vert e^{-s(\tilde{\omega}-\tilde{z})} \vert + e^{-\pi \Im(\nu)}\cdot \vert  e^{-s\left( \tilde{z}+ \tilde{\omega} \right)} \vert \right) \nonumber \\ &\leq \frac{1+e^{-\pi\Im(\nu)}}{\vert s \vert} \cdot  e^{-\Re(s)(\tilde{\omega}-\tilde{z})}.
\end{align}
Therefore there exist some $q\in\mathbb{N}$ such that for all $k\geq q$:
\begin{align*}
\left\vert \int\limits_{C_k} g(s) e^{-i\pi s}P_{\nu}^{-s}(\omega)Q_{\nu}^{s}(z) ds \right\vert &\leq \left( 1+e^{-\pi\Im(\nu)} \right)A \int\limits_{-\frac{\pi}{2}}^{\frac{\pi}{2}} e^{-N_k\cos(\theta) \left( \tilde{\omega}- \tilde{z} \right)} d\theta \\
&= \left( 1+e^{-\pi\Im(\nu)} \right)2A \int\limits_{0}^{\frac{\pi}{2}} e^{-N_k\cos(\theta) \left( \tilde{\omega}- \tilde{z} \right)} d\theta.
\end{align*}
To finish the proof of this corollary, we will show that the last integral converges to $0$ as $k\rightarrow\infty$. 

Fix some $\delta \in (0,\frac{\pi}{2})$. Then 
\begin{align*}
\int\limits_{0}^{\frac{\pi}{2}} e^{-N_k\cos(\theta) \left( \tilde{\omega}- \tilde{z} \right)} d\theta =  \underbrace{\int\limits_{0}^{\frac{\pi}{2}-\delta} e^{-N_k\cos(\theta) \left( \tilde{\omega}- \tilde{z} \right)} d\theta}_{=:I_1} +  \underbrace{\int\limits_{\frac{\pi}{2}-\delta}^{\frac{\pi}{2}} e^{-N_k\cos(\theta) \left( \tilde{\omega}- \tilde{z} \right)} d\theta}_{=:I_2}.
\end{align*}
The first term converges obviously to $0$, since (recall $\tilde{z}<\tilde{\omega}$)
\begin{align*}
I_1\leq \left( \frac{\pi}{2}-\delta \right) e^{-N_k\cos(\frac{\pi}{2}-\delta) \left( \tilde{\omega}- \tilde{z} \right)}\longrightarrow 0\quad \text{ as } k\rightarrow\infty.
\end{align*}
The second integral converges to $0$ as well. In fact, with a simple substitution we get
\begin{align*}
I_2= \int\limits_{0}^{\delta}e^{-N_k\cos(-\theta+\frac{\pi}{2}) \left( \tilde{\omega}- \tilde{z} \right)} d\theta = \int\limits_{0}^{\delta}e^{-N_k\sin(\theta) \left( \tilde{\omega}- \tilde{z} \right)} d\theta.
\end{align*}
Now observe that $\sin(\theta)$ is concave on $\left[ 0,\frac{\pi}{2} \right]$, and thus $\sin(\theta)\geq \frac{2}{\pi}\theta$ for all $\theta\in \left[ 0,\frac{\pi}{2} \right]$. In particular $e^{-N_k\sin(\theta) \left( \tilde{\omega}- \tilde{z} \right)}\leq e^{-N_k \frac{2}{\pi}\theta \left( \tilde{\omega}- \tilde{z} \right)}$. Hence
\begin{align*}
I_2 \leq \int\limits_{0}^{\delta} e^{-N_k \frac{2}{\pi}\theta \left( \tilde{\omega}- \tilde{z} \right)} d\theta = \frac{\pi}{2 \left( \tilde{\omega}- \tilde{z} \right) N_{k}} \left( 1- e^{-N_k \frac{2}{\pi}\delta \left( \tilde{\omega}- \tilde{z} \right)} \right)\longrightarrow 0 \quad \text{ as } k\rightarrow \infty.
\end{align*}
\end{proof}

To get an overview of the preceding discussion, we summarize the relevant parts in a single theorem.

\begin{theorem}
\label{theorem:TheoremZentralIntegral}
Let $z,\omega\in (1,\infty)$ with $\omega<z$. Let $\nu\in \mathbb{C}$ with $\Re(\nu)>-1$. Suppose that $g$ is an odd and meromorphic function on $\mathbb{C}$, which is complex differentiable at all points on the imaginary axis except possibly with a simple pole at the origin. Suppose further:
\begin{itemize}
\item[$(i)$] Either $g(i\rho) = O\left( \frac{1}{\rho}\right)$ as $\rho \rightarrow \infty$ in $\mathbb{R}$ \emph{(}in which case we will say ``$g$ is of type $I"$\emph{)}, or there exists some constant  $C\in\mathbb{C}$ with $g(i\rho) = C \left( 1+ O\left( \frac{1}{\rho}\right) \right)$ as $\rho \rightarrow \infty$ in $\mathbb{R}$ \emph{(}``$g$ is of type $II"$\emph{)}.
\item[$(ii)$] There exists an unbounded and monotonically increasing sequence $\left( N_{k} \right)_{k\in\mathbb{N}}\subset [1,\infty)$ such that $g$ is bounded on $\cup_{k=1}^{\infty} S^{1}(N_k)$, where  $ S^{1}(N_k)$ denotes the circle of radius $N_k$ centered at the origin.
\end{itemize}
Then
\begin{align}
\label{equation:TheoremZentralIntegral}
&\frac{2}{\pi}\int\limits_{0}^{\infty} \frac{\sin(\pi i \rho)}{\pi} g\left(i\rho\right)Q_{\nu}^{-i\rho}(z)Q_{\nu}^{i\rho}(\omega) d\rho \nonumber \\
= & \lim\limits_{k\rightarrow\infty} \left( 2 \sum\limits_{p\in B^{+}\left( N_k \right)} \emph{Res}\left( g;p \right) e^{-i\pi p} P_{\nu}^{-p}(\omega)Q_{\nu}^{p}(z) \right)  + \emph{Res}\left( g;0 \right) P_{\nu}( \omega )Q_{\nu}(z),
\end{align} 
where $B^{+}\left( N_k \right):=\lbrace z\in\mathbb{C} : \vert z \vert < N_k,\text{ }\Re(z)>0 \rbrace$.

Further, we know that the integral on the left-hand side of \eqref{equation:TheoremZentralIntegral} is absolutely convergent for all $z,\omega\in (1,\infty)$ and $\nu\in\mathbb{C}$ with $\Re(\nu)\notin \{ -1, -2, -3,... \}$, if $g$ is of type $I$. If, instead, $g$ is of type $II$, then this integral is convergent if and only if $z\neq \omega$, and it is never absolutely convergent.
%
\end{theorem}

With this theorem, we can solve many integrals involving associated Legendre functions. 

\begin{example}
\label{example:LegendreIntegral}
Let $\nu\in \mathbb{C}$ with $\Re(\nu)>-1$ and $\omega,z\in (1,\infty)$ with $\omega<z$. The following examples of functions $g$ satisfy $\vert g(s) \vert = O\left( \frac{1}{\vert s \vert} \right)$ as $\vert s \vert\rightarrow\infty$ with $s\in\mathbb{C}$. In particular, these $g$ will satisfy the above conditions $(i), (ii)$ \emph{(}more precisely, they are of type $I$\emph{)} so we can apply Theorem $\ref{theorem:TheoremZentralIntegral}$ to solve the corresponding integrals. In order to simplify the integrands, we use the identity $\sin(is)=i\sinh(s)$ for $s\in\mathbb{C}$.
\begin{itemize}
\item[$(a)$] If $g(s)=\frac{1}{s}$, then:
\begin{align}
\label{equation:ErstesWichtigesIntegral}
\frac{2}{\pi}\int\limits_{0}^{\infty} \frac{\sinh(\pi \rho)}{\pi} \cdot \frac{1}{ \rho} Q_{\nu}^{-i\rho}(z)Q_{\nu}^{i\rho}(\omega) d\rho = P_{\nu}( \omega )Q_{\nu}(z).
\end{align}
\label{example:LegendreIntegral1}
\item[$(b)$] If $g(s)=\frac{s}{(s-1)(s+1)}=\frac{s}{s^2-1}$, then:
\begin{align*}
\frac{2}{\pi}\int\limits_{0}^{\infty} \frac{\sinh(\pi \rho)}{\pi} \cdot \frac{ \rho}{\rho^2+1} Q_{\nu}^{-i\rho}(z)Q_{\nu}^{i\rho}(\omega) d\rho = -P_{\nu}^{-1}(\omega)Q_{\nu}^{1}(z).
\end{align*}
\item[$(c)$] We can generalize the above examples. Let $\mathcal{H}_{>0}:=\lbrace s\in\mathbb{C} \mid \Re(s)>0 \rbrace$. For any $n\in\mathbb{N}_{0}$ and pairwise different $a_1,...,a_n\in\mathcal{H}_{>0}$, we define the following polynomials:
\begin{align*}
p\left(s \mid a_1,...,a_n\right):= \prod\limits_{i=1}^{n}\left( s-a_i \right)\left( s + a_i \right) = \prod\limits_{i=1}^{n}\left( s^2-a_i^2 \right).
\end{align*}
As usual we define the empty product as $1$, e.g. 
\begin{align*}
\prod_{\substack{i=1\\i\neq 1}}^{1}(...):=1 ,\quad \prod_{i=1}^{0}(...):=1.
\end{align*}
 Thus, if $n=0$, then $p\left(s \mid a_1,...,a_n\right):=1.$

Now let $m,n\in\mathbb{N}_{0}$ be given. We fix pairwise different $b_1,...,b_m\in\mathcal{H}_{>0}$ and pairwise different  $a_1,...,a_n\in\mathcal{H}_{>0}$. Then the following functions satisfy the conditions of Theorem $\ref{theorem:TheoremZentralIntegral}$:
\begin{align*}
g_1(s):=\frac{1}{s}\cdot \frac{p\left( s\mid a_1,...,a_n \right)}{p\left( s\mid b_1,...,b_m \right)}\quad \text{if }\, m\geq n; \quad 
g_2(s):= s \cdot \frac{p\left( s\mid a_1,...,a_n \right)}{p\left( s\mid b_1,...,b_m \right)},\quad \text{if }\, m > n.
\end{align*}
For $g_1$ we obtain the integrals
\begin{align*}
\frac{2}{\pi}\int\limits_{0}^{\infty}& \frac{\sinh(\pi  \rho)}{\pi} \cdot \frac{(-1)^{n+m}}{ \rho}\cdot \left( \frac{\prod\limits_{i=1}^{n}\left( \rho^2+a_i^2\right)}{\prod\limits_{i=1}^{m}\left(\rho^2+b_i^2\right)}\right) Q_{\nu}^{-i\rho}(z)Q_{\nu}^{i\rho}(\omega) d\rho = \\
& \left(\mathlarger{\sum\limits_{k=1}^{m}} \left(\frac{\prod\limits_{i=1}^{n}\left( b_k^2-a_i^2 \right)}{\prod\limits_{\substack{i=1\\i\neq k}}^{m}\left(b_k^2-b_i^2\right) }\right) \frac{e^{-i\pi b_k}}{b_k} P_{\nu}^{-b_k}( \omega )Q_{\nu}^{b_k}(z)\right) + (-1)^{n+m} \left( \frac{\prod\limits_{i=1}^{n}a_i^2}{\prod\limits_{i=1}^{m}b_i^2 }\right)P_{\nu}( \omega )Q_{\nu}(z).
\end{align*}
For $g_2$ we obtain
\begin{align*}
\frac{2}{\pi}\int\limits_{0}^{\infty}& \frac{\sinh(\pi  \rho)}{\pi} (-1)^{n+m} \cdot \rho\cdot \left( \frac{\prod\limits_{i=1}^{n}\left( \rho^2+a_i^2\right)}{\prod\limits_{i=1}^{m}\left(\rho^2+b_i^2\right)}\right) Q_{\nu}^{-i\rho}(z)Q_{\nu}^{i\rho}(\omega) d\rho = \\
& -\left(\mathlarger{\sum\limits_{k=1}^{m}} \left(\frac{\prod\limits_{i=1}^{n}\left( b_k^2-a_i^2 \right)}{\prod\limits_{\substack{i=1\\i\neq k}}^{m}\left(b_k^2-b_i^2\right) }\right) e^{-i\pi b_k}P_{\nu}^{-b_k}( \omega )Q_{\nu}^{b_k}(z)\right).
\end{align*}
\item[$(d)$] Let $c\in\mathbb{C}$ with $\Re(c)>0$, $n\in\mathbb{N}$, $m\in\mathbb{N}_{0}$ such that $2m+1<4n,$ and set $g(s):=\frac{s^{2m+1}}{(s-c)^{2n}(s+c)^{2n}}$. The only poles of $g$ are of order $2n$ at the points $c$ and $-c$ and one can show that:
\begin{align*}
2\cdot \Res(g;c) = 
\begin{cases}
0, \text{ if } 2m+1 < 4n-1 \\
1, \text{ if } 2m+1 = 4n-1.
\end{cases}
\end{align*}
Thus

\begin{align*}
\frac{2}{\pi}\int\limits_{0}^{\infty} \frac{\sinh(\pi \rho)}{\pi} &\cdot \frac{(-1)^{m+1}\rho^{2m+1}}{\left( \rho^2+c^2 \right)^{2n}} Q_{\nu}^{-i\rho}(z)Q_{\nu}^{i\rho}(\omega) d\rho = \nonumber \\
&= \begin{cases}
0, &\text{if } 2m+1 < 4n-1 \\
e^{-i\pi c}P_{\nu}^{-c}( \omega )Q_{\nu}^{c}(z), &\text{if } 2m+1 = 4n-1.
\end{cases}
\end{align*}
\end{itemize}
\end{example}

%
\begin{remark}
One can show that, equation \eqref{equation:ErstesWichtigesIntegral} is also valid if $\omega = z$ and $\nu\in (-1,\infty)$ (see \cite[Lemma 3.13 $(ii)$]{ErenDiss}).
\end{remark}

We want to stress a special case of the previous theorem by the following corollary. 
\begin{corollary}
\label{corollary:Corollary1Schritt2}
Let $\nu\in \mathbb{C}$ with $\Re(\nu)>-1$ and $\omega,z\in (1,\infty)$ with $\omega<z$. Let $f:\mathbb{C}\rightarrow\mathbb{C}$ be an entire and even function, i.e. $f(s)=f(-s)$ for all $s\in\mathbb{C}$. 
Consider the function 
\begin{align*}
g(s):=\frac{\pi}{\sin(\pi s)}\cdot f(s),
\end{align*}
and the sequence $(N_k)_{k\in\mathbb{N}}$ with $N_k:=k+\frac{1}{2}$. 
Suppose that the conditions $(i)$ and $(ii)$ of Theorem $\ref{theorem:TheoremZentralIntegral}$ are satisfied with this particular $g$ and sequence $\left(N_k\right)_{k\in\mathbb{N}}$. Then we have
\begin{align}
\label{eqiaton:Corollary2Schritt1}
\frac{2}{\pi}\int\limits_{0}^{\infty} f\left(i\rho\right)Q_{\nu}^{-i\rho}(z)Q_{\nu}^{i\rho}(\omega) d\rho = 2 \left( \sum\limits_{m=1}^{\infty} f(m) P_{\nu}^{-m}(\omega)Q_{\nu}^{m}(z) \right)  + f(0) P_{\nu}( \omega )Q_{\nu}(z).
\end{align} 
\end{corollary}

\begin{proof}
Observe that $s\mapsto \frac{\pi}{\sin(\pi s)}$ is an odd function and meromorphic on $\mathbb{C}$. The only singularities are simple poles at all integer points $p\in\mathbb{Z}$ with residue $\Res\left( s\mapsto \frac{\pi}{\sin(\pi s)};p \right)=\left(-1\right)^{p}$. Thus $g$ satisfies the assumptions of Theorem \ref{theorem:TheoremZentralIntegral} with $\Res\left( g;p \right)=\left( -1 \right)^p f(p)$. Therefore \eqref{eqiaton:Corollary2Schritt1} follows when we apply \eqref{equation:TheoremZentralIntegral} to this case.
\end{proof}

Let us apply Corollary \ref{corollary:Corollary1Schritt2} to solve an important integral, which will reappear in the next chapter.

\begin{example}
\label{example:IntegralFuerGreenSumme}
Let $\nu\in \mathbb{C}$ with $\Re(\nu)>-1$ and $\omega,z\in (1,\infty)$ with $\omega<z$. Further, let $f(s):=\cos\left( s\cdot \left( \pi-\theta \right) \right)$, with $\theta\in\left[ 0, 2\pi \right)$. We set $g(s):=\frac{\pi}{\sin(\pi s)}\cdot f(s)$ and $N_k:=k+\frac{1}{2}$ as in Corollary \emph{\ref{corollary:Corollary1Schritt2}}. We first check if the conditions $(i)$ and $(ii)$ of Theorem \emph{\ref{theorem:TheoremZentralIntegral}} are satisfied. 

Note that for $\rho\in\mathbb{R}$ we have
\begin{align*}
g(i\rho) = \frac{\pi}{i} \cdot \frac{\cosh((\pi-\theta)\rho)}{\sinh(\pi\rho)} = \frac{\pi}{i} \cdot \frac{e^{-\theta \rho }+ e^{-(2\pi-\theta)\rho}}{1-e^{-2\pi\rho}}.
\end{align*}
If $\theta\in (0,2\pi)$ then this is in $O(e^{-\alpha\rho})$ as $\rho\rightarrow\infty$ with $\alpha:=\min\lbrace \theta, 2\pi-\theta \rbrace>0$. In particular, it is in $O\left( \frac{1}{\rho} \right)$ as $\rho\rightarrow\infty$, so $g$ is of type I. If $\theta = 0$, then 
\begin{align*}
g(i\rho) = \frac{\pi}{i}\cdot \frac{1 + e^{-2\pi\rho}}{1-e^{-2\pi\rho}} = \frac{\pi}{i} \cdot \left( 1+ \frac{2e^{-2\pi \rho}}{1-e^{-2\pi\rho}} \right).
\end{align*}
In particular, $g$ is of type II \emph{(}with $C:=\frac{\pi}{i}$ in Theorem \emph{\ref{theorem:TheoremZentralIntegral}} $(i)$\emph{)}. Thus, condition $(i)$ of Theorem \emph{\ref{theorem:TheoremZentralIntegral}} is satisfied.

Note that $g$ is bounded on the set $\cup_{k=1}^{\infty} S^1(N_k)$ if and only if $g^2$ is bounded on this set. Furthermore, if $x:=\Re(s),\, y:=\Im(s)$ so that $s=x+iy$, then a straightforward calculation shows:
\begin{align*}
\vert g(s) \vert^2 =\pi^2 \cdot \frac{\vert \cos((\pi-\theta)(x+iy))\vert^2}{\vert \sin(\pi(x+iy))\vert^2} = \pi^2 \cdot \frac{\cosh(2(\pi-\theta)y) + \cos(2(\pi-\theta)x)}{\cosh(2 \pi y) - \cos(2 \pi x)}.
\end{align*}
Let $y_0>0$ be such that 
\begin{align*}
\left\lbrace \, s\in S^1(N_k) \, \Big| \, \vert \Im(s) \vert\leq y_0 \,\right\rbrace \subset \left\lbrace  s\in S^1(N_k) \, \Big| \, \vert \Re(s) \vert \in \left[ k+\frac{1}{3},k+\frac{1}{2}  \right]  \right\rbrace
\end{align*}
for all $k\in\mathbb{N}$ \emph{(}or equivalently, for $k=1$\emph{)}. Then on $\cup_{k=1}^{\infty} \{\, s\in S^1(N_k) \mid \vert \Im(s) \vert\leq y_0  \,\}$ we have $\cos(2\pi x)\in  \left[ -1, -\frac{1}{2}\right]$ and thus
\begin{align*}
\vert g(s) \vert^2 \leq 2\pi^2 \cdot \frac{\cosh(2(\pi-\theta)y)}{\cosh(2\pi y)}\leq 2\pi^2.
\end{align*}
On $\{\, s\in\mathbb{C} \mid \vert \Im(s) \vert\geq y_0 \,\}$ we have
\begin{align*}
\vert g(s) \vert^2 \leq \pi^2 \cdot \frac{\cosh(2(\pi-\theta)y) + 1}{\cosh(2 \pi y) - 1}.
\end{align*}
Since $y\mapsto \pi^2 \cdot \frac{\cosh(2(\pi-\theta)y) + 1}{\cosh(2 \pi y) - 1}$ has a limit for $y\rightarrow\infty$ it is bounded on $[y_0,\infty)$. We conclude that $g^2$ is bounded on $\{\, s\in\mathbb{C} \mid \vert \Im(s) \vert\geq y_0 \,\}$, too. In particular, condition $(ii)$ of Theorem \emph{\ref{theorem:TheoremZentralIntegral}} is satisfied by the function $g$ with the above $N_k$. Therefore, by Corollary \emph{\ref{corollary:Corollary1Schritt2}} we get:
\begin{align*}
\frac{2}{\pi}\int\limits_{0}^{\infty} \cosh\left(\rho \left( \pi-\theta  \right) \right)Q_{\nu}^{-i\rho}(z)  Q_{\nu}^{i\rho}(\omega) d\rho  &= 2 \left( \sum\limits_{m=1}^{\infty} \cos\left(m \left( \pi-\theta  \right) \right) P_{\nu}^{-m}(\omega)Q_{\nu}^{m}(z) \right) +  P_{\nu}( \omega )Q_{\nu}(z)  \\
&= 2 \left( \sum\limits_{m=1}^{\infty} (-1)^m\cos\left(m \theta  \right) P_{\nu}^{-m}(\omega)Q_{\nu}^{m}(z) \right) +  P_{\nu}( \omega )Q_{\nu}(z)  \\
&=Q_{\nu}\left( \omega z -\sqrt{\omega^2-1} \sqrt{z^2-1}\cos\left( \theta \right) \right),
\end{align*}
where the last equality is a classical addition formula \emph{(}see e.g. \emph{\cite[$8.795\text{ } 2$]{Gradshteyn}}\emph{)}.
\end{example}

In the above example we needed to assume $\omega<z$ in order to apply Corollary \ref{corollary:Corollary1Schritt2}, but it turns out that the above formula is also valid for more general values of $\omega$ and $z$. Before we generalize that formula, we will prove some estimates in the following lemma, which will also be useful later.

\begin{lemma}
\label{lemma:EstimateProductLegendreFunctions}
Let $\omega, z\in (1,\infty)$ and $\rho\in [0,\infty)$ be arbitrary. 
\begin{itemize}
\item[$(i)$]
For all $\nu\in\mathbb{C}$ with $\Re(\nu)\geq -\frac{1}{2} $:
\begin{align}
\label{equation:EstimateProductLegendreFunctions}
\vert  Q_{\nu}^{-i\rho}(z)Q_{\nu}^{i\rho}(\omega)\vert \leq \frac{\pi^4}{\sqrt{z-1} \cdot \sqrt{\omega-1}} \cdot \frac{\left( \vert \nu \vert + 1 +\rho\right)^{2\vert \nu \vert +1}}{\vert \Gamma(\nu+1)\vert^2}\cdot e^{-\pi\rho}.
\end{align}
\item[$(ii)$] For all $\nu\in\mathbb{C}$ with $\Re(\nu)\geq 0 $:
\begin{align}
\label{equation:EstimateProductLegendreFunctionsForSmallValuesOfZ}
\vert  Q_{\nu}^{-i\rho}(z)Q_{\nu}^{i\rho}(\omega)\vert \leq \ln\left( \frac{z+1}{z-1} \right)\ln\left( \frac{\omega+1}{\omega-1} \right) \cdot \frac{\pi^2 \left( \vert \nu \vert + 1 +\rho\right)^{2\vert \nu \vert +1}}{\vert \Gamma(\nu+1)\vert^2}\cdot e^{-\pi\rho}.
\end{align}
\item[$(iii)$] For all $\nu\in\mathbb{C}$ with $-1<\Re(\nu)<-\frac{1}{2}$:
\begin{align}
\label{equation:EstimateProductLegendreFunctionsGreaterRange}
\vert  Q_{\nu}^{-i\rho}(z)Q_{\nu}^{i\rho}(\omega)\vert \leq \frac{\pi^3 e^{\frac{1}{3(\Re(\nu)+1)}} }{\vert \Gamma(\nu+1)\vert^2 \cdot ((z-1)(\omega-1))^{\Re(\nu)+1} (\Re(\nu)+1)^{1-2\Re(\nu)} }\, e^{-\pi\rho}.
\end{align}
\end{itemize}
\end{lemma}

\begin{proof}
Let $\omega, z\in (1,\infty)$, $\rho\in [0,\infty)$, and $\nu\in\mathbb{C}$ with $\Re(\nu)> -1 $ be given. From \cite[formula $(5)$ on p. $155$]{Erdelyi} we know that for all $\mu \in\mathbb{C}$ with $\Re(\mu)\geq 0$:
\begin{align}
\label{equation:RepresentationLegendreQIntegral}
 Q_{\nu}^{\mu}(z) = e^{\mu\pi i}\cdot \frac{1}{2^{\nu+1}}\cdot \frac{\Gamma(\nu + 1 +\mu)}{\Gamma(\nu + 1)} \cdot \frac{1}{(z^2-1)^{\frac{\mu}{2}}} \cdot \int\limits_{0}^{\pi} \frac{\sin(t)^{2\nu+1}}{(z+\cos(t))^{-\mu+\nu+1}}\, dt.
\end{align}
Hence,
\begin{align}
\label{equation:AbschaetzungProduktLegendre1}
\vert Q_{\nu}^{-i\rho}(z)  Q_{\nu}^{i\rho}(\omega) \vert \leq & \frac{1}{4^{\Re(\nu)+1}}\cdot \frac{\vert \Gamma(\nu + 1 + i\rho) \Gamma(\nu + 1 - i\rho)\vert}{\vert \Gamma(\nu + 1)\vert^2 }  \cdot \int\limits_{0}^{\pi} \frac{\sin(t)^{2\Re(\nu)+1}}{(z+\cos(t))^{\Re(\nu)+1}}\, dt \, \cdot   \nonumber \\
& \cdot \int\limits_{0}^{\pi} \frac{\sin(t)^{2\Re(\nu)+1}}{(\omega+\cos(t))^{\Re(\nu)+1}}\, dt.
\end{align}

$(i)$. Suppose $\Re(\nu)\geq -\frac{1}{2}$. Then we can rewrite the integrals appearing above as follows:
\begin{align}
\int\limits_{0}^{\pi} \frac{\sin(t)^{2\Re(\nu)+1}}{(z+\cos(t))^{\Re(\nu)+1}}\, dt &= \int\limits_{0}^{\pi} \left(\frac{1-\cos^2(t)}{z+\cos(t)}\right)^{\Re(\nu)+\frac{1}{2}} \cdot \frac{1}{\sqrt{z+\cos(t)}}\, dt \nonumber \\
&=\int\limits_{0}^{\pi}\left( \frac{1+\cos(t)}{z+\cos(t)}\right)^{\Re(\nu)+\frac{1}{2}}  \cdot \frac{ \left( 1-\cos(t) \right)^{\Re(\nu)+\frac{1}{2}}}{\sqrt{z+\cos(t)}}\, dt. \label{equation:AbschaetzungProduktLegendreZwischenschritt}
\end{align}
Using the estimates $0\leq \frac{1+\cos(t)}{z+\cos(t)}<1$, $0\leq 1-\cos(t)\leq 2$, and $0< \frac{1}{\sqrt{z+\cos(t)}}\leq\frac{1}{\sqrt{z-1}}$ for all $t\in [0,\pi]$, we obtain from \eqref{equation:AbschaetzungProduktLegendreZwischenschritt}:
\begin{align}
\label{equation:AbschaetzungProduktLegendre2}
\int\limits_{0}^{\pi} \frac{\sin(t)^{2\Re(\nu)+1}}{(z+\cos(t))^{\Re(\nu)+1}}\, dt \, \leq \, 2^{\Re(\nu)+\frac{1}{2}} \cdot \frac{\pi}{\sqrt{z-1}}.
\end{align}
Hence, from \eqref{equation:AbschaetzungProduktLegendre1} and \eqref{equation:AbschaetzungProduktLegendre2} we get
\begin{align}
\label{equation:EstimateProductLegendreFunctionsBeweisZwischenschritt}
\vert  Q_{\nu}^{-i\rho}(z)Q_{\nu}^{i\rho}(\omega)\vert \leq \frac{\pi^2}{\sqrt{z-1} \cdot \sqrt{\omega-1}} \cdot \frac{\vert \Gamma(\nu+1+i\rho)\Gamma(\nu+1-i\rho) \vert}{2\cdot \vert \Gamma(\nu+1)\vert^2}.
\end{align}

%


Moreover, it is well-known (see \cite[formula $5.6.9$ on p. 138]{NIST}) that for all $z\in \mathbb{C}$ with $\Re(z) > 0$:
\begin{align}
\label{equation:GammaFunctionEstimateGeneral}
\vert \Gamma(z) \vert \leq \sqrt{2\pi} \cdot \vert z \vert ^{\Re(z)-\frac{1}{2}} \cdot e^{-\frac{\pi}{2}\vert \Im(z) \vert} \cdot e^{\frac{1}{6 \vert z \vert}}.
\end{align} 
Thus, 
\begin{align}
\label{equation:GammaImaginaryEstimate}
\vert \Gamma(\nu+1+i\rho)\Gamma(\nu+1-i\rho) \vert &\leq  2\pi e^{\frac{2}{3}}\cdot \left( \vert \nu \vert + 1 + \rho \right)^{2\Re(\nu)+1}\cdot e^{-\frac{\pi}{2}(\vert \Im(\nu) + \rho \vert + \vert \Im(\nu) - \rho \vert)} \nonumber \\
&\leq  2\pi^2 \cdot \left( \vert \nu \vert + 1 + \rho \right)^{2\vert \nu \vert+1} \cdot e^{-\pi\rho}.
\end{align}
Hence, $(i)$ follows from \eqref{equation:GammaImaginaryEstimate} and \eqref{equation:EstimateProductLegendreFunctionsBeweisZwischenschritt}.

$(ii).$ Suppose $\Re(\nu)\geq 0$. Similar as above, we can estimate the integrals on the right-hand side of \eqref{equation:AbschaetzungProduktLegendre1} as follows:
\begin{align*}
\int\limits_{0}^{\pi} \frac{\sin(t)^{2\Re(\nu)+1}}{(z+\cos(t))^{\Re(\nu)+1}}\, dt \, & = \int\limits_{0}^{\pi} \left( \frac{1+\cos(t)}{z+\cos(t)}\right)^{\Re(\nu)} \cdot (1-\cos(t))^{\Re(\nu)}\cdot \frac{\sin(t)}{z+\cos(t)} \, dt \\
& \leq 2^{\Re(\nu)} \int\limits_{0}^{\pi} \frac{\sin(t)}{z+\cos(t)} \, dt = 2^{\Re(\nu)} \ln\left(\frac{z+1}{z-1}\right).
\end{align*}
Hence, $(ii)$ follows from the above estimate combined with \eqref{equation:AbschaetzungProduktLegendre1} and \eqref{equation:GammaImaginaryEstimate}.

$(iii)$. Suppose $-1<\Re(\nu)<-\frac{1}{2}$, and thus $-1<2\Re(\nu)+1<0$. In that case, we estimate the integrals on the right-hand side of \eqref{equation:AbschaetzungProduktLegendre1} as follows:
\begin{align*}
\int\limits_{0}^{\pi} \frac{\sin(t)^{2\Re(\nu)+1}}{(z+\cos(t))^{\Re(\nu)+1}}\, dt \, &\leq \frac{1}{(z-1)^{\Re(\nu)+1}}\int_{0}^{\pi}\sin(t)^{2\Re(\nu)+1} \, dt  \\
&\, =\frac{2}{(z-1)^{\Re(\nu)+1}}\int_{0}^{\frac{\pi}{2}}\sin(t)^{2\Re(\nu)+1} \, dt.
\end{align*}
Note that $[0,\frac{\pi}{2}]\ni t\mapsto \sin(t)\in\mathbb{R}$ is concave and thus $\sin(t)\geq \frac{2}{\pi}t$. Hence, from the above estimate, we obtain
\begin{align*}
\int\limits_{0}^{\pi} \frac{\sin(t)^{2\Re(\nu)+1}}{(z+\cos(t))^{\Re(\nu)+1}}\, dt \,&\,\leq  \frac{2}{(z-1)^{\Re(\nu)+1}}\int_{0}^{\frac{\pi}{2}} \left(\frac{2}{\pi}t\right)^{2\Re(\nu)+1} \, dt \\
&\, =  \frac{\pi}{2\cdot (z-1)^{\Re(\nu)+1} (\Re(\nu)+1)}.
\end{align*}
Using that upper bound and the estimate \eqref{equation:AbschaetzungProduktLegendre1}  we have 
\begin{align*}
\vert  Q_{\nu}^{-i\rho}(z)Q_{\nu}^{i\rho}(\omega)\vert \leq \frac{\pi^2}{2(z-1)^{\Re(\nu)+1} (\omega-1)^{\Re(\nu)+1} (\Re(\nu)+1)^2} \cdot \frac{\vert\Gamma(\nu+1+i\rho)\Gamma(\nu+1-i\rho) \vert }{\vert \Gamma(\nu+1)\vert^2 }.
\end{align*}
Furthermore, from \eqref{equation:GammaFunctionEstimateGeneral} we obtain
\begin{align*}
\vert \Gamma(\nu+1+i\rho)\Gamma(\nu+1-i\rho) \vert \leq (2\pi) e^{\frac{1}{3(\Re(\nu)+1)}} (\Re(\nu)+1)^{2\Re(\nu)+1} e^{-\pi\rho},
\end{align*}
and therefore
\begin{align*}
\vert  Q_{\nu}^{-i\rho}(z)Q_{\nu}^{i\rho}(\omega)\vert &\leq \frac{\pi^3 e^{\frac{1}{3(\Re(\nu)+1)}}}{\vert \Gamma(\nu+1)\vert^2 (z-1)^{\Re(\nu)+1} (\omega-1)^{\Re(\nu)+1} (\Re(\nu)+1)^{1-2\Re(\nu)} }\cdot e^{-\pi \rho}.
\end{align*}
\end{proof}

\begin{corollary}
\label{corollary:IntegralProduktLegendreFuerGreensFunction}
Let $\nu\in \mathbb{C}$ with $\Re(\nu)>-1$, $\theta\in [0,2\pi)$ and $\omega,z\in (1,\infty)$ such that $\omega\neq z$. Then
\begin{align}
\label{example:IntegralFuerGreen}
\frac{2}{\pi}\int\limits_{0}^{\infty} \cosh\left(\rho \left( \pi-\theta  \right) \right)Q_{\nu}^{-i\rho}(z)Q_{\nu}^{i\rho}(\omega) d\rho = Q_{\nu}\left( \omega z -\sqrt{\omega^2-1} \sqrt{z^2-1}\cos\left( \theta \right) \right).
\end{align}
Furthermore, the above equality holds for all $\theta\in (0,2\pi)$ and $\omega=z\in (1,\infty)$.
\end{corollary}

\begin{proof}
Suppose $\theta\in [0,2\pi)$ and $\omega,z\in (1,\infty)$ such that $\omega\neq z$. For $\omega<z$ the proof was given in Example \ref{example:IntegralFuerGreenSumme}. In addition to that, observe that both sides of \eqref{example:IntegralFuerGreen} are symmetric in the variables $\omega,z$. The right-hand side is obviously symmetric in those variables, and the left-hand side is symmetric since the term $Q_{\nu}^{-i\rho}(z)Q_{\nu}^{i\rho}(\omega)$ of the integrand is symmetric due to \eqref{equation:SymmetrieProduktLegendreQ}. Hence the equation holds also if $z<\omega$.

Now suppose that we have $\theta\in (0,2\pi)$ and $\omega = z\in (1,\infty)$. Observe that the function on the right-hand side of \eqref{example:IntegralFuerGreen} is continuous in $z\in (1,\infty)$ and thus for any sequence $\left( z_n \right)_{n\in\mathbb{N}}\subset (1,\infty)$ with $z_n\neq z$ for all $n\in\mathbb{N}$ and  $\lim_{n\rightarrow\infty} z_n = z$ we have:
\begin{align*}
 Q_{\nu}\left( \omega z -\sqrt{\omega^2-1}\sqrt{z^2-1}\cos\left( \theta \right) \right) &= \lim\limits_{n\rightarrow\infty} Q_{\nu}\left( \omega z_n -\sqrt{\omega^2-1} \sqrt{z_n^2-1}\cos\left( \theta \right) \right) \\
 &=\lim\limits_{n\rightarrow\infty} \frac{2}{\pi}\int\limits_{0}^{\infty} \cosh\left(\rho \left( \pi-\theta  \right) \right)Q_{\nu}^{-i\rho}(z_n)Q_{\nu}^{i\rho}(\omega) d\rho \\
&= \frac{2}{\pi}\int\limits_{0}^{\infty} \cosh\left(\rho \left( \pi-\theta  \right) \right)Q_{\nu}^{-i\rho}(z)Q_{\nu}^{i\rho}(\omega) d\rho,
\end{align*}
where we obtain the last equality by Lebesgue's dominated convergence theorem: The integrand is continuous with respect to $z\in (1,\infty)$, and the sequence of integrands is dominated by some integrable function due to Lemma \ref{lemma:EstimateProductLegendreFunctions} $(i)$ and $(iii)$. In fact, by Lemma \ref{lemma:EstimateProductLegendreFunctions} and the estimate $\cosh(\rho(\pi-\theta))\leq e^{\, \rho\vert \pi-\theta \vert}$ for all $\rho\geq 0$, there exist $C,\varepsilon>0$ such that for all $\rho\geq 0$ and $n\in\mathbb{N}$: $\vert \cosh\left(\rho \left( \pi-\theta  \right) \right)Q_{\nu}^{-i\rho}(z_n)Q_{\nu}^{i\rho}(\omega) \vert\leq C \cdot e^{-\rho\varepsilon}$. This upper bound is obviously integrable over $\rho\in[0,\infty)$.
\end{proof}

\begin{remark}
One can consider the above integrals as certain so-called generalized Mehler-Fock transforms (see e.g. \cite{Oberhettinger}). One only needs to apply Whipple's formula \eqref{equation:Whipple1} to the term $Q_{\nu}^{-i\rho}(z)$ on the left-hand side of \eqref{equation:TheoremZentralIntegral}. The corresponding ordinary Mehler-Fock transform is obtained if we additionally set $\nu = -\frac{1}{2}$. Note that if we apply Whipple's formula \eqref{equation:Whipple1} also to the term $Q_{\nu}^{i\rho}(\omega)$ on the left-hand side of \eqref{equation:TheoremZentralIntegral}, then some terms under the integral cancel each other out and the integrand becomes shorter. In particular, if $\nu = -\frac{1}{2}$, then the integrand can be simplified further due to the formula (see \cite[8.334 2]{Gradshteyn})
\begin{align*}
\Gamma\left(\frac{1}{2} + i\rho \right)\Gamma\left(\frac{1}{2} - i\rho \right) = \frac{\pi}{\cos(\pi i\rho)},\quad \rho\in \mathbb{R}.
\end{align*}
\end{remark}

\begin{lemma}
\label{lemma:HolomorphieIntegralLegendreProdukt}
Let $\omega, z\in (1,\infty)$. Further, let $h:[0,\infty)\rightarrow\mathbb{R}$ be continuous such that there exists some $\varepsilon\in (0,\pi)$ with $h(\rho)=O(e^{(\pi-\varepsilon)\rho})$ as $\rho\rightarrow\infty$. Then the function
\begin{align*}
F:\mathcal{H}_{>-\frac{1}{2}} \ni \nu\mapsto \int\limits_{0}^{\infty} Q_{\nu}^{-i\rho}(z)Q_{\nu}^{i\rho}(\omega) h(\rho) d\rho \in\mathbb{C}
\end{align*}
is holomorphic on $\mathcal{H}_{>-\frac{1}{2}}:=\{\, z\in\mathbb{C} \mid \Re(z)>-\frac{1}{2}\,  \}$.
\end{lemma}

\begin{proof}
Let $f(\nu,\rho):=Q_{\nu}^{-i\rho}(z)Q_{\nu}^{i\rho}(\omega) h(\rho)$ for all $\nu\in \mathcal{H}_{>-\frac{1}{2}}$ and $\rho \in [0,\infty)$. Because of \cite[Satz $5.8$ on p. 148]{Elstrodt} it suffices to show for all non-empty compact sets $K\subset \mathcal{H}_{>-\frac{1}{2}}$ the existence of some integrable function $g_K:[0,\infty)\rightarrow\mathbb{R}$ with $\sup_{\nu\in K}\vert f(\nu,\rho) \vert\leq g_K(\rho)$ for all $\rho\in [0,\infty)$. Note that the other conditions of \cite[Satz $5.8$ on p. 148]{Elstrodt} are obviously satisfied by $f$, namely $f(\nu,\cdot)$ is integrable for all $\nu\in \mathcal{H}_{>-\frac{1}{2}}$ because of Theorem \ref{theorem:TheoremZentralIntegral} and $f(\cdot, \rho)$ is holomorphic on $\mathcal{H}_{>-\frac{1}{2}}$ for all $\rho\geq 0$.

Let $K\subset \mathcal{H}_{>-\frac{1}{2}}$ be an arbitrary non-empty compact set. From \eqref{equation:EstimateProductLegendreFunctions} and the assumptions on $h$, there exist constants $\tilde{C}>0$ and $\varepsilon\in (0,\pi)$ such that for all $\nu \in K$ and $\rho\in[0,\infty)$:
\begin{align}
\label{equation:BeweisHolomorphie1}
\vert f(\nu,\rho) \vert \leq \tilde{C}\cdot \frac{(\vert\nu \vert + 1 +\rho)^{2\vert \nu \vert+1}}{\vert \Gamma(\nu+1)\vert^2} \cdot  e^{ -\varepsilon \rho}.
\end{align}

With $D:=\sup_{\nu\in K }\left\lbrace \frac{1}{\vert \Gamma(\nu+1)\vert^2} \right\rbrace$, $M:=\sup_{\nu\in K}\{ \vert \nu \vert \}$ we have for all $\nu\in K$ and $\rho\in[0,\infty)$:
\begin{align*}
\vert f(\nu,\rho) \vert \leq \tilde{C} D   \left( M + 1 + \rho \right)^{2M+1}\cdot e^{-\varepsilon \rho}\leq C\cdot e^{-\frac{\varepsilon}{2}\rho}=: g_K(\rho),
\end{align*}
where $C:=\sup_{\rho \geq 0} \{ \tilde{C}D(M+1+\rho)^{2M+1}\cdot e^{-\frac{\varepsilon}{2}\rho} \}\in (0,\infty)$. That function $g_K$ is obviously integrable over $[0,\infty)$ and satisfies $\sup_{\nu\in K}\vert f(\nu,\rho) \vert\leq g_K(\rho)$ for all $\rho\in [0,\infty)$. Hence, $F$ must be holomorphic.
\end{proof}

\begin{lemma}
\label{lemma:SecondEstimateProductLegendreFunctions}
Let $\nu\in (0,\infty)$ be fixed.
\begin{itemize}
\item[$(i)$] For any $b \in (0,\infty)$ there exist constants $C,D>0$ \emph{(}depending on $\nu$ and b\emph{)} such that for all $a\in (0,\infty)$ and $\mu>0$:
\begin{align}
\label{equation:SecondEstimateProductLegendreFunctions}
\vert P_{\nu}^{-\mu}(\cosh(a))Q_{\nu}^{\mu}(\cosh(b)) \vert \leq Ce^{\left(\nu+\frac{1}{2}\right)a} \cdot  \mu^{\nu+\frac{1}{2}} \cdot \left(D\cdot \sinh\left(\frac{a}{2}\right)\right)^{\mu}.
\end{align}
\item[$(ii)$] For all $a,b\in (0,\infty)$
\begin{align}
\label{equation:SecondEstimateProductLegendreFunctions2}
\vert P_{\nu}(\cosh(a))Q_{\nu}(\cosh(b)) \vert \leq e^{\nu(a-b)}\cdot e^{\frac{a}{2}} \ln\left(\frac{\cosh(b)+1}{\cosh(b)-1}\right).
\end{align}
\end{itemize}
\end{lemma}

\begin{proof}
From \cite[formula 8.715 1]{Gradshteyn} we know that for all $\mu \geq 0$ and $a>0$:
\begin{align*}
P_{\nu}^{-\mu}(\cosh(a)) = \frac{\sqrt{2}}{\sqrt{\pi}\sinh(a)^{\mu}\cdot \Gamma(\frac{1}{2} + \mu)} \int\limits_{0}^{a} \cosh\bigg(\left(\nu+\frac{1}{2}\right)x\bigg) \cdot (\cosh(a)-\cosh(x))^{\mu-\frac{1}{2}}\, dx,
\end{align*}
and therefore
\begin{align}
\label{equation:SecondEstimateProductLegendreFunctionsBeweis1}
\vert P_{\nu}^{-\mu}(\cosh(a)) \vert \leq \sqrt{\frac{2}{\pi}} \cdot \frac{e^{\left(\nu+\frac{1}{2}\right)a} (\cosh(a)-1)^{\mu}}{\sinh(a)^{\mu}\cdot \vert \Gamma(\frac{1}{2} + \mu)\vert} \int\limits_{0}^{a} \frac{1}{\sqrt{\cosh(a)-\cosh(x)}}\, dx. 
\end{align}
Since $\cosh(t)=2 \sinh^2\left(\frac{t}{2}\right) + 1$ for all $t\in\mathbb{R}$, the above integral can be estimated as follows:
\begin{align*}
\int\limits_{0}^{a} \frac{1}{\sqrt{\cosh(a)-\cosh(x)}}\, dx &= \frac{1}{\sqrt{2}}\int\limits_{0}^{a} \frac{1}{\sqrt{\sinh^2\left(\frac{a}{2}\right)-\sinh^2\left(\frac{x}{2}\right)}}\, dx \\
&\leq \frac{1}{\sqrt{2\sinh\left(\frac{a}{2}\right)}} \int\limits_{0}^{a} \frac{1}{\sqrt{\sinh\left(\frac{a}{2}\right)-\sinh\left(\frac{x}{2}\right)}}\, dx,
\end{align*}
and using $\sinh\left(\frac{a}{2}\right)-\sinh\left(\frac{x}{2}\right)\geq \frac{a}{2}-\frac{x}{2}$ for all $x\in (0,a)$ we obtain:
\begin{align}
\label{equation:SecondEstimateProductLegendreFunctionsBeweis2}
\int\limits_{0}^{a} \frac{1}{\sqrt{\cosh(a)-\cosh(x)}}\, dx \leq \frac{1}{\sqrt{\sinh\left(\frac{a}{2}\right)}} \int\limits_{0}^{a} \frac{1}{\sqrt{a-x}}\, dx = 2\sqrt{\frac{a}{\sinh\left(\frac{a}{2}\right)}}\leq 2  \sqrt{2}.
\end{align}
Because of \eqref{equation:SecondEstimateProductLegendreFunctionsBeweis1} and \eqref{equation:SecondEstimateProductLegendreFunctionsBeweis2} we have for all $\mu \geq 0$ and $a>0$:
\begin{align}
\label{equation:BeweisAbschaetzungLegendrePGrob}
\vert P_{\nu}^{-\mu}(\cosh(a)) \vert \leq \frac{3\cdot  e^{\left(\nu+\frac{1}{2}\right)a} }{ \vert \Gamma(\frac{1}{2} + \mu)\vert} \cdot \left(\frac{\cosh(a)-1}{\sinh(a)}\right)^{\mu}\leq \frac{3\cdot  e^{\left(\nu+\frac{1}{2}\right)a} }{ \vert \Gamma(\frac{1}{2} + \mu)\vert} \cdot \left(\sinh\left(\frac{a}{2}\right)\right)^{\mu},
\end{align}
where, for the last equality, we used $\frac{\cosh(a)-1}{\sinh(a)}=\frac{2\sinh^2(\frac{a}{2})}{\sinh(a)}<\sinh(\frac{a}{2})$.

 Further, from \eqref{equation:RepresentationLegendreQIntegral} (or see \cite[formula (5) on p. 155]{Erdelyi}), for all $\mu \geq 0$ and $b>0$:
\begin{align}
\label{equation:SecondEstimateProductLegendreFunctionsBeweis3}
\vert Q_{\nu}^{\mu}(\cosh(b))\vert = \frac{  \Gamma(\nu+1+\mu) }{\Gamma(\nu+1)2^{\nu+1}\sinh(b)^{\mu}}\int\limits_{0}^{\pi} \frac{\sin(t)^{2\nu+1}\cdot(\cosh(b)+\cos(t))^{\mu}}{(\cosh(b)+\cos(t))^{\nu+1}} \,dt.
\end{align}
$(i)$. We will estimate the integrand given above. Note that for any $z>1$ the function $(-1,1)\ni x\mapsto \frac{1-x^2}{z+x}\in\mathbb{R}$ has a global maximum at $x=-z+\sqrt{z^2-1}$, where the maximal value is $2z-2\sqrt{z^2-1}$. Hence, with $z:=\cosh(b)$ we have for all $t\in(0,\pi)$ 
\begin{align}
\label{equation:AbschaetzungUnterIntegralFein}
 \frac{\sin(t)^{2\nu+1}}{(\cosh(b)+\cos(t))^{\nu+1}} &= \left(\frac{1-\cos^2(t)}{\cosh(b)+\cos(t)}\right)^{\nu}\cdot \frac{\sin(t)}{\cosh(b)+\cos(t)} \leq \frac{2^{\nu}e^{-b\nu}\sin(t)}{(\cosh(b)-1)}.
\end{align}
Thus, we can estimate the integral on the right-hand side of \eqref{equation:SecondEstimateProductLegendreFunctionsBeweis3} as follows:
\begin{align*}
\int\limits_{0}^{\pi} \frac{\sin(t)^{2\nu+1}\cdot(\cosh(b)+\cos(t))^{\mu}}{(\cosh(b)+\cos(t))^{\nu+1}} \,dt \,&\leq \frac{2^{\nu}}{(\cosh(b)-1)}\int\limits_{0}^{\pi} \sin(t)(\cosh(b)+\cos(t))^{\mu} \,dt \\
&=\frac{2^{\nu}}{(\cosh(b)-1)}\cdot \frac{(\cosh(b)+1)^{\mu+1}-(\cosh(b)-1)^{\mu+1}}{\mu+1} \\
& \leq \frac{2^{\nu+1}  \cosh(b)^{\mu+1} 2^{\mu}}{(\cosh(b)-1)}.
\end{align*}
Using \eqref{equation:SecondEstimateProductLegendreFunctionsBeweis3} we have
\begin{align}
\label{equation:SecondEstimateProductLegendreFunctionsBeweis4}
\vert Q_{\nu}^{\mu}(\cosh(b))\vert \leq \frac{  \Gamma(\nu+1+\mu) }{\Gamma(\nu+1)}\cdot \frac{\cosh(b)}{\cosh(b)-1}\cdot \left(2\coth(b)\right)^{\mu}.
\end{align}

Let $\tilde{C}>0$ be some constant such that $\big\vert\frac{\Gamma(\nu+1+\mu)}{ \Gamma(\frac{1}{2} + \mu)}\big\vert \leq \tilde{C}\cdot \mu^{\nu+\frac{1}{2}}$ for all $\mu>0$, where such a constant exists because of \eqref{equation:GammaQuotAsymp}. If we set $C:=\tilde{C}\cdot \frac{3 \cosh(b)}{(\cosh(b)-1)\Gamma(\nu+1)}$ and $D:=2\coth(b)$, then the lemma follows from \eqref{equation:BeweisAbschaetzungLegendrePGrob} and \eqref{equation:SecondEstimateProductLegendreFunctionsBeweis4}.

$(ii)$. Similarly, for $\mu=0$ we obtain from equation \eqref{equation:SecondEstimateProductLegendreFunctionsBeweis3}:
\begin{align*}
\vert Q_{\nu}(\cosh(b))\vert &= \frac{1}{2^{\nu+1}}\int\limits_{0}^{\pi} \frac{\sin(t)^{2\nu+1}}{(\cosh(b)+\cos(t))^{\nu+1}} \,dt \\
&= \frac{1}{2^{\nu+1}}\int\limits_{0}^{\pi} \left( \frac{1-\cos^2(t)}{\cosh(b)+\cos(t)}\right)^{\nu} \frac{\sin(t)}{\cosh(b)+\cos(t)} \\
& \leq \frac{e^{-b \nu }}{2}\int\limits_{0}^{\pi} \frac{\sin(t)}{ \cosh(b)+\cos(t)} \,dt \\
& = \frac{e^{-b \nu }}{2}\ln\left(\frac{\cosh(b)+1}{\cosh(b)-1}\right).
\end{align*}
The statement follows from the above estimate together with \eqref{equation:BeweisAbschaetzungLegendrePGrob}.
\end{proof}

\section{Green's function of a wedge}

We will use the generalized Mehler-Fock transformations, which we obtained in the last section, to deduce a new formula for the Green's function of a wedge in the hyperbolic plane. 

Let $\mathbb{H}^2$ be the hyperbolic plane and let $\Delta$ denote the Laplace-Beltrami operator. For any domain $\Omega\subset\mathbb{H}^2$ let $\Delta_{\Omega}$ denote the Dirichlet Laplacian for $\Omega$ and let $K_{\Omega}:\Omega\times\Omega\times (0,\infty)\rightarrow\mathbb{R}$ denote the heat kernel corresponding to the Dirichlet Laplacian. The heat kernel can be defined as the minimal non-negative fundamental solution to the heat equation (see \cite[Section 2.1]{ErenDiss},  \cite{Grigoryan}).

We will use the following notation: For $\delta\in\mathbb{R}$
\begin{align}
\label{equation:HalfPlane}
\mathcal{H}_{>\delta}:=\{\,z\in\mathbb{C} \mid  \Re(z)>\delta \,\},
\end{align}
where $\Re(z)$ denotes the real part of $z$, and for $\Omega\subset\mathbb{H}^2$ let
\begin{align*}
\offdiag(\Omega):&=\{\, \left(x,y\right)\in\Omega\times\Omega \mid x\neq y \,\},
\end{align*} 
which we call the \emph{off-diagonal} of the cartesian product $\Omega\times\Omega$. Moreover, for a given continuous function $f:(0,\infty)\rightarrow \mathbb{R}$ we denote its Laplace transform as $\mathcal{L}\lbrace f \rbrace (s):=\int_{0}^{\infty} e^{-st}f(t) dt$, provided the Laplace integral exists for some $s\in\mathbb{C}$. The inverse Laplace transform is indicated by the symbol $\mathcal{L}^{-1}\lbrace \cdot \rbrace$.

\begin{definition}
\label{definition:Greens}
The Laplace transform of the heat kernel $K_{\Omega}$ is called the \emph{resolvent kernel} or \emph{Green's function} of $\Omega$. More precisely, for any domain $\Omega\subset\mathbb{H}^2$, the Green's function is defined as the function
\begin{align}
\label{equation:Greens}
\begin{split}
&G_{\Omega}: \offdiag(\Omega) \times \mathcal{H}_{>0}\rightarrow \mathbb{C}, \\
&\left((x,y),s\right)\mapsto G_{\Omega}(x,y;s):=\int\limits_{0}^{\infty} e^{-st}K_{\Omega}(x,y;t) dt.
\end{split}
\end{align}
\end{definition}

\begin{proposition}
\label{proposition:PDEGreen}
\emph{(\cite[Exercise 9.10]{Grigoryan})} For any $s>0$, the Green's function $G_{\Omega}\left( \cdot,\cdot;s \right)$ is a non-negative and smooth function on $\offdiag(\Omega)$. Furthermore, for any $y\in\Omega$ and any $s > 0$ the function $u:\Omega\backslash \{y\}\ni x\mapsto G_{\Omega}(x,y;s)\in\mathbb{R}$ belongs to $L^1(\Omega)$ and
\begin{align}
\label{equation:PDEGreen}
\int\limits_{\Omega} u(x)\cdot \left( s + \Delta \right)f(x)\, dx &=f(y)\, \text{ for all } f\in C_c^{\infty}(\Omega).
\end{align}
In particular, $\left( s + \Delta \right)u(x)=0$ for all $x\in\Omega\backslash \{y\}$.
\end{proposition}

It will be useful to introduce the following shifted functions.

\begin{definition}
\label{definition:ShiftedFunctions1}
For any domain $\Omega\subset \mathbb{H}^2$, we call the functions $K_{\Omega}^{\nicefrac{1}{4}}:\Omega\times\Omega\times (0,\infty)\rightarrow\mathbb{R}$ and $G_{\Omega}^{\nicefrac{1}{4}}:\offdiag(\Omega)\times \mathcal{H}_{>\frac{1}{4}}\rightarrow\mathbb{C}$ defined by
\begin{align}
K_{\Omega}^{\nicefrac{1}{4}}(x,y;t):&=e^{\frac{1}{4}t}\cdot K_{\Omega}(x,y;t),\label{equation:ShiftedHeat}\\
G_{\Omega}^{\nicefrac{1}{4}}(x,y;s):&=\mathcal{L}\lbrace K_{\Omega}^{\nicefrac{1}{4}}(x,y;t)\rbrace (s) \label{equation:ShiftedGreen}
\end{align}
the \emph{shifted} heat kernel and the \emph{shifted} Green's function of $\Omega$, respectively.
\end{definition}

For each $s>\frac{1}{4}$, the shifted Green's function  $G_{\Omega}^{\nicefrac{1}{4}}$ satisfies Proposition \ref{proposition:PDEGreen} if the Laplacian is replaced by the \emph{shifted} Laplacian $\Delta^{\nicefrac{1}{4}}:=\Delta-\frac{1}{4}$. The reason is that  $s+\Delta^{\nicefrac{1}{4}}=\left(s-\frac{1}{4}\right)+\Delta$ and
\begin{align}
G_{\Omega}^{\nicefrac{1}{4}}(x,y;s) &= \int\limits_{0}^{\infty} e^{-st }e^{\frac{1}{4}t}\cdot K_{\Omega}(x,y;t) dt= \int\limits_{0}^{\infty} e^{-\left(s-\frac{1}{4}\right)t }\cdot K_{\Omega}(x,y;t) dt \nonumber \\
& = G_{\Omega}\left(x,y;s-\frac{1}{4}\right). \label{equation:ConnectionShiftGreen}
\end{align}
Most of the formulas below are stated in terms of the shifted functions, since then they become shorter. The formulas can always be rewritten into corresponding formulas for the Green's function by \eqref{equation:ConnectionShiftGreen} and for the heat kernel using \eqref{equation:ShiftedHeat}.

The heat kernel $K_{\mathbb{H}^2}$ for the hyperbolic plane is given by the following formulas (see, e.g., \cite[(12) and (13) on p. 246]{Chavel}):
\begin{align}
K_{\mathbb{H}^2}(x,y;t)&= \frac{1}{2\pi}\int\limits_{0}^{\infty} e^{-\left( \frac{1}{4} + \rho^2 \right)t}P_{-\frac{1}{2}+i\rho}\left( \cosh(d(x,y)) \right)\rho\tanh\left( \pi\rho \right)d\rho \label{equation:HyperHeat1}\\
&=\frac{\sqrt{2}}{(4\pi t)^{\nicefrac{3}{2}}}e^{-\frac{t}{4}}\int\limits_{d(x,y)}^{\infty} \frac{\rho e^{-\frac{\rho^2}{4t}}}{\sqrt{\cosh(\rho)-\cosh(d(x,y))}}d\rho, \label{equation:HyperHeat2}
\end{align}
where $d(x,y)$ denotes the hyperbolic distance of $x,y\in\mathbb{H}^2$ and $P_{-\frac{1}{2}+i\rho}$ is the Legendre function of the first kind from Definition \ref{definition:LegendreFirstSecond}.
We start with a helpful new formula for the Green's function $G_{\mathbb{H}^2}$ of the hyperbolic plane in terms of polar coordinates.

\begin{proposition}
\label{proposition:GreenPlane}
Let $s\in\mathbb{C}$ be such that $\Re(s) > \frac{1}{4}$, and let $x,y\in\mathbb{H}^2$ with $x\neq y$ be given. The shifted Green's function for the hyperbolic plane is given by the equation
\begin{align}
G_{\mathbb{H}^2}^{\nicefrac{1}{4}}(x,y;s)&=\frac{1}{2\pi}Q_{\sqrt{s}-\frac{1}{2}}\left( \cosh(d(x,y)) \right). \label{equation:GreenPlane1}
\end{align}
Suppose, further, that we have chosen polar coordinates in $\mathbb{H}^2$ \emph{(}associated with a chosen orientation and a chosen geodesic ray\emph{)} such that $x=(a,\alpha)$ and $y=(b,\beta)$, where $a,b\in(0,\infty)$ and $\alpha,\beta\in [0,2\pi)$. Then
\begin{align}
G_{\mathbb{H}^2}^{\nicefrac{1}{4}}(x,y;s)=\frac{1}{\pi^2} \int\limits_{0}^{\infty} Q_{\sqrt{s}-\frac{1}{2}}^{-i\rho}\left(\cosh(a)\right)Q_{\sqrt{s}-\frac{1}{2}}^{i\rho}\left(\cosh(b)\right) \cosh\left(\rho\cdot \left( \pi-\vert \alpha-\beta \vert \right)\right)d\rho. \label{equation:GreenPlane2}
\end{align}
\end{proposition}

\begin{proof}
From \eqref{equation:HyperHeat1} we obtain
\begin{align}
\label{equation:IntegralformelShiftedHeatPlane}
K_{\mathbb{H}^2}^{\nicefrac{1}{4}}(x,y;t)=\frac{1}{2\pi}\int\limits_{0}^{\infty} e^{-  \rho^2 t}P_{-\frac{1}{2}+i\rho}\left( \cosh(d(x,y)) \right)\rho\tanh\left( \pi\rho \right)d\rho.
\end{align}
Thus using the Fubini-Tonelli theorem we get
\begin{align*}
G_{\mathbb{H}^2}^{\nicefrac{1}{4}}(x,y;s)&=\int\limits_{0}^{\infty} e^{-st}\left(\frac{1}{2\pi}\int\limits_{0}^{\infty} e^{-  \rho^2 t }P_{-\frac{1}{2}+i\rho}\left( \cosh(d(x,y)) \right)\rho\tanh\left( \pi\rho \right)d\rho \right) dt\\
&= \frac{1}{2\pi}\int\limits_{0}^{\infty} \left( \int\limits_{0}^{\infty}  e^{-\left( s+  \rho^2\right) t } dt\right) P_{-\frac{1}{2}+i\rho}\left( \cosh(d(x,y)) \right)\rho\tanh\left( \pi\rho \right)d\rho \\
&= \frac{1}{2\pi}\int\limits_{0}^{\infty} \frac{\rho\tanh\left( \pi\rho \right)}{s+\rho^2} P_{-\frac{1}{2}+i\rho}\left( \cosh(d(x,y)) \right)d\rho \\
&= \frac{1}{2\pi} Q_{\sqrt{s}-\frac{1}{2}}\left( \cosh(d(x,y)) \right),
\end{align*}
where the last equality can be found for example in \cite[7.213]{Gradshteyn} or \cite[p. 20]{Oberhettinger}. Let us explain shortly why the Fubini-Tonelli theorem is applicable, which we used in the second equality above. Because of \eqref{equation:LegendreAsymp2} and since $\rho\mapsto P_{-\frac{1}{2}+i\rho}\left( \cosh(d(x,y))\right)$ is continuous, there exists some constant $C>0$ such that for all $\rho\in (0,\infty)$: $\vert P_{-\frac{1}{2}+i\rho}\left( \cosh(d(x,y))\right) \vert\leq \frac{C}{\sqrt{\rho}}$. Hence,
\begin{align*}
\int\limits_{0}^{\infty}\int\limits_{0}^{\infty} \vert e^{-\left( s+  \rho^2\right) t }P_{-\frac{1}{2}+i\rho}\left( \cosh(d(x,y))\right) \rho\tanh\left( \pi\rho \right) \vert\, dt d\rho &\leq C  \int\limits_{0}^{\infty}\int\limits_{0}^{\infty} e^{-\left( \frac{1}{4} +  \rho^2\right)t} \cdot \sqrt{\rho}\, dt \, d\rho \\
&=C\int\limits_{0}^{\infty} \frac{\sqrt{\rho}}{\frac{1}{4}+\rho^2}\, d\rho<\infty.
\end{align*}
This completes the proof of equation \eqref{equation:GreenPlane1}.

The second formula \eqref{equation:GreenPlane2} follows from \eqref{equation:GreenPlane1} and Corollary \ref{corollary:IntegralProduktLegendreFuerGreensFunction} with $z:=\cosh(a)$, $\omega:=\cosh(b)$, and $\theta:=\vert \alpha-\beta \vert$, since the distance is given in the above polar coordinates by the following formula (see \cite[Theorem 2.2.1 (i)]{Buser}):
\begin{align}
\label{equation:HyperbolicDistancePolarCoordinates}
\cosh(d(x,y))=\cosh(a)\cosh(b)-\sinh(a)\sinh(b)\cos(\vert \alpha-\beta \vert).
\end{align}
Note that the assumptions of Corollary \ref{corollary:IntegralProduktLegendreFuerGreensFunction} are satisfied here because $x\neq y$ implies that either $a\neq b$ or $\vert \alpha-\beta \vert\in (0,2\pi)$.
\end{proof} 

\begin{remark}
Note that the function on the right-hand side of \eqref{equation:GreenPlane1} is defined for all $s\in\mathbb{C}\backslash (-\infty,0]$ and is also holomorphic on that domain. Hence, the function $\mathbb{C}\backslash (-\infty,0] \ni s\mapsto \frac{1}{2\pi}Q_{\sqrt{s}-\frac{1}{2}}\left( \cosh(d(x,y)) \right)\in\mathbb{C}$ is an analytic continuation of $\mathcal{H}_{>\frac{1}{4}}\ni s\mapsto G_{\mathbb{H}^2}^{\nicefrac{1}{4}}(x,y;s)\in\mathbb{C}$. Obviously, the same applies to the right-hand side of \eqref{equation:GreenPlane2}.
\end{remark}

\begin{definition}
\label{definition:Wedge}
A domain $W\subset \mathbb{H}^2$ is called a \emph{hyperbolic wedge} if there exist some $\gamma\in (0,2\pi]$ and geodesic polar coordinates $(a,\alpha)$ with respect to some base point $P\in\mathbb{H}^2$ (where the angle $\alpha$ is measured with respect to a chosen geodesic ray emanating from $P$ and a chosen orientation), such that $W$ is parametrised as
\begin{align*}
W=\{\, (a,\alpha) \mid 0<a<\infty,\, 0<\alpha <\gamma \,\}.
\end{align*}
The point $P$ is called the \emph{vertex} and $\gamma$ is called the \emph{angle} of the wedge. Thus, loosely speaking, a wedge is any domain bounded by two geodesic rays emanating from one point, its vertex. 

\end{definition}

We will now focus on the heat kernel $K_W$ and the Green's function $G_W$ of a hyperbolic wedge. Before we give a rigorous treatment of those functions, we will consider a heuristic derivation of the formula \eqref{equation:GreenWedge} for $G_W^{\nicefrac{1}{4}}$ given below. We want to point out that our heuristic derivation is analogous to a derivation given in \cite{Srisatkunarajah}\footnote{I am grateful to M. van den Berg for sending me the relevant pages of that work.}. In his doctoral thesis, Srisatkunarajah has worked out a proof of a formula due to D.B. Ray for the Green's function of a Euclidean wedge (this formula is published in \cite[p. 44]{McKean} and \cite[formula (2.5)]{VanDenBerg}). However, the proof given in \cite{Srisatkunarajah} is incomplete and has some argumentative gaps, which we will provide in the hyperbolic case by Lemma \ref{lemma:PropertiesOfH} and Theorem \ref{theorem:GreenWedge}. Note that our formula for the Green's function of a hyperbolic wedge is similar to Ray's formula. The only difference is that the Bessel functions appearing in Ray's formula are replaced properly with the associated Legendre functions of the second kind. However, in the Euclidean case all integrals appearing in the derivation and involving Bessel functions are well-known, such that their solutions could be used in \cite{Srisatkunarajah}. We, instead, will have to refer to Section \ref{section:Legendre} to solve the corresponding integrals for the associated Legendre functions of the second kind.

As mentioned, we proceed with a largely heuristic discussion of the Green's function for $W$ in order to motivate the formula \eqref{equation:GreenWedge} below. Recall that $\Delta=-\diverg \circ \nabla$ denotes the Laplacian. Suppose the heat kernel of the wedge $W$ can be written as
\begin{align}
\label{equation:AnsatzW-Kern}
K_W(x,y;t) = K_{\mathbb{H}^2}(x,y;t)-h(x,y;t),
\end{align}
where $h:\overline{W}\times W\times[0,\infty)\rightarrow\mathbb{R}$ is a function satisfying for all $y\in W$, the following conditions:
\begin{align}
\label{equation:PDE1}
\begin{cases}
h\left( \cdot,y;\cdot \right):\overline{W}\times [0,\infty)\rightarrow \mathbb{R}& \text{ is continuous, } \\
\left( \partial_t + \Delta \right)h(x,y;t) = 0& \text{ for all }x\in W \text{ and } t>0, \\
h(x,y;0)=0&  \text{ for all }x\in W, \\
h(x,y;t) = K_{\mathbb{H}^2}(x,y;t)& \text{ for all }x\in \partial W \text{ and } t>0.
\end{cases}
\end{align}

When we multiply both sides of equation \eqref{equation:AnsatzW-Kern} with $e^{\frac{t}{4}}$, we obtain the corresponding equation for the shifted functions. With $h^{\nicefrac{1}{4}}(x,y;t):=e^{\frac{t}{4}}h(x,y;t)$, we have
\begin{align}
\label{equation:AnsatzW-KernShift}
K_{W}^{\nicefrac{1}{4}}(x,y;t) = K_{\mathbb{H}^2}^{\nicefrac{1}{4}}(x,y;t)-h^{\nicefrac{1}{4}}(x,y;t)\quad \text{ for all } t>0\,\text{ and } x,y\in W.
\end{align}
When we apply the Laplace transform on both sides of \eqref{equation:AnsatzW-KernShift}, we obtain the equation
\begin{align}
\label{equation:AnsatzGreenShift}
G_W^{\nicefrac{1}{4}}(x,y;s) = G_{\mathbb{H}^2}^{\nicefrac{1}{4}}(x,y;s) - H^{\nicefrac{1}{4}}(x,y;s)
\end{align}
for all $(x,y)\in \offdiag(W)$ and $s\in\mathbb{C}$ with $\Re(s)>\frac{1}{4}$, where
\begin{align}
\label{equation:GreenShifted}
H^{\nicefrac{1}{4}}(x,y;s):=\mathcal{L}\lbrace h^{\nicefrac{1}{4}}(x,y;t) \rbrace (s).
 \end{align}
For any $s>\frac{1}{4}$ and $y\in W$ the function $H^{\nicefrac{1}{4}}(\cdot ,y;s)$ solves the following boundary value problem: 
\begin{align}
\label{equation:PDEShifted-H}
\begin{cases}
H^{\nicefrac{1}{4}}(\cdot,y;s): \overline{W}\rightarrow \mathbb{R} &\text{ is continuous and bounded,}\\
\left( s + \Delta-\frac{1}{4} \right)H^{\nicefrac{1}{4}}\left( x,y;s \right) = 0 &\text{ for all } x\in W,\\
H^{\nicefrac{1}{4}}\left( x,y;s \right) = G_{\mathbb{H}^2}^{\nicefrac{1}{4}}\left( x,y;s \right) & \text{ for all } x\in \partial W.
\end{cases}
\end{align}

We want to find a solution to the problem \eqref{equation:PDEShifted-H} in terms of polar coordinates, so let us choose for the rest of this section polar coordinates such that
\begin{align*}
W=\{\, (a,\alpha) \mid 0< a <\infty, \, 0<\alpha < \gamma \,\}.
\end{align*}
The Laplacian is given in polar coordinates by the formula
\begin{align*}
-\Delta = \frac{1}{\sinh(a)}\frac{\partial}{\partial a}\left\lbrace \sinh(a)\frac{\partial}{\partial a} \right\rbrace + \frac{1}{\sinh^2(a)}\frac{\partial^2}{\partial^2 \alpha},
\end{align*}
which follows from the well-known representation of the Laplacian in local coordinates (see e.g. \cite[formula $(3.40)$]{Grigoryan}) and since the Riemannian metric $g$ of the hyperbolic plane is represented in polar coordinates by $g=da^2 + \sinh^2(a)d\alpha^2$ (see \cite[formula $(3.70)$]{Grigoryan}). In the following, let $y=(b,\beta)\in W$ and $s>\frac{1}{4}$ both be fixed.

Firstly, we use the method of separation of variables to find product solutions to the partial differential equation (PDE) in \eqref{equation:PDEShifted-H}, i.e. we look for solutions of the form $u(a,\alpha)=v(a)\cdot w(\alpha)$. Substituting into the PDE, we obtain for all points $x=(a,\alpha)$ such that $u(a,\alpha)\neq 0:$
\begin{align*}
&\left( s + \Delta -\frac{1}{4}\right) u(a,\alpha) = 0  \\
\Leftrightarrow &\left( s-\frac{1}{4} \right)v(a)w(\alpha) - \frac{w(\alpha)}{\sinh(a)}\frac{\partial}{\partial a}\left\lbrace \sinh(a)v'(a) \right\rbrace  - \frac{1}{\sinh^2(a)} w''(\alpha)v(a)  = 0 \\
\Leftrightarrow &\left( s-\frac{1}{4} \right)\sinh^2(a) - \frac{\sinh(a)}{v(a)}\frac{\partial}{\partial a}\left\lbrace \sinh(a)v'(a) \right\rbrace = \frac{w''(\alpha)}{w(\alpha)}.
\end{align*} 

Therefore, both sides must be equal to some separation constant $c\in\mathbb{R}$. We assume $c$ to be non-negative and write $c=\rho^2$ with $\rho\geq 0$. The PDE now reduces to the following two ordinary differential equations:

\begin{itemize}
\item[I)] $w''(\alpha)=\rho^2 \cdot w(\alpha)$,
\item[II)] $\left( -s+\frac{1}{4} \right)\sinh^2(a)\cdot v(a) + \sinh(a) \frac{\partial}{\partial a}\left\lbrace \sinh(a)v'(a) \right\rbrace + \rho^2\cdot v(a) = 0.$
\end{itemize}

We first solve the second equation. When we multiply this equation by $\frac{1}{\sinh^2(a)}$, it can be written equivalently as:
\begin{align*}
v''(a) + \frac{\cosh(a)}{\sinh(a)}v'(a) + \left( \left( -s+\frac{1}{4} \right) - \frac{\left(i\rho\right)^2}{\sinh^2(a)} \right)\cdot v(a) = 0.
\end{align*}
Consider $\tilde{v}:=v\circ \arcosh :(1,\infty)\rightarrow\mathbb{C}$. When we substite $\tilde{v}$ into the above equation, we get the following equivalent differential equation:
\begin{align*}
(1-z^2)\tilde{v}''(z) -2z \tilde{v}'(z) + \left( \left( s - \frac{1}{4} \right) - \frac{\left(i\rho\right)^2}{1-z^2} \right)\cdot \tilde{v}(z) = 0\, \text{ for all } z\in (1,\infty).
\end{align*}
If $\nu:=\sqrt{s} - \frac{1}{2} $, then $\nu\cdot(\nu+1) = s-\frac{1}{4}$. This differential equation for $\tilde{v}$ is nothing else than the associated Legendre equation stated in \eqref{equation:LegendreDGL} with parameters $\nu=\sqrt{s}-\frac{1}{2}$ and $\mu=i\rho$. For all $\rho\geq 0$ two linearly independent solutions are (see \cite[\S 14.2(iii)]{NIST})
\begin{align*}
z\mapsto P_{\sqrt{s} - \frac{1}{2}}^{i\rho}(z)\,\text{ and }\, z\mapsto Q_{\sqrt{s} - \frac{1}{2} }^{-i\rho}(z).
\end{align*}
On the other hand, two linearly independent solutions for equation I) are the functions $\alpha\mapsto\cosh(\rho \alpha)$ and $\alpha\mapsto \sinh(\rho \alpha)$. Thus we have the following product solutions to the PDE:
\begin{align*}
x=(a,\alpha)\mapsto \left( \tilde{A}_1\cosh(\rho \alpha)+ \tilde{A}_2\sinh(\rho \alpha) \right)\cdot \left( \tilde{B}_1 P_{\sqrt{s} - \frac{1}{2}}^{i\rho}(\cosh(a)) + \tilde{B}_2 Q_{\sqrt{s} - \frac{1}{2}}^{-i\rho}(\cosh(a)) \right)
\end{align*}
with arbitrary constants $\tilde{A}_1, \tilde{A}_2, \tilde{B}_1, \tilde{B}_2\in\mathbb{C}$.

Secondly, we construct from the product solutions above a suitable solution $u=H^{\nicefrac{1}{4}}(\cdot,y;s):\overline{W}\rightarrow\mathbb{R}$ to the PDE which also meets the boundary condition in \eqref{equation:PDEShifted-H}. Since 
\begin{align*}
\lim\limits_{a\rightarrow\infty} \vert P_{\sqrt{s} - \frac{1}{2}}^{i\rho}(\cosh(a)) \vert =\infty
\end{align*}
(see e.g. \cite[14.8.12]{NIST}), we set $\tilde{B}_1=0$. Further, because of the formula \eqref{equation:GreenPlane2} for the boundary condition we try, by using the superposition principle, 
\begin{align*}
u(a,\alpha):=\int\limits_{0}^{\infty}\left( A_1(\rho)\cosh(\rho \alpha)+ A_2(\rho)\sinh(\rho \alpha) \right) Q_{\sqrt{s} - \frac{1}{2}}^{-i\rho}(\cosh(a)) d\rho.
\end{align*}
We determine suitable functions $A_1(\rho),\, A_2(\rho)$ from the boundary conditions. The boundary $\partial W$ of the wedge contains the two rays emanating from $P$, which are described in polar coordinates by
\begin{align*}
 \{\, x=(a,\alpha)\mid 0<a<\infty\, ;\, \alpha=0\, \text{ or }\, \alpha=\gamma \,\}.
\end{align*}
For $\alpha=0$ we have
\begin{align*}
u(a,0) &= \int\limits_{0}^{\infty}  A_1(\rho)Q_{ \sqrt{s} - \frac{1}{2}}^{-i\rho}(\cosh(a)) d\rho,\\
G_{\mathbb{H}^2}^{\nicefrac{1}{4}}((a,0),(b,\beta);s) &= \frac{1}{\pi^2}\int\limits_{0}^{\infty}\cosh(\rho(\pi-\beta))Q_{\sqrt{s} - \frac{1}{2}}^{-i\rho}(\cosh(a))Q_{ \sqrt{s} - \frac{1}{2}}^{i\rho}(\cosh(b)) d\rho.
\end{align*}
Therefore we set
\begin{align*}
A_1(\rho):=\frac{1}{\pi^2}\cosh(\rho(\pi-\beta))Q_{\sqrt{s} - \frac{1}{2} }^{i\rho}(\cosh(b)),
\end{align*}
such that the boundary condition at $\alpha=0$ is satisfied.

Moreover,
\begin{align*}
u(a,\gamma) &= \int\limits_{0}^{\infty} \lbrace A_1(\rho)\cosh(\rho\gamma) + A_2(\rho)\sinh(\rho\gamma)\rbrace Q_{ \sqrt{s} - \frac{1}{2}}^{-i\rho}(\cosh(a)) d\rho,\\
G_{\mathbb{H}^2}^{\nicefrac{1}{4}}((a,\gamma),(b,\beta);s) &= \frac{1}{\pi^2}\int\limits_{0}^{\infty}\cosh(\rho(\pi-\vert \gamma - \beta \vert))Q_{\sqrt{s} - \frac{1}{2}}^{-i\rho}(\cosh(a))Q_{ \sqrt{s} - \frac{1}{2}}^{i\rho}(\cosh(b)) d\rho.
\end{align*}
Obviously $\vert \gamma - \beta \vert=\gamma-\beta$ and thus we set:
\begin{align*}
A_2(\rho):=\frac{1}{\pi^2}Q_{\sqrt{s} - \frac{1}{2} }^{i\rho}(\cosh(b))\left(\frac{\cosh(\rho(\pi-(\gamma-\beta))) - \cosh(\rho(\pi-\beta))\cosh(\rho\gamma)}{\sinh(\rho\gamma)}\right).
\end{align*}
Finally we obtain the following candidate for a solution of \eqref{equation:PDEShifted-H}, defined for all $(a,\alpha)\in \overline{W}\backslash\{ P \}$ as:
\begin{align*}
u(a,\alpha) =&\, \frac{1}{\pi^2} \int\limits_{0}^{\infty} Q_{ \sqrt{s} - \frac{1}{2}}^{-i\rho}(\cosh(a))  Q_{ \sqrt{s} - \frac{1}{2}}^{i\rho}(\cosh(b))\Bigg( \cosh(\rho(\pi-\beta))\cosh(\rho\alpha) \\
& + \left( \cosh(\rho(\pi-(\gamma-\beta))) - \cosh(\rho(\pi-\beta))\cosh(\rho\gamma) \right)\frac{\sinh(\rho\alpha)}{\sinh(\rho\gamma)} \Bigg) d\rho.
\end{align*}
Now observe that the long expression in brackets above can be simplified, since
\begin{align}
&\cosh(\rho(\pi-\beta))\cosh(\rho\alpha) + \left( \cosh(\rho(\pi-(\gamma-\beta))) - \cosh(\rho(\pi-\beta))\cosh(\rho \gamma) \right)\frac{\sinh(\rho\alpha)}{\sinh(\rho\gamma)} \nonumber \\
&=\frac{\sinh(\pi\rho)}{\sinh(\gamma\rho)}\cosh(\rho(\gamma-\alpha-\beta)) - \frac{\sinh((\pi-\gamma)\rho)}{\sinh(\gamma\rho)}\cosh((\alpha-\beta)\rho). \label{equation:IdentitaetCosHyperbolicusMonster}
\end{align}
Thus we set for all $x=(a,\alpha)\in \overline{W}\backslash\{ P \}$:
\begin{align}
\label{equation:SolutionH}
H^{\nicefrac{1}{4}} (x,y;s) := & \, \frac{1}{\pi^2} \int\limits_{0}^{\infty} Q_{ \sqrt{s} - \frac{1}{2}}^{-i\rho}(\cosh(a))  Q_{ \sqrt{s} - \frac{1}{2}}^{i\rho}(\cosh(b)) \Bigg( \frac{\sinh(\pi\rho)}{\sinh(\gamma\rho)}\cosh(\rho(\gamma-\alpha-\beta)) \nonumber \\
&- \frac{\sinh((\pi-\gamma)\rho)}{\sinh(\gamma\rho)}\cosh((\alpha-\beta)\rho) \Bigg)d\rho.
\end{align}
Now, a candidate for a formula (in polar coordinates) of the shifted Green's function $G_W^{\nicefrac{1}{4}}$ can be deduced using \eqref{equation:AnsatzGreenShift}, \eqref{equation:GreenPlane2} and \eqref{equation:SolutionH}. That formula is stated explicitly in Theorem \ref{theorem:GreenWedge} below. In order to give a rigorous proof we first study the function $H^{\nicefrac{1}{4}}$, which is done in the following lemma. 

\begin{lemma}
\label{lemma:PropertiesOfH}
Consider the function 
\begin{align*}
H^{\nicefrac{1}{4}}:W\times W\times\mathbb{C}\backslash(-\infty,0]\ni (x,y, s)\mapsto H^{\nicefrac{1}{4}}(x,y;s)\in \mathbb{C},
\end{align*}
where $H^{\nicefrac{1}{4}}(x,y;s)$ is defined by the right-hand side of \eqref{equation:SolutionH} with $x=(a,\alpha)$ and $y=(b,\beta)$ \emph{(}with respect to the polar coordinates chosen above\emph{)}. Then:
\begin{itemize}
\item[$(i)$] For all $x,y\in W$ the function $\mathbb{C}\backslash(-\infty,0]\ni s\mapsto H^{\nicefrac{1}{4}}(x,y;s)\in \mathbb{C}$ is holomorphic.
\item[$(ii)$] For all $s\in\mathbb{C}$ with $\Re(s)>0$ the function
\begin{align*}
W\times W \ni (x,y)\mapsto H^{\nicefrac{1}{4}}(x,y;s)\in \mathbb{R}
\end{align*}
is continuous. Moreover, $H^{\nicefrac{1}{4}}(x,y;s)$ is real valued for all $s>0$ and $x,y\in W$. Further, for all $s>\frac{1}{4}$ and $y\in W$ the function $W\ni x\mapsto H^{\nicefrac{1}{4}}(x,y;s)\in\mathbb{R}$ is smooth and $(s+\Delta^{\nicefrac{1}{4}})H^{\nicefrac{1}{4}}(\cdot,y;s) \equiv 0$ \emph{(}recall that $\Delta^{\nicefrac{1}{4}}=\Delta-\frac{1}{4}$\emph{)}.
\item[$(iii)$] For all $s>\frac{1}{4}$ and $y\in W$ the function $W\ni x\mapsto H^{\nicefrac{1}{4}}(x,y;s)\in \mathbb{R}$ can be extended continuously to $\overline{W}$. That extension \emph{(}which we also denote by $H^{\nicefrac{1}{4}}$\emph{)} satisfies $H^{\nicefrac{1}{4}}(x,y;s)=G_{\mathbb{H}^2}^{\nicefrac{1}{4}}(x,y;s)$ for all $x\in \partial W$. Furthermore, we have $0 < H^{\nicefrac{1}{4}}(x,y;s)$ for all $x\in W$ and $H^{\nicefrac{1}{4}}(x,y;s)\leq G_{\mathbb{H}^2}^{\nicefrac{1}{4}}(x,y;s)$ for all $x\in W\backslash\{y\}$.
\end{itemize}
\end{lemma} 

\begin{proof}
$(i)$. Let $x,y\in W$ be arbitrary. Obviously, the function $H^{\nicefrac{1}{4}}(x,y;\cdot)$ can be written as a composition of holomorphic functions due to Lemma \ref{lemma:HolomorphieIntegralLegendreProdukt}. Note that the function $ s\mapsto \sqrt{s}-\frac{1}{2}$ is holomorphic on the domain $\mathbb{C}\backslash(-\infty,0]$, where that domain is mapped onto $\mathcal{H}_{>-\frac{1}{2}}$. Therefore, Lemma \ref{lemma:HolomorphieIntegralLegendreProdukt} can indeed be applied here; the other conditions of Lemma \ref{lemma:HolomorphieIntegralLegendreProdukt} are obviously satisfied with $h(\rho) := \frac{1}{\pi^2} \cdot \big( \frac{\sinh(\pi\rho)}{\sinh(\gamma\rho)}\cosh(\rho(\gamma-\alpha-\beta))- \frac{\sinh((\pi-\gamma)\rho)}{\sinh(\gamma\rho)}\cosh((\alpha-\beta)\rho) \big)$.

$(ii)$. Suppose $s\in\mathbb{C}$ with $\Re(s)>0$ is given. We write $H^{\nicefrac{1}{4}}(\cdot,\cdot ;s)$ as
\begin{align*}
H^{\nicefrac{1}{4}}(x,y ;s) = \int\limits_{0}^{\infty} \psi(x,y,\rho) d\rho \,\, \text{ for all } x,y\in W
\end{align*}
with
\begin{align*}
\psi(x,y,\rho):=  \frac{1}{\pi^2} \, Q_{ \sqrt{s} - \frac{1}{2}}^{-i\rho}(\cosh(a))  Q_{ \sqrt{s} - \frac{1}{2}}^{i\rho}(\cosh(b)) &\Bigg( \frac{\sinh(\pi\rho)}{\sinh(\gamma\rho)}\cosh(\rho(\gamma-\alpha-\beta)) \\
&- \frac{\sinh((\pi-\gamma)\rho)}{\sinh(\gamma\rho)}\cosh((\alpha-\beta)\rho) \Bigg)
\end{align*}
for all $x,y\in W$ and $\rho>0$. We also set $\psi(x,y,0):=\lim_{\rho\searrow 0}\psi(x,y,\rho)$ for all $x,y\in W$, where the limit obviously exists.

First of all, note that for all $x,y\in W$ the function $\rho\mapsto \psi(x,y,\rho)$ is absolutely integrable over $[0,\infty)$ (see Theorem \ref{theorem:TheoremZentralIntegral}). Thus, the functions
\begin{align*}
q: W\times W \ni (x,y)\mapsto \int\limits_{0}^{\infty} \vert \psi(x,y,\rho) \vert \, d\rho \in [0,\infty)
\end{align*}
and $H^{\nicefrac{1}{4}}(\cdot,\cdot;s)$ are continuous, which is an immediate consequence of Lebesgue's dominated convergence theorem and \eqref{equation:EstimateProductLegendreFunctions}.

Suppose now $s>0$. Note that $Q_{ \sqrt{s} - \frac{1}{2}}^{-i\rho}(\cosh(a))  Q_{ \sqrt{s} - \frac{1}{2}}^{i\rho}(\cosh(b))$ is real valued for all $a,b>0$ and $\rho\geq 0$. This can be seen from \eqref{equation:ConditionsSchritt1.1} and \cite[\S 14.20]{NIST} (see the comment preceding formula $14.20.6$ of \cite{NIST}). Thus, $\psi(x,y,\rho)$ is real valued for all $x,y\in W$ and $\rho\geq 0$ and, consequently, $H^{\nicefrac{1}{4}}(x,y ;s)$ is also real valued for all $x,y\in W$. Also note that for all $\rho\geq0$ the function $W\times W \ni (x,y)\mapsto \psi(x,y,\rho)\in\mathbb{R}$ is smooth. In the following, we deal with the other (less obvious) properties of $H^{\nicefrac{1}{4}}$ stated in $(ii)$.

We suppose now $s>\frac{1}{4}$. First, we show that for all $x\in W$ the function $H^{\nicefrac{1}{4}}(x,\cdot ;s)$ is smooth, which can be seen as follows: For all $x\in W$ the function $W \ni y \mapsto H^{\nicefrac{1}{4}}(x,y;s)\in\mathbb{R}$ belongs to $L_{\text{loc}}^2(W)$ as any continuous function, this means (by definition of the space $L_{\text{loc}}^2(W)$, see \cite[p. 98]{Grigoryan}), $H^{\nicefrac{1}{4}}(x,\cdot;s)\in L^2(U)$ for any relatively compact open set $U\subset W$. Moreover, for all $f\in C_{c}^{\infty}(W)$ and $x\in W$
\begin{align}
\int\limits_W H^{\nicefrac{1}{4}}(x,y;s) \cdot (sf+\Delta^{\nicefrac{1}{4}} f)(y) \,dy &= \int\limits_W \left( \int\limits_{0}^{\infty} \psi(x,y,\rho) d\rho \right) \cdot (sf+\Delta^{\nicefrac{1}{4}} f) (y) \, dy \nonumber \\
& =  \int\limits_{0}^{\infty}  \int\limits_W \psi(x,y,\rho) \cdot (sf+\Delta^{\nicefrac{1}{4}} f) (y) \, dy  \,    d\rho \nonumber \\
& = \int\limits_{0}^{\infty}  \int\limits_W  \underbrace{(s\psi(x,\cdot,\rho) + \Delta^{\nicefrac{1}{4}} \psi(x,\cdot,\rho))}_{\equiv 0}(y) \cdot f(y) \, dy  \,    d\rho \nonumber \\
&= 0. \label{equation:HSmoothFunction}
\end{align}
Note that we used the Fubini-Tonelli theorem for the second equality which is possible because $q(x,\cdot)$ is continuous for all $x\in W$, and $\vert(s+\Delta^{\nicefrac{1}{4}})f \vert$ is compactly supported as well as bounded, and thus
\begin{align*}
\int\limits_W  \int\limits_{0}^{\infty} \vert \psi(x,y,\rho)\cdot(sf+\Delta^{\nicefrac{1}{4}} f)(y)\vert  \,d\rho \,  dy = \int\limits_W q(x,y) \cdot \vert(s f+\Delta^{\nicefrac{1}{4}} f)(y) \vert \, dy <\infty.
\end{align*}
Further, we used ``Green's formula'' for the third equality (see \cite[Theorem 3.16]{Grigoryan}). Lastly, note that $(s+\Delta^{\nicefrac{1}{4}})\psi(x,\cdot,\rho)\equiv 0$ by the discussion preceding Lemma \ref{lemma:PropertiesOfH}. Thus, by elliptic regularity (see \cite[Corollary 7.3]{Grigoryan}), the function $y\mapsto H^{\nicefrac{1}{4}}(x,y;s)$ is smooth. 

Note that, since $H^{\nicefrac{1}{4}}(x,y;s)=H^{\nicefrac{1}{4}}(y,x;s)$ for all $x,y\in W$ (recall equation \eqref{equation:SymmetrieProduktLegendreQ}), the function $W\ni x\mapsto H^{\nicefrac{1}{4}}(\cdot,y;s)\in\mathbb{R}$ is smooth for all $y\in W$, too.

Furthermore, by \eqref{equation:HSmoothFunction} and Green's formula, it follows that $(s+\Delta^{\nicefrac{1}{4}})H^{\nicefrac{1}{4}}(x,\cdot;s)\equiv 0$ for arbitrary $x\in W$. More precisely, for all $x\in W$ and $f\in C_{c}^{\infty}(W)$:
\begin{align*}
0 = \int\limits_W H^{\nicefrac{1}{4}}(x,y;s) \cdot (sf+\Delta^{\nicefrac{1}{4}} f)(y) \,dy = \int\limits_W f(y) \cdot (sH^{\nicefrac{1}{4}}(x,\cdot;s)+\Delta^{\nicefrac{1}{4}} H^{\nicefrac{1}{4}}(x,\cdot;s))(y) dy.
\end{align*}
Because $f\in C_{c}^{\infty}(W)$ was arbitrary and $H^{\nicefrac{1}{4}}(x,\cdot;s)$ is smooth, it follows for all $x\in W$: $(s+\Delta^{\nicefrac{1}{4}})H^{\nicefrac{1}{4}}(x,\cdot;s)\equiv 0$ (see \cite[Lemma 3.13]{Grigoryan}). Finally, because of the symmetry $H^{\nicefrac{1}{4}}(x,y;s)=H^{\nicefrac{1}{4}}(y,x;s)$ for all $x,y\in W$, we also have $(s+\Delta^{\nicefrac{1}{4}})H^{\nicefrac{1}{4}}(\cdot,y;s)\equiv 0$ for arbitrary $y\in W$.

$(iii)$. Let $s>\frac{1}{4}$ and let $y\in W$ be fixed. We want to show that the function $W\ni x\mapsto H^{\nicefrac{1}{4}}(x,y;s)\in\mathbb{R}$ can be extended continuously to $\overline{W}$. Suppose $x\in\partial W$ is any boundary point other than the vertex of $W$, i.e. in polar coordinates $x=(a,\alpha)$ with $a\in (0,\infty)$ and $\alpha\in \{ 0, \gamma \}$. For those boundary points, we define the value of $H^{\nicefrac{1}{4}}(x,y;s)$ by the formula given on the right-hand side of \eqref{equation:SolutionH}. Note that the integral which is obtained after setting $\alpha=0$ or $\alpha=\gamma$ on the right-hand side of \eqref{equation:SolutionH} is absolutely convergent (see Theorem \ref{theorem:TheoremZentralIntegral}). Further, that continuation indeed provides a continuous function $W\backslash\{P\}\ni x\mapsto H^{\nicefrac{1}{4}}(x,y;s)\in\mathbb{R}$ (recall that $P$ denotes the vertex of the wedge). This can easily be seen by using Lebesgue's dominated convergence theorem and \eqref{equation:EstimateProductLegendreFunctions}. Lastly, note that $H^{\nicefrac{1}{4}}(x,y;s) = G_{\mathbb{H}^2}^{\nicefrac{1}{4}}(x,y;s)$ for all $x\in \partial W\backslash\{P\}$, which is obvious when using \eqref{equation:IdentitaetCosHyperbolicusMonster}.

On the contrary, it is more challenging to show that $H^{\nicefrac{1}{4}}(\cdot,y;s)$ can also be extended continuously at the vertex $P$. Note for example, that the formula on the right-hand side of \eqref{equation:SolutionH} has no meaning for $a=0$ because $Q^{-i\rho}_{\sqrt{s}-\frac{1}{2}}(z)$ is formally not defined for $z=1$ (also the limit of $Q^{-i\rho}_{\sqrt{s}-\frac{1}{2}}(\cosh(a))$ as $a\searrow 0$ does not exist). However, we will show that
\begin{align}
\label{equation:WertSpitzenpunkt}
\lim\limits_{x\rightarrow P} H^{\nicefrac{1}{4}}(x,y;s) = \frac{1}{2\pi} Q_{\sqrt{s}-\frac{1}{2}}(\cosh(b)).
\end{align}
For that purpose, let us write $H^{\nicefrac{1}{4}}(x,y;s) = I_1(x) - I_2(x)$ for all $x\in W$ with
\begin{align*}
I_1(x):&=\frac{1}{\pi^2}\int\limits_{0}^{\infty} Q_{ \sqrt{s} - \frac{1}{2}}^{-i\rho}(\cosh(a))  Q_{ \sqrt{s} - \frac{1}{2}}^{i\rho}(\cosh(b))\cdot \frac{\sinh(\pi\rho)}{\sinh(\gamma\rho)}\cosh(\rho(\gamma-\alpha-\beta)) d\rho,\\
I_2(x):&=\frac{1}{\pi^2}\int\limits_{0}^{\infty} Q_{ \sqrt{s} - \frac{1}{2}}^{-i\rho}(\cosh(a))  Q_{ \sqrt{s} - \frac{1}{2}}^{i\rho}(\cosh(b))\cdot \frac{\sinh((\pi-\gamma)\rho)}{\sinh(\gamma\rho)}\cosh((\alpha-\beta)\rho) d\rho.
\end{align*}
First, we will consider the limit of $I_1(x)$ as $x\rightarrow P$. Note that $I_1$ can be written as
\begin{align*}
I_1(x) = \frac{1}{2}\cdot \frac{2}{\pi} \int\limits_{0}^{\infty} Q_{ \sqrt{s} - \frac{1}{2}}^{-i\rho}(\cosh(a))  Q_{ \sqrt{s} - \frac{1}{2}}^{i\rho}(\cosh(b))\cdot \frac{\sin(i\pi\rho)}{\pi} \cdot \frac{\cos(i\rho(\gamma-\alpha-\beta))}{\sin(i\gamma\rho)} d\rho.
\end{align*}
Since we are only interested in the limit of $I_1(x)$ as $x=(a,\alpha)\rightarrow P$ we may assume in the followig $a<b$. The idea is to apply Theorem \ref{theorem:TheoremZentralIntegral} with $g(z)=\frac{\cos(z(\gamma-\alpha-\beta))}{\sin(z\gamma)}$. The only singularities of $g$ are simple poles at $k\frac{\pi}{\gamma}$ with $k\in\mathbb{Z}$ with residue $\Res(g;k\frac{\pi}{\gamma})=\frac{(-1)^k}{\gamma}\cos\left(\frac{\pi k}{\gamma}(\gamma-\alpha-\beta)\right)$. Hence, by Theorem \ref{theorem:TheoremZentralIntegral},
\begin{align}
I_1(x) =\, &\sum\limits_{k=1}^{\infty} \frac{(-1)^k}{\gamma} \cos\left(\frac{\pi k}{\gamma}(\gamma-\alpha-\beta)\right)e^{-i\pi k \frac{\pi}{\gamma}}P_{ \sqrt{s} - \frac{1}{2}}^{-k\frac{\pi}{\gamma}}(\cosh(a))  Q_{ \sqrt{s} - \frac{1}{2}}^{k\frac{\pi}{\gamma}}(\cosh(b)) \,+ \nonumber\\
&+ \frac{1}{2\gamma} P_{ \sqrt{s} - \frac{1}{2}}(\cosh(a))  Q_{ \sqrt{s} - \frac{1}{2}}(\cosh(b)). \label{equation:TheoremAufI1}
\end{align}
Note that
\begin{align}
\label{equation:BehaviourSingularityOneLegendreP}
\lim\limits_{a\searrow 0}P_{ \sqrt{s} - \frac{1}{2}}(\cosh(a)) = 1,
\end{align}
(see \cite[formula 14.8.7 on p. 361]{NIST}) and thus
\begin{align*}
\lim\limits_{a\searrow 0} \frac{1}{2\gamma} P_{ \sqrt{s} - \frac{1}{2}}(\cosh(a))  Q_{ \sqrt{s} - \frac{1}{2}}(\cosh(b)) = \frac{1}{2\gamma} Q_{ \sqrt{s} - \frac{1}{2}}(\cosh(b)).
\end{align*}
Now, we want to show that the series appearing on the right-hand side of \eqref{equation:TheoremAufI1} converges to zero as $a\searrow 0$, uniformly for all $\alpha\in (0,\gamma)$. By Lemma \ref{lemma:SecondEstimateProductLegendreFunctions} $(i)$, there exist constants $C,D>0$ such that for all $\mu>0$ and $a\in\mathbb{R}$ with $0<a< 2\arsinh\left(\frac{1}{2D}\right) <\frac{1}{D}$ (the last inequality follows because $\arsinh(x)<x$ for all $x>0$):
\begin{align*}
\vert  P_{\sqrt{s}-\frac{1}{2}}^{-\mu}(\cosh(a))Q_{\sqrt{s}-\frac{1}{2}}^{\mu}(\cosh(b)) &\vert \leq Ce^{\sqrt{s}a} \cdot  \mu^{\sqrt{s}} \cdot \left(D\sinh\left(\frac{a}{2}\right)\right)^{\mu}\\
&\leq Ce^{\frac{\sqrt{s}}{D}}\left(\frac{1}{2}\right)^{\frac{\mu}{2}} \cdot  \mu^{\sqrt{s}} \cdot \left(D\sinh\left(\frac{a}{2}\right)\right)^{\frac{\mu}{2}}\\
&\leq \hat{C}\cdot \left(D\sinh\left(\frac{a}{2}\right)\right)^{\frac{\mu}{2}},
\end{align*}
where $\hat{C}:=Ce^{\frac{\sqrt{s}}{D}}\cdot \sup_{\mu>0}\big\{ \left(\frac{1}{2}\right)^{\frac{\mu}{2}}\mu^{\sqrt{s}} \,\big\} \in (0,\infty)$. Thus, for all $\alpha\in (0,\gamma)$ and $a\in\mathbb{R}$ with $0<a< 2  \arsinh\left(\frac{1}{2D}\right)$, the absolute value of the series in \eqref{equation:TheoremAufI1} is bounded from above by
\begin{align*}
&\quad\, \sum\limits_{k=1}^{\infty} \bigg\vert \frac{(-1)^k}{\gamma} \cos\left(\frac{\pi k}{\gamma}(\gamma-\alpha-\beta)\right)e^{-i\pi k \frac{\pi}{\gamma}}P_{ \sqrt{s} - \frac{1}{2}}^{-k\frac{\pi}{\gamma}}(\cosh(a))  Q_{ \sqrt{s} - \frac{1}{2}}^{k\frac{\pi}{\gamma}}(\cosh(b)) \bigg\vert \\
&\leq \, \frac{\hat{C}}{\gamma}\sum\limits_{k=1}^{\infty}\left(\left(D\sinh\left(\frac{a}{2}\right)\right)^{\frac{\pi}{2\gamma}}\right)^{k} = \frac{\hat{C}}{\gamma} \cdot \frac{ \left( D\sinh\left(\frac{a}{2}\right)\right)^{\frac{\pi}{2\gamma}}}{1-\left(D\sinh\left(\frac{a}{2}\right)\right)^{\frac{\pi}{2\gamma}}}.
\end{align*}
Obviously, that upper bound holds uniformly for all $\alpha\in(0,\gamma)$ and converges to $0$ as $a\searrow 0$. We conclude 
\begin{align}
\label{equation:LimitI1}
\lim\limits_{x\rightarrow P}I_1(x)=\frac{1}{2\gamma} Q_{ \sqrt{s} - \frac{1}{2}}(\cosh(b)).
\end{align}

Next, we consider $I_2(x)$ as $x\rightarrow P$. If $\gamma=\pi$ then $I_2(x) = 0$ for all $x\in W$ and thus \eqref{equation:WertSpitzenpunkt} follows. Suppose now $\gamma\neq \pi$. We write $I_2$ as
\begin{align*}
I_2(x) = \frac{1}{\pi}\int\limits_{0}^{\infty}Q_{ \sqrt{s} - \frac{1}{2}}^{-i\rho}(\cosh(a))  Q_{ \sqrt{s} - \frac{1}{2}}^{i\rho}(\cosh(b)) \frac{\sin(i\pi \rho)}{\pi} \cdot \frac{\sin((\pi-\gamma)i\rho)\cos((\alpha-\beta)i\rho)}{\sin(i\pi \rho)\sin(\gamma i\rho)} d\rho.
\end{align*}
This time, we will apply Theorem \ref{theorem:TheoremZentralIntegral} with
\begin{align*}
g(z)=\frac{\sin((\pi-\gamma)z)\cos((\alpha-\beta)z)}{\sin(\pi z)\sin(\gamma z)} = \frac{\cos(\gamma z)\cos((\alpha-\beta)z)}{\sin(\gamma z)}-\frac{\cos(\pi z)\cos((\alpha-\beta)z)}{\sin(\pi z)}.
\end{align*}
The singularities of that function are simple poles and the set of all poles is given by $ \mathbb{Z} \cup \big\{ \, k\frac{\pi}{\gamma} \mid k\in\mathbb{Z} \,\big\}$. Let $(p_{\ell})_{\ell=1}^{\infty}$ be the sequence of all positive poles such that $0<p_1<p_2<p_3<...,$ and thus
\begin{align*}
I_2(x)=&\sum\limits_{\ell=1}^{\infty} \text{Res}\left( g;p_{\ell} \right) e^{-i\pi p_{\ell}}P_{\sqrt{s}-\frac{1}{2}}^{-p_{\ell}}(\cosh(a))Q_{\sqrt{s}-\frac{1}{2}}^{p_{\ell}}(\cosh(b)) \\
&+ \frac{1}{2}\left( \frac{1}{\gamma} - \frac{1}{\pi} \right)P_{\sqrt{s}-\frac{1}{2}}(\cosh(a))Q_{\sqrt{s}-\frac{1}{2}}(\cosh(b)).
\end{align*}
As above, one can show that the series on the right-hand side converges to $0$ as $a\searrow 0$, uniformly for all $\alpha\in (0,\gamma)$. Just note that the set of all residue $\{ \text{Res}\left( g;p_{\ell} \right)  \}_{\ell\in\mathbb{N}}$ is bounded. Hence, we obtain together with \eqref{equation:BehaviourSingularityOneLegendreP}:
\begin{align*}
\lim\limits_{x\rightarrow P}I_2(x) = \frac{1}{2}\left( \frac{1}{\gamma} - \frac{1}{\pi} \right)Q_{\sqrt{s}-\frac{1}{2}}(\cosh(b)).
\end{align*}

In particular, we have shown that
\begin{align*}
\lim\limits_{x\rightarrow P} H^{\nicefrac{1}{4}}(x,y;s) = \lim\limits_{x\rightarrow P} ( I_1(x) - I_2(x) ) = \frac{1}{2\pi}Q_{\sqrt{s}-\frac{1}{2}}(\cosh(b))=:H^{\nicefrac{1}{4}}(P,y;s).
\end{align*}
Note that we also have $\frac{1}{2\pi}Q_{\sqrt{s}-\frac{1}{2}}(\cosh(b))=G_{\mathbb{H}^2}^{\nicefrac{1}{4}}(P,y;s)$, which follows directly from \eqref{equation:GreenPlane1} because of $d(P,y)=b$.


It remains to show $0 < H^{\nicefrac{1}{4}}(x,y;s)$ for all $x\in W$ and $H^{\nicefrac{1}{4}}(x,y;s)\leq G_{\mathbb{H}^2}^{\nicefrac{1}{4}}(x,y;s)$ for all $x\in W\backslash\{y\}$. To prove those estimates, we will use the following properties of $H^{\nicefrac{1}{4}}$ and $G_{\mathbb{H}^2}^{\nicefrac{1}{4}}$, respectively. First, for all sequences $(x_n)_{n\in\mathbb{N}}$ such that $x_n\in W\backslash\{ y \}$ and $\lim_{n\rightarrow \infty}d(x_n,P)=\infty$: $\lim_{n\rightarrow\infty}G_{\mathbb{H}^2}^{\nicefrac{1}{4}}(x_n,y;s)=0$ as well as $\lim_{n\rightarrow\infty}H^{\nicefrac{1}{4}}(x_n,y;s)=0$. The former limit is easy to check since $G_{\mathbb{H}^2}^{\nicefrac{1}{4}}(x_n,y;s)=\frac{1}{2\pi}Q_{\sqrt{s}-\frac{1}{2}}(\cosh(d(x_n,y)))$ by \eqref{equation:GreenPlane1}, $\lim_{n\rightarrow \infty}d(x_n,y)=\infty$ by our assumptions on $(x_n)_{n\in\mathbb{N}}$, and $Q_{\sqrt{s}-\frac{1}{2}}(z_n)$ converges to $0$ for all sequences $(z_n)_{n\in\mathbb{N}}$ with $z_n\in (1,\infty)$ and $\lim_{n\rightarrow\infty} z_n =\infty$ (see \cite[14.8.15]{NIST}). For the other limit note that, for all sequences $x_n=(a_n,\alpha_n)\in W$ as above, we have by \eqref{equation:EstimateProductLegendreFunctions} the estimate $\vert H^{\nicefrac{1}{4}}(x_n,y;s)\vert\leq \frac{\hat{D}}{\sqrt{a_n-1}}$ for some constant $\hat{D}>0$. Since $\lim_{n\rightarrow \infty}d(x_n,P)=\infty$, we have $\lim_{n\rightarrow\infty}a_n = \infty$ and thus $\lim_{n\rightarrow\infty}H^{\nicefrac{1}{4}}(x_n,y;s)=0$. Second, note that the heat kernel $K_{\mathbb{H}^2}$ is strictly positive, which can be seen, for example, from \eqref{equation:HyperHeat2}. Consequently $G_{\mathbb{H}^2}^{\nicefrac{1}{4}}(x,y;s)>0$ for all $x\in W\backslash\{ y\}$.

To prove positivity of $H^{\nicefrac{1}{4}}$, let us assume that there exists some $x_{0}\in W$ with $H^{\nicefrac{1}{4}}(x_{0},y;s) < 0$. We choose some radius $R>0$ such that $d(x_0,P) < R$ and $H^{\nicefrac{1}{4}}(x_{0},y;s) < H^{\nicefrac{1}{4}}(x,y;s)$ for all $x \in W$ with $d(x,P) \geq R$. The latter condition can be achieved due to $\lim_{d(x,P)\rightarrow\infty}H^{\nicefrac{1}{4}}(x,y;s) = 0$ as shown above. Now, with $U:=B_R(P)\cap W$, we have $H^{\nicefrac{1}{4}}(\cdot,y;s)_{\vert U}\in C(\overline{U})\cap C^{\infty}(U)$ (which means, by definition, the function belongs to $C^{\infty}(U)$ and can be extended continuously to the closure $\overline{U}$). Furthermore, by our assumptions on $R$, we have $H^{\nicefrac{1}{4}}(x_0,y;s)< H^{\nicefrac{1}{4}}(x,y;s)$ for all $x\in \partial U$ (recall that $H^{\nicefrac{1}{4}}(x,y;s)=G_{\mathbb{H}^2}^{\nicefrac{1}{4}}(x,y;s)>0$ for all $x\in \partial W$). This is a contradiction to the elliptic minimum principle (see \cite[Corollary 8.16]{Grigoryan} and recall that $H^{\nicefrac{1}{4}}(\cdot,y;s)$ satisfies $(sH^{\nicefrac{1}{4}}(\cdot,y;s)+\Delta^{\nicefrac{1}{4}} H^{\nicefrac{1}{4}}(\cdot,y;s))(x)=0$ for all $x\in U$). Thus, we conclude $H^{\nicefrac{1}{4}}(x,y;s)\geq 0$ for all $x\in W$. Moreover, by the \emph{strong} elliptic minimum principle (see \cite[Corollary 8.14]{Grigoryan}), it follows that $H^{\nicefrac{1}{4}}(x,y;s)> 0$ for all $x\in W$.

The second estimate is shown similarly. Note that $u:W\backslash\{y\}\ni x\mapsto G_{\mathbb{H}^2}^{\nicefrac{1}{4}}(x,y;s)-H^{\nicefrac{1}{4}}(x,y;s) \in \mathbb{R}$ is smooth and can be extended continuously to $\overline{W}\backslash\{y\}$ by $u(x):=0$ for all $x\in\partial W$. Moreover, $\lim_{d(x,P)\rightarrow\infty}u(x)=0$ (as shown above) and $\lim_{x\rightarrow y}u(x)=\infty$. The latter limit holds because $H^{\frac{1}{4}}(\cdot,y;s)$ is a bounded function on $W$ and $\lim_{x\rightarrow y} G_{\mathbb{H}^2}^{\nicefrac{1}{4}}(x,y;s) = \lim_{x\rightarrow y}\frac{1}{2\pi}Q_{\sqrt{s}-\frac{1}{2}}(\cosh(d(x,y)))=\infty$, where the last equality follows from \cite[formula (12.23)]{Olver}. Now, suppose there exists some $x_{0}\in W\backslash\{ y \}$ such that $u(x_0)<0$. We choose radii $R, \tau>0$ so that the following conditions are satisfied: $\overline{B_{\tau}(y)}\subset W$, $u(x)>0$ for all $x\in \overline{B_{\tau}(y)}\backslash\{y \}$, $u(x_0)<u(x)$ for all $x\in W$ with $d(x,P)\geq R$, and $x_0 \in \left(W\cap B_R(P) \right) \backslash \overline{B_{\tau}(y)}=:V$. Obviously, $V$ is a relatively compact open set and $u\in C(\overline{V})\cap C^{\infty}(V)$ satisfies $(s+\Delta^{\nicefrac{1}{4}})u(x) = 0$ for all $x\in V$. Moreover $x_0\in V$ is some interior point with the property $u(x_0)<u(x)$ for all $x\in\partial V$, which is a contradiction to the elliptic minimum principle  as above (see \cite[Corollary 8.16]{Grigoryan}). Thus, we conclude $u \geq 0$, or equivalently, $G_{\mathbb{H}^2}^{\nicefrac{1}{4}}(x,y;s) \geq H^{\nicefrac{1}{4}}(x,y;s)$ for all $x\in W\backslash\{ y\}$.
\end{proof}

\begin{theorem}
\label{theorem:GreenWedge}
For all $(x,y)\in\offdiag(W)$ and $s\in\mathbb{C}$ with $\Re(s)>\frac{1}{4}$:
\begin{align}
\label{equation:GreenWedge}
\begin{split}
G_W^{\nicefrac{1}{4}}(x,y;s)&=\frac{1}{\pi^2}\int\limits_{0}^{\infty} Q_{\sqrt{s}-\frac{1}{2}}^{-i\rho}(\cosh(a))Q_{\sqrt{s}-\frac{1}{2}}^{i\rho}(\cosh(b))\cdot \bigg( \cosh(\rho(\pi-\vert \alpha-\beta \vert))-    \\
& \frac{\sinh(\pi\rho)}{\sinh(\gamma\rho)}\cosh(\rho\left( \gamma-\alpha-\beta \right)) + \frac{\sinh(\rho(\pi-\gamma))}{\sinh(\gamma\rho)}\cosh((\alpha-\beta)\rho)\bigg) d\rho,
\end{split}
\end{align}
where $x=(a,\alpha)$ and $y=(b,\beta)$ \emph{(}with respect to the polar coordinates chosen above\emph{)}.
\end{theorem}

\begin{proof}
For brevity, we denote the right-hand side of \eqref{equation:GreenWedge} by $\tilde{G}_W^{\nicefrac{1}{4}}(x,y;s)$. More precisely,
\begin{align*}
\tilde{G}_W^{\nicefrac{1}{4}}:&\offdiag(W) \times \mathcal{H}_{>\frac{1}{4}} \rightarrow\mathbb{C},\\
 &((x,y),s)\mapsto \tilde{G}_W^{\nicefrac{1}{4}}(x,y;s):=G_{\mathbb{H}^2}^{\nicefrac{1}{4}}(x,y;s) - H^{\nicefrac{1}{4}}(x,y;s),
\end{align*} 
where $H^{\nicefrac{1}{4}}$ is defined as in \eqref{equation:SolutionH}. Recall the following facts: For all $x,y\in W$ with $x\neq y$ the functions $G_{\mathbb{H}^2}^{\nicefrac{1}{4}}(x,y;\cdot)$ and $G_W^{\nicefrac{1}{4}}(x,y;\cdot)$ are holomorphic on $\mathcal{H}_{>\frac{1}{4}}$. Similary, for all $x,y\in W$ the function $H^{\nicefrac{1}{4}}(x,y;\cdot)$ is also holomorphic on $\mathcal{H}_{>\frac{1}{4}}$ which follows from Lemma \ref{lemma:PropertiesOfH} $(i)$. Hence, it suffices to prove the equality $G_W^{\nicefrac{1}{4}}(x,y;s) = \tilde{G}_W^{\nicefrac{1}{4}}(x,y;s)$ for all $s\in(\frac{1}{4},\infty)$ and $x,y\in W$ with $x\neq y$.

Let $s>\frac{1}{4}$ be fixed and let $f\in C_{c}^{\infty}(W)$ be given such that $f(x)\geq 0$ for all $x\in W$. We extend that function to $\mathbb{H}^2$ by $f(y):=0$ for all $y\in\mathbb{H}^2\backslash W$ and we denote the extended function also by $f$, so that it belongs to $C_{c}^{\infty}(\mathbb{H}^2)$. We define
\begin{align*}
&u_W:W\ni x\mapsto \int\limits_{W}G_W^{\nicefrac{1}{4}}(x,y;s)f(y)\, dy\in\mathbb{R} \\
\text{ and }\,\, &u_{\mathbb{H}^2}:\mathbb{H}^2\ni x\mapsto \int\limits_{\mathbb{H}^2}G_{\mathbb{H}^2}^{\nicefrac{1}{4}}(x,y;s)f(y) \, dy\in\mathbb{R},
\end{align*}
which are known to be smooth functions (see \cite[Theorem $8.7 (ii)$]{Grigoryan}). Moreover, one can show that $(s+\Delta^{\nicefrac{1}{4}})u_W(x)=f(x)$ for all $x\in W$, respectively $(s+\Delta^{\nicefrac{1}{4}})u_{\mathbb{H}^2}(x)=f(x)$ for all $x\in \mathbb{H}^2$ (see \cite[Theorem $8.4 (b)$]{Grigoryan}). Similarly, we define
\begin{align*}
\Phi:W\ni x\mapsto \int\limits_{W}H^{\nicefrac{1}{4}}(x,y;s)f(y) \, dy\in\mathbb{R}.
\end{align*}

Note that $\Phi$ is smooth and satisfies $(s+\Delta^{\nicefrac{1}{4}})\Phi(x)=0$ for all $x\in W$, which follows e.g. from Lemma \ref{lemma:PropertiesOfH} $(ii)$, \cite[Satz $5.7$ Zusatz on p. 148]{Elstrodt} and the fact that $f$ is compactly supported in $W$ (by assumption).

Next, we define
\begin{align*}
&\tilde{u}_W:W\ni x\mapsto \int\limits_{W}\tilde{G}_W^{\nicefrac{1}{4}}(x,y;s)f(y)\, dy\in\mathbb{R}
\end{align*}
and we will show $u_W(x)=\tilde{u}_W(x)$ for all $x\in W$ by applying the result in \cite[Exercise 8.2]{Grigoryan} (where the notation given in \cite[Exercise 8.2]{Grigoryan} corresponds to our notation as follows: $\alpha:=s-\frac{1}{4}$, $R_{\alpha}f:=u_W$ and $u:=\tilde{u}_W$). Let us show that the conditions of \cite[Exercise 8.2]{Grigoryan} are indeed satisfied by $\tilde{u}_W$:

First, note that $\tilde{u}_W = u_{\mathbb{H}^2}-\Phi$ and thus, by the discussion above, $\tilde{u}_W$ is smooth and satisfies $(s+\Delta^{\nicefrac{1}{4}})\tilde{u}_W(x)=f(x)$ for all $x\in W$. Moreover, $\tilde{u}_W$ is non-negative, since for all $x\in W$
\begin{align*}
\tilde{u}_W(x)=\int\limits_{W}\tilde{G}_W^{\nicefrac{1}{4}}(x,y;s)f(y) dy = \int\limits_{W}(G_{\mathbb{H}^2}^{\nicefrac{1}{4}}(x,y;s) - H^{\nicefrac{1}{4}}(x,y;s)) \cdot f(y) dy,
\end{align*}
where $f(y)\geq 0$ for all $y\in W$ (by assumption), and $\left(G_{\mathbb{H}^2}^{\nicefrac{1}{4}} - H^{\nicefrac{1}{4}}\right)(x,y;s)\geq 0$ for all $(x,y)\in\offdiag(W)$ by Lemma \ref{lemma:PropertiesOfH} $(iii)$. 
Lastly, we need to check whether $\lim_{n\rightarrow\infty}\tilde{u}_W(x_n)=0$ for all sequences $(x_n)_{n\in\mathbb{N}}\subset W$ such that either there exists some $x_{\ast}\in \partial W$ with $\lim_{n\rightarrow \infty}x_n=x_{\ast}$ or $\lim_{n\rightarrow\infty}d(P,x_n)=\infty$. That property follows immediately from Lebesgue's dominated convergence theorem and the following facts:

Recall that for all $y\in W$ the function $\tilde{G}_W^{\nicefrac{1}{4}}(\cdot,y;s):W\backslash\{ y \} \rightarrow \mathbb{R}$ can be extended continuously to $\overline{W}\backslash\{ y \}$ by $\tilde{G}_W^{\nicefrac{1}{4}}(x,y;s):=0$ for all $x\in\partial W$ (see Lemma \ref{lemma:PropertiesOfH} $(iii)$). Further, we have shown in the proof of Lemma \ref{lemma:PropertiesOfH} $(iii)$ that $\lim_{d(x,P)\rightarrow\infty}\tilde{G}_W^{\nicefrac{1}{4}}(x,y;t)=0$ for all $y\in W$. Lastly, for all $\varepsilon >0$ there exists some constant $D>0$ such that for all $x,y\in W$ with $d(x,y)\geq \varepsilon$:
\begin{align*}
0\leq \tilde{G}_W^{\nicefrac{1}{4}}(x,y;s)\leq G_{\mathbb{H}^2}^{\nicefrac{1}{4}}(x,y;s) \leq D.
\end{align*}
The first estimate follows from Lemma \ref{lemma:PropertiesOfH} $(iii)$ and for the second estimate note that there exists some constant $C>0$ such that
\begin{align}
\label{equation:EstimateHeatKernelPlane}
K_{\mathbb{H}^2} (x,y;t) \leq \frac{C}{t} \cdot e^{-\frac{d(x,y)^2}{8t}},\quad \forall x,y\in\mathbb{H}^2,\, t>0
\end{align}
(see e.g. \cite[Lemma 7.4.26]{Buser}).
Thus, it follows for all $x,y\in \mathbb{H}^2$ with $d(x,y)\geq \epsilon$:
\begin{align*}
G_{\mathbb{H}^2}^{\nicefrac{1}{4}}(x,y;s) &= \int\limits_{0}^{\infty} e^{(\frac{1}{4}-s)t} K_{\mathbb{H}^2} (x,y;t) dt \leq \int\limits_{0}^{\infty} e^{(\frac{1}{4}-s)t} \frac{C}{t} \cdot e^{-\frac{\varepsilon^2}{8t}} dt=:D.
\end{align*}

Thus, by Lebesgue's dominated convergence theorem we have $\lim_{n\rightarrow\infty}\tilde{u}_W(x_n)=0$ for any sequence as above. Using \cite[Exercise 8.2]{Grigoryan} we have $u_W(x)=\tilde{u}_W(x)$ for all $x\in W$. 

Since $f\in C_{c}^{\infty}(W)$ with $f\geq 0$ was arbitrary, we conclude that for all $x\in W$: $G_W^{\nicefrac{1}{4}}(x,y;s)=G_{\mathbb{H}^2}^{\nicefrac{1}{4}}(x,y;s)-H^{\nicefrac{1}{4}}(x,y;s)$ for almost all $y\in W$. Further, for any $x\in W$, both sides are continuous in $y\in W\backslash\{x\}$ and thus they must be equal for all $y\in W\backslash\{x\}$.

\end{proof}

By definition, the (shifted) Green's function is defined as the Laplace transform of the (shifted) heat kernel. Thus we obtain a formula for the (shifted) heat kernel from Theorem \ref{theorem:GreenWedge}.

\begin{corollary}
\label{corollary:FormelHeatKernelWedgeShift}
For all $t>0$ and $x,y \in W$
\begin{align}
\label{equation:FormelHeatKernelWedgeShift}
&K_W^{\nicefrac{1}{4}}(x,y;t) = K_{\mathbb{H}^2}^{\nicefrac{1}{4}}(x,y;t) \, - \mathcal{L}^{-1}\left\lbrace s\mapsto H^{\nicefrac{1}{4}}(x,y;s)\right\rbrace (t),
\end{align}
where $H^{\nicefrac{1}{4}}$ is the same function as in Lemma \emph{\ref{lemma:PropertiesOfH}}.
 
Moreover, for all $x\in W$ the function $(0,\infty)\ni t\mapsto  K_{\mathbb{H}^2}^{\nicefrac{1}{4}}(x,x;t) - K_W^{\nicefrac{1}{4}}(x,x;t) \in\mathbb{R}$ is non-negative, can be extended continuously at $t=0$, its Laplace integral is \emph{(}absolutely\emph{)} convergent and equal to $H^{\nicefrac{1}{4}}(x,x;s)$ for all $s\in\mathcal{H}_{>\frac{1}{4}}$. 
\end{corollary}

\begin{proof}
Let $x\in W$ be arbitrary. By definition, for all $y\in W$ with $y\neq x$ and for all $s\in\mathcal{H}_{>\frac{1}{4}}$:
\begin{align*}
G_{W}^{\nicefrac{1}{4}}(x,y;s)=\mathcal{L}\lbrace K_{W}^{\nicefrac{1}{4}}(x,y;\cdot) \rbrace(s)\, \text{ and }\, G_{\mathbb{H}^2}^{\nicefrac{1}{4}}(x,y;s) = \mathcal{L}\lbrace K_{\mathbb{H}^2}^{\nicefrac{1}{4}}(x,y;\cdot) \rbrace(s).
\end{align*}
Furthermore, by Theorem \ref{theorem:GreenWedge}, we have $H^{\nicefrac{1}{4}}(x,y;s)=G_{\mathbb{H}^2}^{\nicefrac{1}{4}}(x,y;s)-G_{W}^{\nicefrac{1}{4}}(x,y;s)$ for all $y\in W$ with $y\neq x$ and $s\in\mathcal{H}_{>\frac{1}{4}}$. Thus, for all $y\in W$ with $y\neq x$ and $t>0$:
\begin{align*}
\mathcal{L}^{-1}\left\lbrace s\mapsto H^{\nicefrac{1}{4}}(x,y;s)\right\rbrace (t) = K_{\mathbb{H}^2}^{\nicefrac{1}{4}}(x,y;t) - K_W^{\nicefrac{1}{4}}(x,y;t).
\end{align*}
We will show that the above equation is also valid for $y=x$. For that purpose we consider $u:W\times [0,\infty)\rightarrow \mathbb{R}$ defined for all $(y,t)\in W\times [0,\infty)$ as
\begin{align*}
u(y,t):=\begin{cases}
K_{\mathbb{H}^2}^{\nicefrac{1}{4}}(x,y;t)-K_{W}^{\nicefrac{1}{4}}(x,y;t),&\text{ if }t>0\\
0,& \text{ if }t=0.
\end{cases}
\end{align*}
For all $y\in W$ and $t>0$ we have $0\leq K_{W}^{\nicefrac{1}{4}} (x,y;t)\leq K_{\mathbb{H}^2}^{\nicefrac{1}{4}}(x,y;t)$, and thus $0\leq u(y,t)\leq K_{\mathbb{H}^2}^{\nicefrac{1}{4}}(x,y;t)$. Moreover, $u$ is continuous (see \cite[Corollary $9.21$ and Exercise $9.7$]{Grigoryan}). Hence, the Laplace transform of $u(y,\cdot)$ exists for all $y\in W$ and $s\in\mathcal{H}_{>\frac{1}{4}}$. Just note that for all $y\in W$ and $s\in\mathcal{H}_{>\frac{1}{4}}$:
\begin{align*}
\int\limits_{0}^{\infty}e^{-st}u(y,t)dt \leq  \int\limits_{0}^{1}e^{-st}u(y,t)dt+\int\limits_{1}^{\infty}e^{-st}K_{\mathbb{H}^2}^{\nicefrac{1}{4}}(x,y;t)dt,
\end{align*}
where the first integral is convergent because of the continuity of $u(y,\cdot)$ on $[0,\infty)$, and the second integral is convergent due to \eqref{equation:EstimateHeatKernelPlane}.

Let $(y_n)_{n\in\mathbb{N}}\subset W\backslash\{x\}$ be a sequence such that $\lim_{n\rightarrow\infty}y_n = x$. Then for all $s\in\mathcal{H}_{>\frac{1}{4}}$:
\begin{align*}
\mathcal{L}\left\lbrace u(x,\cdot)\right\rbrace  (s) &=\int\limits_{0}^{\infty} e^{-st}(K_{\mathbb{H}^2}^{\nicefrac{1}{4}}(x,x;t)-K_{W}^{\nicefrac{1}{4}}(x,x;t)) dt\\
&=\lim\limits_{n\rightarrow\infty}\int\limits_{0}^{\infty}e^{-st}(K_{\mathbb{H}^2}^{\nicefrac{1}{4}}(x,y_n;t)-K_{W}^{\nicefrac{1}{4}}(x,y_n;t)) dt\\
&=\lim\limits_{n\rightarrow\infty} (G_{\mathbb{H}^2}^{\nicefrac{1}{4}}(x,y_n;s)-G_{W}^{\nicefrac{1}{4}}(x,y_n;s)) =\lim\limits_{n\rightarrow\infty} H^{\nicefrac{1}{4}}(x,y_n;s) \\
&=H^{\nicefrac{1}{4}}(x,x;s).
\end{align*}
Thus $u(x,t)= \mathcal{L}^{-1}\{s\mapsto H^{\nicefrac{1}{4}}(x,x;s)\}(t)$ for all $t>0$. Note that we used continuity of $H^{\nicefrac{1}{4}}$ for the last equation (see Lemma \ref{lemma:PropertiesOfH} $(ii)$), and, for the second equation, we used Lebesgue's dominated convergence theorem which is allowed because of the following: First, for all compact $K\subset W$ with $x\in K$, $u_{\vert K\times [0,1]}$ is bounded by continuity of $u$ and also the function $K\times [1,\infty)\ni (y,t)\mapsto e^{-\frac{1}{4}t}u(y,t)\in\mathbb{R}$ is bounded because of $0\leq e^{-\frac{1}{4}t}u(y,t)\leq K_{\mathbb{H}^2}(x,y;t)$ and \eqref{equation:EstimateHeatKernelPlane}. Thus, $K\times [0,\infty)\ni (y,t)\mapsto e^{-\frac{1}{4}t}u(y,t)\in\mathbb{R}$ is bounded as well. In particular, there exists some constant $C>0$ such that for all $t\geq 0$ and $n\in\mathbb{N}$: $\vert e^{-st} u(x,y_n;t) \vert \leq C\cdot e^{(\frac{1}{4}-\Re(s))t}$. That upper bound is integrable over $[0,\infty)$ since, by assumption, $\Re(s)>\frac{1}{4}$.

\end{proof}

    \bibliographystyle{alpha}
        \bibliography{bibliography}

\begin{thebibliography}{EMOT53}

\bibitem[Bar08]{BarnesLegendre}
E.~W. Barnes.
\newblock On generalized {L}egendre functions.
\newblock {\em Quart. J. Math. Oxford Ser.}, 39:pp. 97--204, 1908.

\bibitem[Bus92]{Buser}
P.~Buser.
\newblock {\em Geometry and spectra of compact Riemann surfaces}, volume 106 of
  {\em Progress in Mathematics}.
\newblock Birkh\"auser, Boston, 1992.

\bibitem[Cha84]{Chavel}
I.~Chavel.
\newblock {\em Eigenvalues in {R}iemannian geometry}, volume 115 of {\em Pure
  and Applied Mathematics}.
\newblock Academic Press Inc., Orlando, 1984.

\bibitem[Els09]{Elstrodt}
J.~Elstrodt.
\newblock {\em Ma{\ss}- und Integrationstheorie}.
\newblock Springer-Verlag, Berlin, 6th corrected edition, 2009.

\bibitem[EMOT53]{Erdelyi}
A.~Erd\'elyi, W.~Magnus, F.~Oberhettinger, and F.~G. Tricomi.
\newblock {\em Higher transcendental functions}, volume~I.
\newblock McGraw-Hill Book Company, Inc., New York, 1953.

\bibitem[G{\"o}t65]{Goetze}
F.~G{\"o}tze.
\newblock Verallgemeinerung einer {I}ntegraltransformation von {M}ehler-{F}ock
  durch den von {K}uipers und {M}eulenbeld eingef\"uhrten {K}ern
  {$P_{k}^{m,n}(z)$}.
\newblock {\em Proc. Knkl. Nederl. Akad. Weten. A. = Indagationes Mathematicae
  (Proceedings)}, 68:pp. 396--404, 1965.

\bibitem[GR07]{Gradshteyn}
I.~S. Gradshteyn and I.~M. Ryzhik.
\newblock {\em Table of integrals, series and products}.
\newblock Elsevier/Academic Press Inc., Amsterdam, seventh edition, 2007.

\bibitem[Gri09]{Grigoryan}
A.~Grigor'yan.
\newblock {\em Heat kernel and analysis on manifolds}, volume~47 of {\em AMS/IP
  Studies in Advanced Mathematics}.
\newblock American Mathematical Society, Providence; International Press,
  Boston, 2009.

\bibitem[Hob65]{Hobson}
E.~W. Hobson.
\newblock {\em The Theory of Spherical and Ellipsoidal Harmonics}.
\newblock Chelsea Publishing Company, New York, 1965.

\bibitem[Leb65]{Lebedev}
N.~N. Lebedev.
\newblock {\em Special functions and their applications}.
\newblock Revised English edition. Translated and Edited by Richard A.
  Silverman. Prentice-Hall, Inc., Englewood Cliffs, N.J., 1965.

\bibitem[Low64]{Lowndes}
J.~S. Lowndes.
\newblock Note on the generalized {M}ehler transform.
\newblock {\em Proc. Camb. Phil. Soc.}, 60:pp. 57--59, 1964.

\bibitem[Luk69]{Luke1}
Y.~L. Luke.
\newblock {\em The special functions and their approximations}, volume~I.
\newblock Academic Press, Inc., New York, 1969.

\bibitem[MS67]{McKean}
H.~P. McKean, Jr. and I.~M. Singer.
\newblock Curvature and the eigenvalues of the {L}aplacian.
\newblock {\em J. Differential Geometry}, 1:pp. 43--69, 1967.

\bibitem[OH61]{Oberhettinger}
F.~Oberhettinger and T.~P. Higgins.
\newblock Tables of {L}ebedev, {M}ehler, and generalized {M}ehler transforms.
\newblock {\em Boeing scientific research laboratories}, Mathematical Note No.
  246:pp. 1--48, 1961.

\bibitem[OLBC10]{NIST}
F.~W.~J. Olver, D.~W. Lozier, R.~F. Boisvert, and C.~W. Clark, editors.
\newblock {\em NIST handbook of mathematical functions}.
\newblock U. S. Department of Commerce, National Institute of Standards and
  Technology; Cambridge University Press, Cambridge, 2010.

\bibitem[Olv97]{Olver}
F.~W.~J. Olver.
\newblock {\em Asymptotics and special functions}.
\newblock A K Peters, Wellesley, Massachusetts, 1997.

\bibitem[Ros74]{Rosenthal}
P.~Rosenthal.
\newblock On an inversion theorem for the general {M}ehler-{F}ock transform
  pair.
\newblock {\em Pacific J. of Math.}, 52(2):pp. 539--545, 1974.

\bibitem[Sne72]{Sneddon}
I.~N. Sneddon.
\newblock {\em The use of integral transforms}.
\newblock McGraw-Hill, New York, 1972.

\bibitem[Sri88]{Srisatkunarajah}
S.~Srisatkunarajah.
\newblock {\em On the asymptotics of the heat equation for polygonal domains}.
\newblock PhD thesis, Heriot-Watt University, 1988.
\newblock pp. 40-44.

\bibitem[TE51]{Tricomi}
F.~G. Tricomi and A.~Erd\'elyi.
\newblock The asymptotic expansion of a ration of gamma functions.
\newblock {\em Pacific J. Math.}, 1:pp. 133--142, 1951.

\bibitem[Tem96]{Temme}
N.~M. Temme.
\newblock {\em Special functions: an introduction to the classical functions of
  mathematical physics}.
\newblock John Wiley \& Sons, Inc., New York, 1996.

\bibitem[Ter85]{Terras}
A.~Terras.
\newblock {\em Harmonic analysis on symmetric spaces and applications},
  volume~I.
\newblock Springer-Verlag, New York, 1985.

\bibitem[U{\c{c}}a17]{ErenDiss}
E.~U{\c{c}}ar.
\newblock {\em Spectral invariants for polygons and orbisurfaces}.
\newblock PhD thesis, Humboldt-Universit\"at zu Berlin, 2017.

\bibitem[vdBS88]{VanDenBerg}
M.~van~den Berg and S.~Srisatkunarajah.
\newblock Heat equation for a region in $\mathbb{R}^2$ with a polygonal
  boundary.
\newblock {\em J. London Math. Soc. (2)}, 37:pp. 119--127, 1988.

\bibitem[VF01]{Virchenko}
N.~Virchenko and I.~Fedotova.
\newblock {\em Generalized associated {L}egendre functions and their
  applications}.
\newblock World Scientific Publishing Co. Pte. Ltd., Singapore, 2001.

\bibitem[Wat18]{WatsonAsymp}
G.~N. Watson.
\newblock Asymptotic expansions of hypergeometric functions.
\newblock {\em Trans. Cambridge Philos. Soc}, 22:pp. 277--308, 1918.

\end{thebibliography}

\Addresses

\end{document}